\documentclass[reqno]{amsart}
\usepackage{bm,amsmath,amsthm,amssymb,mathtools,verbatim,amsfonts,tikz-cd,mathrsfs}
\usepackage{enumitem}
\usepackage{float}

\usepackage[hidelinks,colorlinks=true,linkcolor=blue, citecolor=black,linktocpage=true]{hyperref}
\usepackage{graphicx}
\usepackage[alphabetic]{amsrefs}
\usepackage{caption}
\usepackage{morefloats}
\usepackage[top=1.3in, bottom=1.1in, left=1.3in, right=1.3in,marginpar=1in]{geometry}
\usepackage{subfigure}

\usetikzlibrary{matrix,arrows,decorations.pathmorphing}

\usepackage{chngcntr}
\counterwithin{figure}{section}

\newtheorem{thm}{Theorem}[section]
\newtheorem*{thm*}{Theorem}
\newtheorem{prop}[thm]{Proposition}

\newtheorem{cor}[thm]{Corollary}
\newtheorem{lem}[thm]{Lemma}

\theoremstyle{definition}
\newtheorem{define}[thm]{Definition}

\theoremstyle{remark}
\newtheorem{rem}[thm]{Remark}
\newtheorem{example}[thm]{Example}

\newtheorem{question}[thm]{Question}

\newcommand{\ve}[1]{\boldsymbol{\mathbf{#1}}}
\newcommand{\R}{\mathbb{R}}
\newcommand{\Z}{\mathbb{Z}}
\newcommand{\F}{\mathbb{F}}
\newcommand{\N}{\mathbb{N}}
\newcommand{\Q}{\mathbb{Q}}
\newcommand{\C}{\mathbb{C}}

\newcommand{\iso}{\cong}

\newcommand{\scU}{\mathscr{U}}
\newcommand{\scV}{\mathscr{V}}

\renewcommand{\a}{\alpha}
\renewcommand{\b}{\beta}
\newcommand{\g}{\gamma}

\newcommand{\veps}{\varepsilon}

\renewcommand{\d}{\partial}
\renewcommand{\subset}{\subseteq}

\renewcommand{\tilde}{\widetilde}
\renewcommand{\bar}{\overline}

\DeclareMathOperator{\Char}{{Char}}

\DeclareMathOperator{\gr}{{gr}}

\DeclareMathOperator{\id}{{id}}
\newcommand{\Inc}{Q}

\DeclareMathOperator{\im}{{im}}

\DeclareMathOperator{\lk}{{lk}}

\DeclareMathOperator{\Spin}{{Spin}}

\DeclareMathOperator{\Span}{{Span}}
\DeclareMathOperator{\Sym}{{Sym}}

\DeclareMathOperator{\coker}{{coker}}

\newcommand{\bA}{\mathbb{A}}
\newcommand{\bB}{\mathbb{B}}
\newcommand{\bC}{\mathbb{C}}
\newcommand{\bD}{\mathbb{D}}

\newcommand{\bF}{\mathbb{F}}

\newcommand{\bH}{\mathbb{H}}

\newcommand{\bP}{\mathbb{P}}

\newcommand{\bT}{\mathbb{T}}
\newcommand{\bU}{\mathbb{U}}
\newcommand{\bX}{\mathbb{X}}
\newcommand{\bk}{\mathbf{k}}
\newcommand{\bx}{\mathbf{x}}

\newcommand{\cA}{\mathcal{A}}

\newcommand{\cC}{\mathcal{C}}

\newcommand{\cF}{\mathcal{F}}
\newcommand{\cG}{\mathcal{G}}

\newcommand{\cK}{\mathcal{K}}

\newcommand{\cM}{\mathcal{M}}

\newcommand{\cR}{\mathcal{R}}
\newcommand{\scR}{\mathcal{R}}

\newcommand{\cU}{\mathcal{U}}
\newcommand{\cV}{\mathcal{V}}
\newcommand{\cW}{\mathcal{W}}
\newcommand{\cX}{\mathcal{X}}

\newcommand{\cZ}{\mathcal{Z}}

\newcommand{\frA}{\mathfrak{A}}

\newcommand{\frs}{\mathfrak{s}}
\newcommand{\frt}{\mathfrak{t}}

\newcommand{\frz}{\mathfrak{z}}

\newcommand{\scF}{\mathscr{F}}

\newcommand{\scK}{\mathscr{K}}
\newcommand{\scP}{\mathscr{P}}
\newcommand{\scQ}{\mathscr{Q}}

\newcommand{\xs}{\ve{x}}
\newcommand{\ys}{\ve{y}}
\newcommand{\zs}{\ve{z}}

\newcommand{\cCFL}{\mathcal{C\!F\!L}}
\newcommand{\cCFK}{\mathcal{C\hspace{-.5mm}F\hspace{-.3mm}K}}
\newcommand{\cHFL}{\mathcal{H\!F\! L}}
\newcommand{\CF}{\mathit{CF}}

\newcommand{\HF}{\mathit{HF}}

\newcommand{\HFK}{\mathit{HFK}}

\newcommand{\CFL}{\mathit{CFL}}
\newcommand{\HFL}{\mathit{HFL}}

\newcommand{\PD}{\mathit{PD}}

\usepackage{leftidx}

\DeclareMathOperator{\Cone}{{Cone}}

\newcommand{\wt}[1]{\widetilde{#1}}

\numberwithin{equation}{section}

\newcommand{\ar}{\mathrm{a.r.}}

\allowdisplaybreaks

\def\XG{X_G}
\def\YG{Y_G}

\title{Lattice homology, formality, and plumbed L-space links}
\author{Maciej Borodzik}
\address{Institute of Mathematics of Polish Academy of Sciences, ul. \'Sniadeckich 8,
00-656 Warsaw, Poland}
\email{mcboro@mimuw.edu.pl}

\author{Beibei Liu}
\address{Department of Mathematics\\ The Ohio State University \\ Columbus, OH, USA}
\email{liu.11302@osu.edu}

\author{Ian Zemke}
\address{Department of Mathematics\\Princeton University\\  Princeton, NJ, USA}
\email{izemke@math.princeton.edu}

\thanks{MB was supported by OPUS 2019/B/35/ST1/01120 grant
  of the Polish National Science Center. BL is partially supported by NSF grant DMS-2203237.  IZ was partially supported by NSF grants DMS-1703685 and DMS-2204375.}

\begin{document}

\begin{abstract}
  We define a link lattice complex for plumbed links, generalizing constructions
  of Ozsv\'ath, Stipsicz and Szab\'o, and of Gorsky and N\'emethi. We prove that for all plumbed links in rational homology 3-spheres, the link lattice complex is homotopy equivalent to the  link Floer complex as an $A_\infty$-module. Additionally, we prove that the  link Floer complex of a plumbed L-space link is a free resolution of its homology. As a consequence, we give an algorithm to compute the link Floer complexes of plumbed L-space links, in particular of algebraic links, from their multivariable Alexander polynomial.
\end{abstract}
\maketitle

\tableofcontents

\section{Introduction}
\subsection{Overview}

Lattice homology is a combinatorial invariant of plumbed 3-manifolds defined by N\'emethi \cite{NemethiAR,NemethiLattice}, see also \cite{Nemethi_opus_magnum}*{Chapter 11}. N\'{e}methi's construction is a formalization of earlier work of Ozsv\'{a}th and Szab\'{o} \cite{OSPlumbed}, which computes the Heegaard Floer homology groups of many plumbed 3-manifolds. If $Y$ is a plumbed 3-manifold, we write $\bH \bF(Y)$ for its lattice homology, which is a module over the power series ring $\bF[[U]]$.

Work of N\'{e}methi, Ozsv\'{a}th, Stipsicz and Szab\'{o} \cite{NemethiAR,NemethiLattice,OSPlumbed,OSSLattice} proves that
\[
\mathbb{HF}(Y)\iso \ve{\HF}^-(Y)
\]
for many important families of plumbed 3-manifolds. More recently,  the third author proved the isomorphism in general \cite{ZemHFLattice}.

Given a knot $K$ in $S^3$, Ozsv\'{a}th--Szab\'{o} \cite{OSKnots}, and Rasmussen  \cite{RasmussenKnots} defined a refinement of Heegaard Floer homology, called knot Floer homology. There are several equivalent formulations of this invariant. For our purposes, it is most convenient to consider knot Floer homology as taking the form of a free, finitely generated chain complex $\cCFK(K)$ over the 2-variable polynomial ring $\bF[\scU,\scV]$.

Ozsv\'{a}th and Szab\'{o} also defined a version of Heegaard Floer theory for links in 3-manifolds \cite{OSLinks}. For a link $L\subset S^3$, we will focus on the description of the link Floer complex as a finitely generated free chain complex $\cCFL(L)$ over the polynomial ring $\bF[\scU_1,\scV_1,\dots, \scU_{\ell},\scV_\ell]$, where $\ell=|L|$. For our purposes, it is also helpful to consider the knot and link Floer complexes over the completed ring $\bF[[\scU_1,\scV_1,\dots, \scU_\ell,\scV_{\ell}]]$, denoted
\[
\ve{\cCFL}(L):=\cCFL(L)\otimes_{\bF[\scU_1,\scV_1,\dots, \scU_\ell,\scV_\ell]} \bF[[\scU_1,\scV_1,\dots, \scU_\ell,\scV_\ell]].
\]

Their construction also applies more generally when $L$ is a link in a rational homology 3-sphere $Y$. In this case, we denote the link Floer complex $\cCFL(Y,L)$.

 A relative version of lattice homology for plumbed knots is defined by Ozsv\'{a}th, Stipsicz and Szab\'{o}  \cite{OSSKnotLatticeHomology}. Modulo notational differences, their version of knot lattice homology is analogous to the complex $\cCFK(K)$. They proved that for a plumbed L-space knot in $S^3$, knot lattice homology coincides with the knot Floer complex $\cCFK(K)$.

 Gorsky and N\'emethi \cite{GorskyNemethiLattice} defined a relative version of lattice homology for L-space links. (Note, their construction does not require the link to be plumbed). They constructed a spectral sequence from their version of link lattice homology to a version of link Floer homology, which is  the homology of  the quotient complex $\cCFL(L)/(\scV_1,\dots, \scV_\ell)$, 
and proved that it degenerates for all algebraic links. 
In particular, their version of link lattice homology is isomorphic (as a graded group) to the version of link Floer homology for algebraic links in $S^3$.

In our paper, we construct a new version of link lattice homology. Our link lattice complex is more closely related to the construction of Ozsv\'{a}th, Stipsicz and Szab\'{o} \cite{OSSKnotLatticeHomology}, and is modeled on the full link Floer complex $\cCFL(L)$. We use this link lattice complex to study plumbed L-space links, a family which includes all algebraic links in $S^3$.

\subsection{The  link lattice complex}\label{sub:link_lattice_complex}

 Suppose that $L$ is a plumbed link in a plumbed 3-manifold $Y$. Such a pair $(Y,L)$ is presented by a weighted graph $\Gamma$, whose vertices are partitioned into two sets
 \[
 V_\Gamma=V_G\cup V_\uparrow.
 \] 
 The vertices $V_G$ are equipped with integral weights. The vertices in $V_\uparrow$ have no weights, and we refer to them as \emph{arrow} vertices. Unless specified explicitly otherwise, we will always assume that $\Gamma$ is a tree.

   From the tree $\Gamma$, we obtain a partitioned link $L_\Gamma=L_G\cup L_{\uparrow}$ in $S^3$. This link may be described as a connected sum of Hopf links, with one unknotted component for each vertex of $\Gamma$, and one clasp for each edge of $\Gamma$. The manifold $Y$ is the result of surgery on $L_G$, with framing $\Lambda$ obtained from the weights. Inside of $Y\iso S^3_{\Lambda}(L_G)$, the link $L$ is identified with $L_{\uparrow}$.
 
 Given a plumbing tree $\Gamma$ presenting a plumbed link $(Y,L)$, we will construct a chain complex
 \[
 \mathbb{CFL}(\Gamma,V_\uparrow).
 \]
 Given an orientation of $L_\uparrow$,  we will equip the chain complex $\mathbb{CFL}(\Gamma,V_\uparrow)$  with the structure of a module over the ring $\bF[[\scU_1,\scV_1,\dots, \scU_{\ell},\scV_{\ell}]]$, where $\ell=|V_{\uparrow}|$, as well as with a Maslov grading and a $\Q^\ell$-valued Alexander grading.

  It is helpful to view $\mathbb{CFL}(\Gamma,V_\uparrow)$ as an $A_\infty$-module over $\bF[[\scU_1,\scV_1,\dots, \scU_{\ell},\scV_{\ell}]]$ with only $m_1$ and $m_2$ non-vanishing. We note that the complex $\mathbb{CFL}(\Gamma,V_\uparrow)$ is not free over this ring unless $|V_\uparrow|=1$. 

A central result of the paper is the following:
\begin{thm}
\label{thm:equivalence-intro}
Suppose that $\Gamma$ is a plumbing link diagram which is a tree, and write $(Y,L)$ for the associated 3-manifold and link. If $Y$ is a rational homology sphere, then $\ve{\cCFL}(Y,L)$ is homotopy equivalent to $\mathbb{CFL}(\Gamma,V_\uparrow)$ as an absolutely graded  $A_\infty$-module over $\bF[[\scU_1,\scV_1,\dots, \scU_\ell,\scV_{\ell}]]$.
\end{thm}

See Theorem~\ref{thm:lattice=HFL} for further details.
In the above, we are writing $\ve{\cCFL}(Y,L)$ for the full link Floer complex $\cCFL(Y,L)$ completed over the power series ring $\bF[[\scU_1,\scV_1,\dots, \scU_\ell,\scV_\ell]]$.

\subsection{Algebraic and plumbed L-space links}

 We recall that a rational homology 3-sphere $Y$ is called an \emph{L-space} if 
\[
\widehat{\HF}(Y,\frs)\iso \bF
\]
for each $\frs\in \Spin^c(Y)$.  A link $L\subset S^3$ is called an \emph{L-space link} if all sufficiently positive surgeries are L-spaces. We note that since Dehn surgery does not depend on the orientation of the link, the property of being an L-space link is independent of orientations.

An important family of plumbed links in $S^3$ are algebraic links, which are the links of complex plane curve singularities. According to  Gorsky and N\'{e}methi \cite{GorskyNemethiAlgebraicLinks}, algebraic links in $S^3$ are L-space links.

There is a useful characterization of L-space links in terms of the link Floer complex. If $L\subset S^3$, then $L$ is an L-space link if and only if the homology group $\cHFL(L)$ is torsion free as an $\bF[U]$-module, where $U$ acts by $\scU_i\scV_i$ for any $i$. (Since $\scU_i\scV_i$ and $\scU_j\scV_j$ have chain homotopic actions for all $i$ and $j$, the definition is independent of the choice of $i$).

Ozsv\'{a}th and Szab\'{o} \cite{OSlens} proved the knot Floer complex of an L-space knot is a staircase complex. A very natural question is whether an analog of Ozsv\'{a}th and Szab\'{o}'s result holds for L-space links. L-space links have also been extensively studied in the literature, see e.g. \cite{BorodzikGorskyImmersed, CavalloLiu, GorskyHom, GLMoore, GorskyNemethiLattice, LiuSurgeries,LiuLspace, LiuB}.  Despite the interest in L-space links and many interesting results concerning them, there is to date no result which characterizes the structure of the link Floer complex of an L-space link in parallel with Ozsv\'{a}th and Szab\'{o}'s result for L-space knots.

We prove the following:
\begin{thm}\label{thm:free resolution-intro}
 Suppose that $L\subset S^3$ is a plumbed L-space link. Then the link Floer complex $\cCFL(L)$ is homotopy equivalent to a free resolution of its homology over $\bF[\scU_1,\scV_1,\dots, \scU_\ell,\scV_{\ell}]$. Equivalently, the complex $\cCFL(L)$ is homotopy equivalent to its homology $\cHFL(L)$ as an $A_\infty$-module over $\bF[\scU_1,\scV_1,\dots, \scU_\ell,\scV_{\ell}]$, where we equip $\cHFL(L)$ with the $A_\infty$-module structure which has $m_j=0$ unless $j=2$. 
\end{thm}

\begin{figure}[ht]
$T(3,4)$\begin{tikzcd}[row sep =.3cm, column sep=2cm,labels=description]
&\bullet\\
\bullet 
	\ar[ur, "\scV"]
	\ar[dr, "\scU^2"]
&\\
&\bullet\\
\bullet
	\ar[ur, "\scV^2"]
	\ar[dr, "\scU"]
\\
&\bullet
\end{tikzcd}
\qquad \qquad
\begin{tikzcd}[labels=description,row sep=1.5cm, column sep=3cm]
\bullet \ar[r, "\scU_2"] \ar[dr, "\scV_1", pos=.2]& \bullet\\
\bullet \ar[ur, "\scU_1", pos=.2,crossing over] \ar[r, "\scV_2"]& \bullet
\end{tikzcd} $T(2,2)$
\\
$T(2,4)$\begin{tikzcd}[row sep =.4cm, column sep=2cm,labels=description]
&&\bullet
\\
&\bullet
	\ar[ur, "\scU_2"]
	\ar[ddr, "\scV_1"]
\\
&\bullet
	\ar[uur,crossing over, "\scU_1"]
	\ar[dr, "\scV_2"]
\\
\bullet
	\ar[ruu, "\scU_2"]
	\ar[ru, "\scU_1"]
	\ar[dr, "\scV_2"]
	\ar[ddr, "\scV_1"]
&&\bullet
\\
&\bullet
	\ar[ur, "\scU_2"]
	\ar[ddr, "\scV_1"]
\\
&\bullet
	\ar[uur,crossing over, "\scU_1"]
	\ar[dr, "\scV_2"]
\\
&&\bullet
\end{tikzcd}
\caption{The knot and link Floer complexes of $T(3,4)$, $T(2,2)$ and $T(2,4)$. Each dot denotes a generator in a free basis. The horizontal direction indicates the grading of the free resolution.}
\label{fig:free-resolutions}
\end{figure}
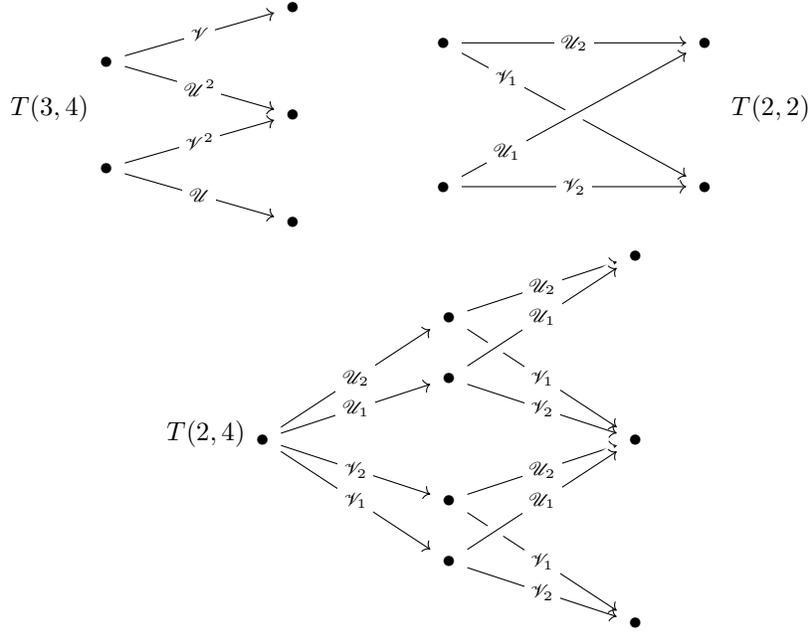

\begin{rem}
\begin{enumerate}
\item Theorem~\ref{thm:free resolution-intro} is a natural generalization of the result of Ozsv\'{a}th and Szab\'{o} for L-space knots because the staircase complex of an L-space knot is easily seen to be a free resolution of its homology over $\bF[\scU,\scV]$.
\item Theorem~\ref{thm:free resolution-intro} also holds more generally for plumbed L-space links $L$ in plumbed 3-manifolds $Y$ which are themselves L-spaces. For details, see Section \ref{sec:plumbed_l_space}. 
\item A $dg$-module $M$ over a ring $A$ such that $(H_*(M),m_2)$ is homotopy equivalent to $M$ as an $A_\infty$-module over $A$ is called \emph{formal}. Hence, we may restate Theorem~\ref{thm:free resolution-intro} by saying that plumbed L-space links have formal link Floer complexes.
\end{enumerate}
\end{rem}
We do not know whether Theorem~\ref{thm:free resolution-intro} holds for non-plumbed L-space links. We state the following open question:
\begin{question} Are there non-plumbed L-space links $L\subset S^3$ for which $\cCFL(L)$ is not homotopy equivalent to a free-resolution of $\cHFL(L)$?
\end{question}

From Theorem~\ref{thm:free resolution-intro}, we obtain an algorithm to compute the link Floer complex of a plumbed L-space link. Namely, we observe that for an L-space link $L$, the $\bF[\scU_1,\scV_\ell,\dots,\scU_\ell, \scV_{\ell}]$-module $\cHFL(L)$ is completely determined by the $H$-function (or equivalently the $d$-invariants of large surgeries). According to \cite{GorskyNemethiLattice}, the $H$-function of an L-space link in $S^3$ is determined by the multivariable Alexander polynomials  of $L$ and its sublinks. After determining $\cHFL(L)$, one may compute $\cCFL(L)$ by computing its free resolution. Finding such a resolution
is algorithmic, see e.g. \cite{PeevaBook}, and may be done using computer algebra software such as Macaulay2 \cite{M2}. We carry this out for $T(3,3)$ and $T(4,4)$ in Section~\ref{sec:examples}.

Our algorithm to compute the full link Floer complexes enables us to compute the minus and the hat version of link Floer homology of plumbed L-space links by setting  respectively $\scV_i=0$ and $\scU_i=\scV_i=0$ for all $i$. Moreover, the hat version of link Floer homology detects the Thurston norm of the link complement in the three sphere \cite{ozsvath2008link}. Hence, the algorithm also gives a way to compute the Thurston norm of plumbed L-space links in $S^3$.  

%

\begin{cor}
The link Floer complex of a plumbed L-space link $L$ in $S^3$ is computable from the multivariable Alexander polynomials of  the link $L$ and its sublinks.
\end{cor}

\begin{rem} Computing free resolutions is in general a challenging task. For many purposes (e.g. taking tensor products), it is more practical to understand the homology group $\cHFL(L)$ and use the fact that $\cCFL(L)$ is homotopy equivalent as an $A_\infty$-module to $\cHFL(L)$ with only $m_2$ non-trivial. Often the homology group $\cHFL(L)$ has a much more concise description than its free resolution. For example in Section~\ref{sec:examples}, we give a simple description of the module $\cHFL(T(n,n))$. We describe free resolutions for $T(3,3)$ and $T(4,4)$, which are considerably larger in complexity.
\end{rem}


We note that for algebraic links in $S^3$ (which are particular examples of plumbed L-space links) the fact that the link Floer complex is determined by the Alexander polynomial may be seen indirectly using the fact that the Alexander polynomial of an algebraic link determines the link.
 See work of Yamamoto \cite{Yamamoto}. This is not the case for
algebraic links in other $3$-manifolds, or other plumbed L-space links in $S^3$; see \cite{CDG} and also Proposition~\ref{prop:non_isotopic}. Nonetheless, our techniques give a concrete algorithm for computing the link Floer complex based on its Alexander polynomial. Although foundational, Yamamoto's work does not give a practical algorithm for computing link Floer complexes of algebraic links in $S^3$.

\subsection{Gorsky and N\'{e}methi's link lattice homology}

We also consider Gorsky and N\'{e}methi's link lattice homology, which is defined for all L-space links. If $L$ is an L-space link, then they described a chain complex $\scK(L)$, which they called the link lattice complex. They proved that if $L$ is an algebraic link, then
\begin{equation}
H_*(\scK(L))\iso H_*(\cCFL(L)/(\scV_1,\dots, \scV_\ell)),\label{eq:GN-isomorphism-intro}
\end{equation}
as graded groups, where the right-hand side is identified with the minus version of link Floer homology.

 In our paper, we give an alternate perspective on the complex $\scK(L)$. Namely, we show that $\scK(L)$ is homotopy equivalent as an $A_\infty$-module to the derived tensor product
\[
 \scK(L)\simeq \cHFL(L) \mathrel{\tilde{\otimes}_{\bF[\scU_1,\scV_1,\dots, \scU_\ell,\scV_{\ell}]}} \bF[\scU_1,\scV_1,\dots, \scU_{\ell},\scV_{\ell}]/(\scV_1,\dots, \scV_{\ell})
\]
where we equip $\cHFL(L)$ with the $A_\infty$-module structure which has only $m_2$ non-trivial. As a corollary of Theorem~\ref{thm:free resolution-intro}, we  obtain the following result:

\begin{cor}\label{cor:GN-lattice} If $L$ is a plumbed L-space link, then Gorsky and N\'{e}methi's link lattice complex $\scK(L)$ is homotopy equivalent to $\cCFL(L)/(\scV_1,\dots, \scV_\ell)$ as a graded chain complex. 
\end{cor}

We note Gorsky and N\'{e}methi only prove the isomorphism in Equation~\eqref{eq:GN-isomorphism-intro} at the level of graded groups. Our proof of Corollary~\ref{cor:GN-lattice} improves on their result additionally because it equips $\scK(L)$ with 
a $dg$-module structure over $\bF[\scU_1,\scV_1,\dots, \scU_\ell, \scV_\ell]$ (i.e. an $A_\infty$-module structure with only $m_1$ and $m_2$ non-trivial) and proves the isomorphism at the level of $A_\infty$-modules. 


Gorsky and N\'{e}methi also  constructed a spectral sequence from $\cHFL(L)\otimes \Lambda^*_{\cR_\ell}(\xi_1,\dots, \xi_\ell)$ to the homology of $\cCFL(L)/(\scV_1,\dots, \scV_n)$. Here, $\Lambda^*_{\cR_\ell}(\xi_1,\dots, \xi_\ell)$ is the exterior algebra on $\ell$-generators and $\cR_\ell=\bF[\scU_1,\scV_1,\dots, \scU_\ell, \scV_{\ell}]$. Our Corollary~\ref{cor:GN-lattice} naturally recovers this spectral sequence, as follows. We first recall that the module $\bF[\scU_1,\scV_1,\dots, \scU_\ell, \scV_\ell]/(\scV_1,\dots, \scV_\ell)$ has free resolution over $\cR_\ell$ equal to the Koszul complex $\Lambda^*_{\cR_\ell}(\xi_1,\dots, \xi_\ell)$, where we equip each $\xi_i$ with differential $d(\xi_i)=\scV_i$ (extended via the Leibniz rule). The chain complex $(\Lambda^*_{\cR_\ell}(\xi_1,\dots, \xi_\ell), d)$ has a cube filtration, and the spectral sequence associated to the induced cube filtration on $\cHFL(L)\otimes \Lambda^*_{\cR_\ell}(\xi_1,\dots, \xi_\ell)$ coincides with Gorsky and N\'{e}methi's spectral sequence.

\subsection{Structure of the paper}
Section~\ref{sec:algebraic_background} describes background on $A_\infty$-modules and the homological perturbation lemma. We give
a basic example from knot Floer theory of two complexes with isomorphic homology groups, such that one of them is chain homotopy equivalent
to its homology (regarded as a chain complex with trivial differentials), while the other one is not.

Next, we recall some topological background in Section~\ref{sec:topological}. We provide the definitions of plumbed manifolds and plumbed links.
We state the result of Gorsky and N\'emethi that a plumbed link in a rational analytic singularity is an L-space link
\cite{GorskyNemethiAlgebraicLinks}.
In particular, we show that plumbed L-space links are a natural generalization of algebraic links in $S^3$.
In Section~\ref{sec:lattice}, we define our link lattice complex. We describe the gradings and the $\bF[\scU_1,\scV_1,\dots, \scU_\ell,\scV_\ell]$-module structure.

Section~\ref{sec:equivalence} proves the equivalence of the link lattice complex and the link Floer complex, stated above in Theorem~\ref{thm:equivalence-intro}. This is the main technical result of the paper.

In Section~\ref{sec:plumbed_l_space} we focus on plumbed L-space links.  Using Theorem~\ref{thm:lattice=HFL} and the homological perturbation
lemma, we prove Theorem~\ref{thm:free resolution-intro}, 
which states that the link Floer complex of a plumbed L-space link is a free resolution of its homology.  As a consequence, we show that the link Floer complex of a plumbed L-space link is computable from its multivariable Alexander polynomial. See Theorem~\ref{thm:alexchain_rational}. Additionally, we describe how our link lattice
complex recovers the theory described by Gorsky and N\'emethi in the case of plumbed L-space links. See Theorem~\ref{thm:rai_GN}.

Section~\ref{sec:examples} describes some algorithms and examples. 
We provide a concrete way of presenting the $\bF[\scU_1,\scV_1,\dots,\scU_n,\scV_n]$-module $\cHFL(L)$ for an L-space link $L$,
from its $H$-function. The algorithm of Lemma~\ref{lemma:generating} provides a presentation of $\cHFL(T(n,n))$ 
compatible with the description of Gorsky and Hom \cite{GorskyHom}*{Section 5}.
Next, we compute $\cCFL(T(3,3))$ and $\cCFL(T(4,4))$ by explicitly constructing free resolutions. 

\subsection*{Acknowledgments.} The authors would like to thank Eugene Gorsky, Chuck Livingston, Andr\'as N\'emethi and Lorenzo Traldi 
for stimulating discussions. We are grateful
to Marco Marengon for the help with the complex of the $T(3,3)$ torus link. We have benefited a lot from Karol Palka's explanations of a current state-of-art on computing resolutions of modules.

\section{Algebraic background}\label{sec:algebraic_background}

In this section we recall the notion of an $A_\infty$-module (Subsection~\ref{sub:ainfty_modules}). Then, in Subsection~\ref{sub:hpl},
we state the homological perturbation lemma in $A_\infty$-category. In Subsection~\ref{sub:free}, we consider the homological perturbation lemma in the context of free-resolutions of modules. In Subsection~\ref{sub:modules}, we present two complexes over $\bF[\scU,\scV]$
with isomorphic homology, but which are not chain homotopy equivalent. These two complexes appear naturally in
knot Floer homology.

\subsection{$A_\infty$-modules}\label{sub:ainfty_modules}

Throughout the paper, we make use of the category of $A_\infty$-modules. The motivation is that $A_\infty$-module structures may be transferred along homotopy equivalences of groups. Suppose that $\cA$ is a ring that is an algebra over a field $\ve{k}$. Given a finitely generated chain complex $C$ over $\cA$, in very general circumstances, one may pick a homotopy equivalence over $\ve{k}$ between $C$ and $H_*(C)$. Given such a homotopy equivalence, the homotopy transfer lemma (cf. \cite{Kadeishvili_Ainfinity}) equips $H_*(C)$ with the structure of an $A_\infty$-module over $\cA$, such that $C$ and $H_*(C)$ are homotopy equivalent as $A_\infty$-modules. Note that unless $\cA$ is a field, it is rarely the case that $C$ and $H_*(C)$ are homotopy equivalent as $dg$-modules over $\cA$. When the homotopy equivalence between $C$ and $H_*(C)$ is suitably simple, the $A_\infty$-module maps on $H_*(C)$ are usually computed using a version of the homological perturbation lemma, stated in Lemma~\ref{lem:homological-perturbation-modules}. 

We now recall the basics of $A_\infty$-modules. We mostly
follow the notation of Lipshitz, Ozsv\'{a}th and Thurston \cite{LOTBimodules, LOTBordered}.

Let $\cA$ be an associative algebra with unit over a ring
$\bk$. We will assume that $\bk=\F_2$. We write $\mu_2$ for the multiplication on $\cA$.

\begin{define}
  A left $A_\infty$-module ${}_{\cA}M$ over $\cA$ is a left $\bk$-module
  equipped with $\bk$-module maps
  \[ m_{j+1}\colon\cA^{\otimes j}\otimes_{\bk}M \to M, \quad j\ge 0\]
  such that $m_1\circ m_1=0$, and   for each $n$ and any $a_1,\dots,a_n\in\cA$, $\bx\in M$,
  the following holds.
  \[
  \begin{split}
    &\sum_{i=0}^n m_{n-i+1}(a_n,a_{n-1},\dots,a_{i+1},
    m_{i+1}(a_i,\dots,a_1,\bx))\\
    +&\sum_{k=1}^{n-1}
    m_n(a_n,a_{n-1},\dots,\mu_2(a_{k+1}, a_k),\dots,a_1,\bx)=0.
    \end{split}
\]
\end{define}
Lipshitz, Ozsv\'{a}th and Thurston refer to $A_\infty$-modules as \emph{type-$A$} modules,
in contrast to \emph{type-D} modules, which we now introduce.
\begin{define}
  A right \emph{type-D module} $N^{\cA}$ over $\cA$ is a right $\bk$-module $N$,
  together with a $\bk$-linear structure map
  \[\delta^1\colon N\to N\otimes_{\bk}\cA,\]
  such that
  \[(\id_N\otimes \mu_2)\circ(\delta^1\otimes\id_{\cA})\circ\delta^1=0.\]
\end{define}

\subsection{The homological perturbation lemma}\label{sub:hpl}

It is a general fact that $A_\infty$-algebra structures may be transferred along homotopy equivalences of the chain complex underlying an $A_\infty$-algebra. This was proved by Kadeishvili \cite{Kadeishvili_Ainfinity}. Homological perturbation theory gives concrete formulas for the resulting $A_\infty$-module structure under certain restrictions on the chain homotopy equivalence.  See \cite{KontsevichSoibelman}*{Theorem~3}. An exposition of the technique may be found in Ph.D. thesis of Lef\`evre--Hasegawa \cite{Kenji}.

\begin{lem}\label{lem:homological-perturbation-modules}
Suppose that $\cA$ is an associative algebra over a ground ring $\ve{k}$,
${}_{\cA} M=(M,m_j)$ is an $A_\infty$-module over $\cA$,  $(Z,\d)$ is a chain complex over $\ve{k}$, and that we have three maps of left $\ve{k}$-modules
\[
i\colon Z\to M,\quad \pi\colon M\to Z\quad \text{and} \quad h\colon M\to M
\]
satisfying the following:
\begin{enumerate}
\item $i$ and $\pi$ are chain maps.
\item $\pi\circ i=\id_Z$.
\item $i\circ \pi=\id_{M}+\d(h),$ where $\d(h):=m_1\circ h+h\circ m_1$.
\item $h\circ i=0$.
\item $\pi\circ h=0$.
\item $h\circ h=0$.
\end{enumerate}
 Then there are $A_\infty$-module structure maps $m_j^Z$ on $Z$, as well as $A_\infty$-module morphisms 
 \[
 I_*\colon {}_{A} Z\to {}_{A} M, \quad \Pi_*\colon {}_{A} M\to {}_{A} Z\quad \text{and} \quad H_*\colon {}_A M\to {}_A M
 \]
 satisfying $m_1^Z=\d$, $I_1=i$, $\Pi_1=\pi$ and $H_1=h$, and such that the analogs of relations (1)--(6) are also satisfied as by the $A_\infty$-module morphisms $I_*$, $\Pi_*$ and $H_*$.
\end{lem}

\begin{rem}
  It is important to note that the maps $i,\pi,h$
  in the assumption of homological perturbation lemma are
  required to be only $\bk$-module maps, not 
  necessarily $\cA$-module maps. We refer the interested reader to \cite[Section 1.4]{Kenji}
  for a detailed proof.
\end{rem}

The  extended $A_{\infty}$-module maps in the homological perturbation lemma have a concrete description below. The structure maps on $Z$ are given by the diagrams shown in Figure~\ref{fig:homological-perturbation}. Therein $m_{>1}$ denote the $A_\infty$ structure maps of $M$, with $m_1$ excluded, and $\Delta$ is the comultiplication on the tensor algebra $T^{\ast}\cA$.

\begin{figure}[ht]
\[
m^Z(\ve{a},\zs)=\begin{tikzcd}[column sep=.1cm,row sep=.4cm]
\ve{a}\ar[d, Rightarrow]& \ve{z} \ar[d]\\
 \Delta\ar[dr,bend right=10, Rightarrow]\ar[dd,Rightarrow]&i\ar[d]\\
\,& m_{>1} \ar[d]\\
\Delta\ar[dr,bend right=10, Rightarrow]\ar[dd,Rightarrow]&h\ar[d]\\
\,&m_{>1}\ar[d]\\
\vdots \ar[ddr,bend right=20,Rightarrow] & \vdots \ar[d]\\
\,&h\ar[d]\\
\,&m_{>1}\ar[d]\\
\,&\pi\ar[d] \\
\, &\, 
\end{tikzcd}
\quad 
I(\ve{a}, \zs)=\begin{tikzcd}[column sep=.1cm,row sep=.4cm]
\ve{a}\ar[d, Rightarrow]& \ve{z} \ar[d]\\
 \Delta\ar[dr,bend right=10, Rightarrow]\ar[dd,Rightarrow]&i\ar[d]\\
\,& m_{>1} \ar[d]\\
\Delta\ar[dr,bend right=10, Rightarrow]\ar[dd,Rightarrow]&h\ar[d]\\
\,&m_{>1}\ar[d]\\
\vdots \ar[ddr,bend right=20,Rightarrow] & \vdots\ar[d]\\
\,&h\ar[d]\\
\,&m_{>1}\ar[d]\\
\,&h\ar[d] \\
\, &\, 
\end{tikzcd}
\quad
\Pi(\ve{a}, \xs)=\begin{tikzcd}[column sep=.1cm,row sep=.4cm]
\ve{a}\ar[d, Rightarrow]& \ve{x} \ar[d]\\
 \Delta\ar[dr,bend right=10, Rightarrow]\ar[dd,Rightarrow]&h\ar[d]\\
\,& m_{>1} \ar[d]\\
\Delta\ar[dr,bend right=10, Rightarrow]\ar[dd,Rightarrow]&h\ar[d]\\
\,&m_{>1}\ar[d]\\
\vdots \ar[ddr,bend right=20,Rightarrow] & \vdots\ar[d]\\
\,&h\ar[d]\\
\,&m_{>1}\ar[d]\\
\,&\pi\ar[d] \\
\, &\, 
\end{tikzcd}
\quad H(\ve{a},\xs)=
\begin{tikzcd}[column sep=.1cm,row sep=.4cm]
\ve{a}\ar[d, Rightarrow]& \ve{x} \ar[d]\\
 \Delta\ar[dr,bend right=10, Rightarrow]\ar[dd,Rightarrow]&h\ar[d]\\
\,& m_{>1} \ar[d]\\
\Delta\ar[dr,bend right=10, Rightarrow]\ar[dd,Rightarrow]&h\ar[d]\\
\,&m_{>1}\ar[d]\\
\vdots \ar[ddr,bend right=20,Rightarrow] & \vdots\ar[d]\\
\,&h\ar[d]\\
\,&m_{>1}\ar[d]\\
\,&h\ar[d] \\
\, &\, 
\end{tikzcd}
\]
\caption{The maps appearing in the homological perturbation lemma for $A_\infty$-modules. The notation is introduced in \cite{LOTBordered}*{Section 2}. Shortly, a single arrow represents an element of ${}_\cA M$, while a double arrow represents an element of $\bigoplus_{i}\cA^{\otimes i}$.}
\label{fig:homological-perturbation}
\end{figure}
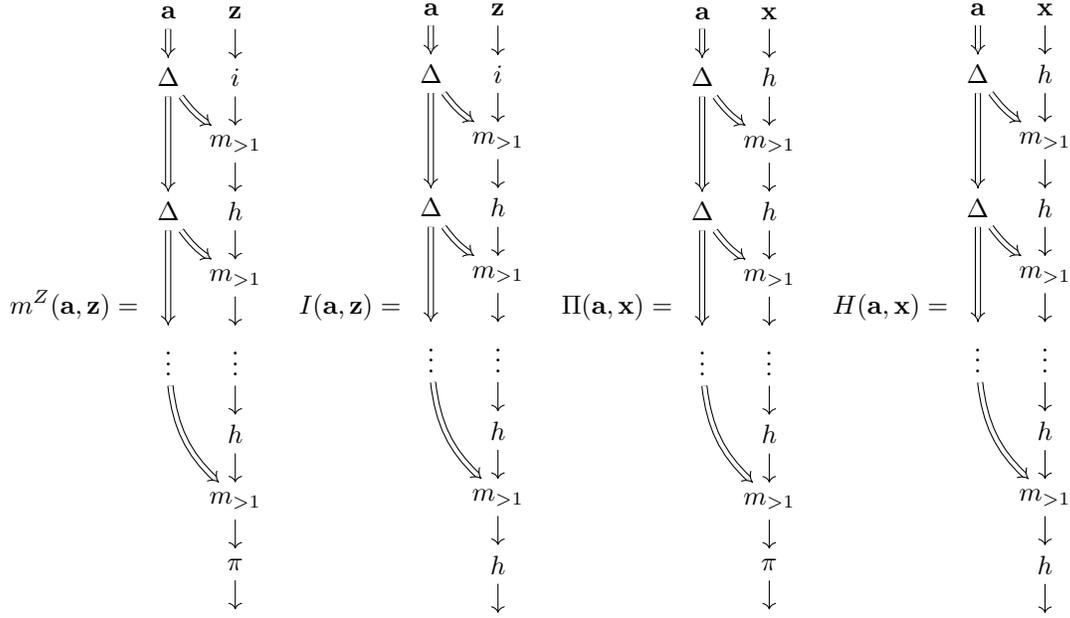

\subsection{Free resolutions and $A_\infty$-actions}\label{sub:free}

In this section, we describe a useful relation between free resolutions and $A_\infty$-module structures. We assume that $\cA$ is an algebra over $\bF=\Z/2$. Suppose that $({}_\cA M,m_j)$ is a type-$A$ module which has $m_j=0$ for $j\neq 2$. That is, $M$ is a left $\cA$-module in the ordinary sense. A \emph{free resolution} of $M$ is a collection of free $\cA$-modules $(F_i,f_{i})_{i\in \N}$ and $\cA$-linear maps, which form an exact sequence of the following form:
\[
\begin{tikzcd}
\cdots \ar[r]&F_i\ar[r, "f_i"]& F_{i-1} \ar[r, "f_{i-1}"]& \cdots \ar[r, "f_2"]& F_1 \ar[r, "f_1"]& F_0\ar[r, "f_0"] & M\ar[r]& 0. 
\end{tikzcd}
\]
For such a free resolution, write $\scF$ for the chain complex which is the direct sum of the $F_i$.

By definition, a free resolution ${}_{\cA} \scF$ is quasi-isomorphic to ${}_{\cA} M$, since the canonical projection map from $\scF$ to $M$ is a chain map which induces an isomorphism on homology. In the category of $A_\infty$-modules, quasi-isomorphisms are always invertible as $A_\infty$ morphisms. An exposition of this principle may be found in \cite{KellerNotes}*{Section~4}. In our present case, it is also helpful to construct explicitly the homotopy equivalence, since we will use it and similar homotopy equivalences later.

\begin{prop}\label{prop:used_for_free_resolution}
Suppose that ${}_\cA \scF$ is a free resolution of a left $\cA$-module ${}_\cA M$. Then ${}_\cA \scF$ and ${}_\cA M$ are homotopy equivalent as $A_\infty$-modules over $\cA$. 
\end{prop}
\begin{proof} There is a canonical projection map $\pi\colon \scF\to M$, which is just $f_0$ on $F_0$, and $0$ on every other summand. We pick any map of $\bF$-vector spaces $i\colon M\to F_0$ such that $f_0\circ i=\id_M$.  The map $i$ induces a direct sum splitting (of $\bF$-vector spaces) $F_0=\im f_1\oplus \im i$. The map $f_1\colon F_1\to \im f_1$ is surjective, so we pick a section of $\bF$-vector space maps, which we denote by $h_0$. The map $h_0$ induces a splitting of $F_1$ into $\im h_0\oplus \im f_2$. We define a map $h_1\colon \im f_2\to F_2$ to be a splitting of the map $f_2$. We proceed in this manner to split the entire free resolution to obtain a diagram
\[
\begin{tikzcd}[column sep=.2cm]
\cdots
	\ar[r]
&\im f_{i+1}\oplus \im h_{i-1}
	\ar[r, bend left, "f_i"]
& \im f_{i}\oplus \im h_{i-2}
	 \ar[l, bend left, "h_{i-1}"]
	 \ar[r, bend left=40, pos=.45, "f_{i-1}"]
&\cdots
	\ar[l, bend left=40, pos=.55, "h_{i-2}"] 
	\ar[r, bend left=40, pos=.55,"f_2"]
&\im f_2\oplus \im h_0
	\ar[l, bend left=40, pos=.45, "h_{1}"]
	\ar[r, bend left, "f_1"]
&
\im f_1\oplus \im i
	 \ar[l, bend left, "h_0"]
	 \ar[r, bend left=40, pos=.48,"f_0"] 
	 &
M
	\ar[r]
	\ar[l, bend left=40, pos=.53,"i"]& 0. 
\end{tikzcd}
\]

Clearly the maps $i$ and $\pi$ are chain maps. Furthermore, $\pi\circ i=\id_M$, and
\[
i\circ \pi=\id_{\scF}+[\d, h]
\]
where $h$ is the direct sum of the $h_i$. Additionally,
\[
h\circ i=0\quad h\circ h=0\quad \text{and}\quad \pi\circ h=0.
\]
In particular, the homological perturbation lemma induces an $A_\infty$-module structure on $M$, which we denote by ${}_\cA \tilde{M}$, such that ${}_\cA \tilde{M}$ and ${}_{\cA} \scF$ are homotopy equivalent as $A_\infty$-modules. 
We claim that the higher actions on 
${}_\cA \tilde{M}$, vanish.
Indeed, the map $h$ always maps $F_i$ to $F_{i+1}$, while $m_2^{\scF}$ preserves the index $F_i$, and $\pi$ is only non-vanishing on $F_0$.
A quick inspection of the left-most map in Figure~\ref{fig:homological-perturbation} shows that $m_j=0$ for $j>2$.

From the claim it follows that ${}_\cA \tilde{M}={}_\cA M$, completing the proof.
\end{proof}

The following is a helpful restatement of the above result:
\begin{cor}\label{cor:free resolution=simplest-action} Suppose $\cA$ is an algebra over $\bF=\Z/2$. Let ${}_\cA M$ be an $A_\infty$-module over $\cA$, and let $\scF$ be the total complex of a free resolution of $H_*(M)$. Then ${}_\cA M$ is homotopy equivalent to ${}_{\cA}\scF$ as an $A_\infty$-module if and only if ${}_\cA M$ is homotopy equivalent as an $A_\infty$-module to $H_*(M)$, equipped with vanishing $m_1$ and vanishing $m_j$ for $j>2$. 
\end{cor}

\subsection{Example of non-formal chain complexes}\label{sub:modules}
As we mentioned in Subsection~\ref{sub:ainfty_modules} 
if $\cA$ is a field, then any finitely generated chain complex is homotopy equivalent to its homology. If $\cA$ is a PID, then it is not hard to show that any finitely generated free chain complex is quasi-isomorphic to its homology. For general rings, this is not always the case.
In this section, we illustrate the case $\cA=\bF[\scU,\scV]$ with examples from the theory of knot Floer homology.

We consider the two free $\bF[\scU,\scV]$-complexes $C$ and $D$ with generators $\ve{u},\ve{v},\ve{x},\ve{y},\ve{z}$,
respectively $\ve{a},\ve{b}$, and~$\ve{c}$, and with differential represented by arrows as shown below:
\[
C=\ve{u}\oplus \begin{tikzcd}[labels=description] \ve{w}\ar[r, "\scU"] \ar[d, "\scV"]& \ve{x} \ar[d,"\scV"]\\
\ve{y}\ar[r, "\scU"]&\ve{z}
\end{tikzcd}
\qquad \text{and} \qquad D=\begin{tikzcd}[labels=description] \,& \ve{a}\ar[d, "\scV"]\\
\ve{c}\ar[r, "\scU"]& \ve{b}
\end{tikzcd}
\]
For appropriate choices of gradings, there is an isomorphism $H_*(C)\iso H_*(D)$, since both are isomorphic to $\bF[\scU,\scV]\oplus \bF$, where $\bF$ has vanishing action of $\scU$ and $\scV$. On the other hand, it is easy to see that $C$ and $D$ are not homotopy equivalent over $\bF[\scU,\scV]$. For example, if we tensor both with the module $\bF[\scU,\scV]/(\scU,\scV)$ and take homology, we obtain vector spaces of different rank over $\bF$.

We now equip both $H_*(C)$ and $H_*(D)$ with $A_\infty$-actions by applying the homological perturbation lemma. Since $C$ is a free resolution of its homology, the induced $A_\infty$-action has only $m_2$ non-trivial.

 For $D$, the homology is the $\bF$ span of $\scU^i\scV^j(\scU\ve{a}+\scV \ve{c})$ and $\ve{b}$, where $i,j\in \N$. We may define a homotopy equivalence of chain complexes over $\bF$ with $D$ and the complex $\bF[\scU,\scV]\oplus \bF$ (with vanishing differential). Write $\ve{e}$ for the generator of $\bF[\scU,\scV]$ and write $\ve{f}$ for the generator of $\bF$.  The inclusion map $i\colon H_*(D)\to D$ is given by
\[
i(\scU^n\scV^m\ve{e})=\scU^n\scV^m(\scU\ve{a}+\scV \ve{c}),  \quad i(\ve{f})=\ve{b},
\]
for all $n,m\ge 0$.
  We define a projection map $\pi\colon D\to H_*(D)$ by setting
\[
    \pi(\scU^n\scV^m \ve{a})=\begin{cases}\scU^{n-1}\scV^m \ve{e}& \text{ if } n>0\\
 0& \text{ otherwise},
 \end{cases}\qquad \text{and}  \qquad
 \pi(\scU^n \scV^m\ve{b})=\begin{cases}
\ve{f}& \text{ if } n=m=0\\
0&\text{ otherwise}.
\end{cases}
\]
The map $\pi$ vanishes on multiples of $\ve{c}$. The maps $\pi$ and $i$ are clearly chain maps, and $\pi\circ i=\id$. We define a homotopy $h\colon D\to D$, by setting $h(\ve{a})=h(\ve{c})=0$, and 
 \[
 h(\scU^n\scV^m\ve{b})=\begin{cases} \scU^{n-1}\scV^m \ve{c}& \text{ if } n>0\\
 \scV^{m-1}\ve{a}& \text{ if } n=0, m>0\\
0& \text{ otherwise}.
 \end{cases}
 \]
 It is straightforward to see that $i\circ \pi=\id+[\d, h]$. Furthermore, $h\circ i$, $h\circ h$ and $\pi\circ h$ vanish. In particular, the maps $i$, $\pi$ and $h$ induce an $A_\infty$-module structure maps $m_j^{H(D)}$ on $H_*(D)$.

 We claim that $m_3^{H(D)}(\scU,\scV,\ve{f})=\ve{e}$. To this end, we use the formula on the left of Figure~\ref{fig:homological-perturbation}. In the present situation, $\ve{a}=\scU\otimes\scV$, $\ve{z}=\ve{f}$. Note that $\Delta(\scU\otimes \scV)=(\scU)\otimes (\scV)$, so we compute
\[
m_3^{H(D)}(\scU,\scV,\ve{f})=\pi( m_2^D(\scU,h(m_2^D(\scV,i (\ve{f})))))=\pi(m_2^D(\scU, \ve{a}))=\ve{e}. 
\]

 \begin{rem} In Heegaard Floer theory, it is common to also consider the ring $\bF[\scU,\scV]/\scU\scV$ (see, e.g., \cite{DHSTmore}).
 A similar computation as above shows that $H_*(C/\scU\scV)\iso H_*(D/\scU\scV)$ as $\bF[\scU,\scV]/\scU\scV$-modules, but that the complexes $C/\scU\scV$ and $D/\scU\scV$ are not homotopy equivalent.
\end{rem}

\section{Plumbed manifolds and plumbed links}\label{sec:topological}

The goal of this section is to recall notions like plumbed manifolds, resolution graphs, and plumbed links. We pay particular attention to algebraic links, which are L-space links by work of Gorsky and N\'{e}methi \cite{GorskyNemethiAlgebraicLinks}. Additionally, in Proposition~\ref{prop:GN_algebraic} we describe a slightly wider class of plumbed links which are also L-space links.

%

\subsection{Review of plumbed manifolds}\label{sub:plumbing_calculus}
To set up the notation, we recall the constructions of 3-manifolds via plumbing. We refer the reader
to \cite{NeumannCalculus} for a detailed exposition.

Suppose $G$ is a finite graph. We let $V_G$ be the set of its vertices. We assume that each $v\in V_G$ has an associated 
weight $\lambda_v\in\Z$. 
From $V_G$ we construct a real four-manifold $\XG$, as follows. For each $v\in V_G$, we take $T_v$, the oriented disk bundle over $S^2$ with Euler number $\lambda_v$. 
The manifold
$\XG$ is obtained by taking a disjoint union  of all the $T_v$ and gluing them using the following recipe.
Whenever two vertices $v,v'\in V_G$ are connected by an edge~$e$, we trivialize
the bundles $T_v$ and $T_{v'}$ over chosen disks in the base. Then, we glue these bundles together by an orientation-presenting
diffeomorphism that swaps the base and the fiber. Refer to \cite{GompfStipsicz}*{Example 4.6.2} or \cite{NeumannCalculus}
for more details.

By convention, if $G$ is not connected, we take a boundary connected sum of manifolds ${\XG}_i$ corresponding to connected components $G_i$
of $G$.
\begin{define}
The manifold $\XG$ is called the \emph{plumbed $4$-manifold} associated with $G$. The boundary $\YG=\partial \XG$ is the 
\emph{plumbed $3$-manifold associated with $G$.}
\end{define}

\smallskip
The construction of a plumbed manifold can be done in a relative setting, providing a pair consisting of a three-manifold and a link
contained in it. The starting data is a graph $\Gamma$ with vertices partitioned into two sets $V_G\sqcup V_\uparrow$. We call $V_\uparrow$ the \emph{arrow} vertices, and we call $V_G$ the \emph{non-arrow} vertices. We do not add weights to $V_{\uparrow}$.
\begin{rem}

From a topological perspective, it is most natural to require each vertex of $V_\uparrow$ to have valence 1. However, in the combinatorial
  construction of link lattice homology, we do not need to make this assumption.
\end{rem}

We write $G\subset \Gamma$ for the full subgraph spanned by the non-arrow vertices. 
The vertices $V_\uparrow$ determine a link $L_\uparrow$ in $\YG$ as follows. Suppose $v\in V_{\uparrow}$ is adjacent
to a non-arrow vertex $w\in V_G$. We let $L_v$ be a circle fiber of the $S^1$-bundle $\partial T_{w}\to S^2$,
such that the projection of $L_v$ onto $S^2$ is disjoint from all the disks used to plumb the disk bundles of other non-arrow vertices. 
If more than one arrow vertex is adjacent to the same non-arrow vertex $w$, we require each of the corresponding components of $L_\uparrow$ to be fibers of $\d T_w\to S^2$ over distinct points.

We define $L_{\uparrow}$ as the union of the circle fibers $L_v$ ranging over $v\in V_\uparrow$.

\begin{define} The link $L_{\uparrow}\subset Y_G$ is called the \emph{plumbed link associated with $\Gamma$}. We say that a link $L\subset Y$ is a \emph{plumbed link} if there exists an arrow-decorated plumbing graph $\Gamma$ and a diffeomorphism
$(Y,L)\iso (\YG,L_\uparrow)$.
\end{define}

\begin{rem}\label{rem:not_a_tree}
One can also consider more general plumbings of disk bundles over higher genus surfaces. To do so, one considers a plumbing graph where each vertex  $v\in V_G$ is assigned an additional weight, corresponding to the genus of the base space of disk bundle.
However, the resulting manifold $Y_G$ is not a rational homology sphere if at least one surface has positive genus. In the present paper, we are mostly concerned with rational homology spheres, so we restrict the discussion to the case where all surfaces are spheres.
 We refer to \cite{NeumannCalculus} for more details.
\end{rem}

\begin{define}\label{def:incidence_matrix}
Suppose $G$ is a plumbing tree with no arrow vertices. We define the \emph{incidence matrix} $\Inc_G$, as follows. The diagonal entries are the
weights associated to vertices, while the off-diagonal terms are $1$ or zero, depending on whether the two vertices
are connected by an edge.
\end{define}

By construction, $\Inc_G$ represents the intersection form on $\XG$. As the intersection form on $\XG$ determines the homology
of $\YG$, we have:
\begin{lem}
  There is an isomorphism $H_1(\YG;\Z)\cong \coker \Inc_G$. In particular, $\YG$ is a rational homology sphere if and only if $\det \Inc_G\neq 0$.
\end{lem}

There is a well known description of $(\YG,L_\uparrow)$ in terms of Dehn surgery. See \cite{GompfStipsicz}*{Example~4.6.2}.  
We form a partitioned link $L_\Gamma=L_G\sqcup L_\uparrow$ in $S^3$, as follows. For each vertex of $\Gamma$, we add an unknotted component to $L_\Gamma$, and for each edge, we add a clasp between the corresponding components. The link $L_\Gamma$ may alternatively be described as an iterated connected sum of Hopf links. The weights on the vertices in $V_G$ determine an integral framing $\Lambda$ on $L_G$, as in Definition~\ref{def:incidence_matrix}. Then
\[
(Y_G,L_\uparrow)\iso (S^3_{\Lambda}(L_G),L_\uparrow).
\]

We remark a slight abuse of notation: $L_\uparrow$ denotes both the link in $S^3$ (as a part of $L_\Gamma\subset S^3$) and its
image in the plumbed manifold $\YG$.

\subsection{Plumbed manifolds and resolutions of analytic singularities}

One of the main motivations for introducing plumbed manifolds comes from resolutions of singularities. We give now a short account on plumbed
manifolds obtained from surface singularities. We refer the reader to introductory lectures of N\'emethi \cite{NemethiFive},
or to \cite[Section 3.3]{Nemethi_opus_magnum}, \cite{LooijengaBook}, \cite[Chapter 4]{NemethiSzilard} for more details and references.

First we focus on the absolute case corresponding to graphs with no arrow vertices. 
Later on we discuss the relative case of embedded resolutions, leading to plumbed links. 

Suppose $(X,x_0)$ is (a germ of) a normal complex analytic surface. 
The word `normal' refers to the property of the local ring $\mathcal{O}_{x_0}(X)$ being integrally closed, see \cite{Hartshorne}*{Exercise I.3.7}. It implies, among other things, that $x_0$ is an isolated singular point. See \cite{Laufer,NemethiFive} for more details.
The surface $(X,x_0)$ can be analytically embedded into $(\C^N,0)$ for $N$ sufficiently
large. Let $B_\varepsilon$ be a ball in $\C^N$ with center at $0$ and radius $\varepsilon>0$. 
It is known, see \cite{Milnor_singular}, that the diffeomorphism type of the intersection $L_X:=X\cap\partial B_\varepsilon$
is independent of $\varepsilon$ and of the embedding of $X$ into $\mathbb{C}^N$, provided $\varepsilon>0$ is small enough.   
Moreover, the pair $(B_\varepsilon,X\cap B_\varepsilon)$ 
is topologically a cone over $(\partial B_\varepsilon,L_X)$. The space $L_X$ is a smooth real 3-dimensional manifold.  We call it the \emph{link of the surface singularity} $(X,x_0)$. 

We stress that we study local behavior of $X$ near $x_0$. From the perspective of algebraic geometry, this is emphasized by saying
that $X$ is a germ of a surface. The reader unfamiliar with this notion, might assume that we replace $X$ by $X\cap B_\varepsilon$,
where $B_\varepsilon$ is as above.

The manifold $L_X$ admits another description.
We let $(\wt{X},E)$ be a resolution of $(X,x_0)$, that is, a smooth complex analytic
surface together with a map $\pi\colon(\wt{X},E)\to (X,x_0)$, which is one-to-one except on $\pi^{-1}(x_0)=E$. Now $E=\sum E_i$
is a union of smooth complex curves (Riemann surfaces) intersecting transversally. Each of the $E_i$ is assigned a number $\lambda_i$
which is its self-intersection. The curves $E_i$ are referred to as the \emph{exceptional components} of
the map $\pi$.

With the resolution we can assign two objects. One is the \emph{dual graph} $G_X$
of the resolution.  
Its vertices correspond to divisors $E_i$. Each
vertex is assigned a weight $\lambda_i$. There is an additional weight of a vertex by the genus of $E_i$ (in the present paper we will consider only
the case where each of the $E_i$ is a sphere, compare Remark~\ref{rem:not_a_tree}). We add  $|E_i\cap E_j|$ edges between vertices $v_i$ and $v_j$, when $i\neq j$. We add no self-edges. 
 There is a matrix $\Inc_G$ associated with $G_X$ as in Definition~\ref{def:incidence_matrix}. We have the following result.
\begin{prop}\
  \begin{itemize}
    \item[(a)] The link $L_X$ of singularity $(X,x_0)$ is diffeomorphic to $Y_{G_X}$.
    \item[(b)] $L_X$ is a rational homology sphere if each of the $E_i$ is a sphere and $G_X$ is a tree.
  \end{itemize}
\end{prop}

The manifold $L_X$ determines the graph $G_X$ up to a precisely described equivalence relation; see \cite{NeumannCalculus}.
The way the resolution is constructed implies that $\Inc_G$ is negative definite. A deep theorem of Grauert \cite{Grauert}
shows that this characterizes links of analytic
singularities among all plumbed links.

\begin{thm}\label{thm:grauert}
  If $\Inc_G$ is negative definite,
  then $\YG$ is a link of an analytic singularity. 
\end{thm}
We stress that the statement is far from true if the word `analytic' is replaced by `algebraic'. Moreover,
in general, there is no uniqueness. While $G_X$ determines the diffeomorphism type of $L_X$, there might be analytically different 
singularities with the same link $L_X$. There is a vast research area concerning which invariants of $X$ depend on the link $L_X$,
and which depend on the analytic structure. We refer the reader to the book \cite{Nemethi_opus_magnum}. 
Examples of invariants depending on the analytic structure of $X$ include
geometric genus $p_g$ (see \cite{Nemethi_opus_magnum}*{Section 6.8}), the embedding dimension (minimal $N$ for which $X$ embeds into $\C^N$,
see \cite{Nemethi_opus_magnum}*{Example 6.7.17}), and the Hilbert-Samuel function \cite{Nemethi_opus_magnum}*{Section 5.1.40}). 

\subsection{Embedded resolutions}

We now consider embedded singularities, which are pairs of analytic spaces $Z\subset X$, with a point $x_0\in Z$, where $X$ and $Z$ are possibly singular at $x_0$.
We restrict our attention to the case where $\dim_{\C}X=2$ and $\dim_{\C} Z=1$. For an introduction to embedded
singularities, we refer to \cite{Nemethi_opus_magnum}, especially Section 2.2. An overview
of singularity theory is given in \cite{NemethiSzilard}*{Section 4.3}. Graph links, and their connection
to singularity theory, are described in \cite{EisenbudNeumannGraphLinks}.

\begin{define}
  An \emph{embedded singularity} is a triple $(X,Z,x_0)$, where $(X,x_0)$ is a (germ of a) normal complex analytic surface and $Z\subset X$
  is a complex analytic curve passing through $x_0$.
\end{define}
\begin{example}
  If $X=\C^2$, an embedded singularity is precisely a plane curve singularity.
\end{example}
Embed $X$ analytically in $\C^N$ with $x_0$ mapped to $0$. Take a small ball $B_\varepsilon$ around $x_0$ in $\C^N$ as above.
For sufficiently small $\varepsilon>0$, the triple $(B_\varepsilon,X\cap B_\varepsilon,Z\cap B_\varepsilon)$
is topologically a cone over $(\partial B_\varepsilon, L_X,L_Z)$, where $L_X$ and $L_Z$ are, respectively,
intersections of $X$ and $Z$ with $\partial B_\varepsilon$. 
The diffeomorphism type of the pair $(L_X,L_Z)$ 
depends on neither the choice of~$\varepsilon$ nor the choice of embedding.
\begin{define}
The pair $(L_X,L_Z)$
is called the \emph{link of the embedded singularity}.
\end{define}
\begin{example}
  Suppose $(X,x_0)=(\C^2,0)$, and $Z$ is a plane algebraic curve passing through $0$. Then, $L_X=S^3$, and the link $L_Z$ is precisely the algebraic link in the ordinary sense.
\end{example} 
Since the study of singularities is local, we consider only the germ of the singularity. We note that, by definition, $(X,Z)$ and $(X\cap B_{\veps}, Z\cap B_{\veps})$ have the same germ. In particular, we may and will assume that $(X,Z)$ is a topologically a cone over $(L_X,L_Z)$.

We can recover $(L_X,L_Z)$ from an embedded resolution. By an \emph{embedded resolution} of $(X,Z,x_0)$ we mean
the triple $(\wt{X},\wt{Z},E)$ together with a proper analytic
map $\pi\colon(\wt{X},\wt{Z},E)\to(X,Z,x_0)$ with the following conditions
\begin{itemize}
  \item $\pi\colon\wt{X}\to X$ is one-to-one away from $E$. I in particular, the restriction $\pi|_{\wt{Z}}$ is one-to-one
    away from $\wt{Z}\cap E$; 
  \item $\wt{X}$ is a smooth surface and $\wt{Z}$ is a smooth complex curve;
  \item Each algebraic component of $E$ is a projective (that is, closed) smooth complex curve;
  \item The union $E\cup\wt{Z}$ has only transverse double points as singularities; 
\end{itemize}
As in the non-embedded case, the smooth complex curves whose union in $E$ are referred  as the exceptional components.

Given the embedded resolution, we can create a dual graph of the resolution. The construction is in two steps.
First, out of $E$, we construct the graph $G_X$ as above. Next, if $v_i\in V_{G_X}$ and $E_i$ is the corresponding component of $E$, we adjoin $|E_i\cap \wt{Z}|$ arrow vertices to $v_i$. 
We denote the resulting
graph $\Gamma_{X,Z}$. 
Recall that we work locally (topologically, we have replaced $X$ by $X\cap B_\varepsilon$). Therefore,
$\wt{Z}$ is the union of disks, each intersecting the graph $E$ precisely at one point. 
That is, every arrow vertex of $\Gamma_{X,Z}$ corresponds to a connected
component of $\wt{Z}$. 
The following result is classical. 
\begin{prop}[see e.g. \cite{Nemethi_opus_magnum}*{Proposition 3.3.8}]
  The pair $(Y_{G_X},L_{\uparrow})$ is diffeomorphic to $(L_X,L_Z)$.
\end{prop}
Our next aim is to explain the relative analog of Grauert's Theorem~\ref{thm:grauert}. As the statement is slightly technical,
we give some extra explanation. Let $g\colon X\to\C$ be a reduced analytic map such that $g^{-1}(0)=Z$. Here, \emph{reduced}
means that $g$ is not divisible by a square of a non-invertible analytic function on $X$. Then, $g$ induces an analytic map
$\wt{g}\colon\wt{X}\to\C$ via $\wt{g}=g\circ\pi$. Let $v$ be a vertex of $\Gamma:=\Gamma_{X,Z}$. The vertex $v$ corresponds either to
an exceptional component $E_v$ (if $v$ is a non-arrow vertex), or to a component $\wt{Z}_v$ of $\wt{Z}$, if $v$
is an arrow vertex. In both cases, $\wt{g}$ vanishes on that component. We let $m_v>0$ denote the order of vanishing.
This quantity is called the \emph{multiplicity} of the vertex $v$.
Note that since $g$ is  reduced, $m_v=1$ for all arrow vertices;
see \cite[Section 4.3.2]{NemethiSzilard}.
The multiplicities and the weights satisfy the following compatibility relation (see \cite[Equation (4.1.5)]{NemethiSzilard}):
\begin{equation}\label{eq:compat}
  \lambda_v m_v+\sum_{w\in V_v} m_w=0,
\end{equation}
for each non-arrow vertex $v\in V_G$, where $V_v$ denote the set of all vertices in $V_\Gamma$ adjacent to $v$.  Note that \eqref{eq:compat},
together with the condition $m_v=1$ for all arrow vertices,
determines uniquely all other multiplicities.
However, unless $\Inc_G$ is unimodular,
the multiplicities need not be integral. If that is the case, such a plumbed link cannot be realized as an embedded link of an analytic
singularity.

We now state a relative analog of Grauert's Theorem~\ref{thm:grauert}. For reference, see \cite[Corollaire 5.5]{Pichon}.
\begin{prop}\label{prop:Winter}
  Let $\Gamma$ be a graph with vertices $V=V_G\cup V_\uparrow$ such that $\Inc_G$ is negative definite. If assigning multiplicity $1$
  to each arrow vertex of $\Gamma$ leads to integral positive multiplicities on all vertices of $V_G$ via the compatibility relation~\eqref{eq:compat}, 
  then $(\YG,L_\uparrow)$ is a link
  of an embedded analytic singularity.
\end{prop}

\subsection{Rationality}\label{sub:ratio}

Suppose $(X,x_0)$ is an analytic singularity. We define the \emph{geometric genus} $p_g=h^1(\mathcal{O}_{\wt{X}})$;
see \cite[Section 2]{NemethiFive}. Many properties of geometric genus are given in various chapters of \cite{Nemethi_opus_magnum}. 
The definition of $p_g$ does not depend on the choice of resolution. Geometric genus
is an invariant of the analytic structure of $X$; there are known examples of singularities with the same link, but different geometric
genus, see \cite[Paragraph 4.8]{NemethiFive}. Put differently, in general $p_g$ cannot be read off from the combinatorics of the resolution graph $G$.

An exception is the case of \emph{rational} singularities, which are characterized by the property that $p_g=0$; see \cite[Section 7.1]{Nemethi_opus_magnum}. Given a graph $G$,
we can determine, whether it represents a rational 
singularity; this result is due to Artin \cite{Artin}, see also \cite[Theorem 3.8]{NemethiFive} and \cite{Nemethi_opus_magnum}*{Theorem 7.1.2}.
For instance, if $Y_G$ is a link of a rational singularity, then $b_1(Y_G)=0$.

  By studying the relation between $p_g$ and combinatorial invariants of $X_G$ encoded by $G$, N\'{e}methi proves the following groundbreaking result:

\begin{thm}[\cite{Nemethi_LO_and_Links}]\label{thm:nemethiLO}
  The singularity $(X,x_0)$ is rational if and only if $L_X$ is an L-space.
\end{thm}
The proof of Theorem~\ref{thm:nemethiLO} can also be done using recently proved
equivalence of lattice and Heegaard Floer homology \cite{ZemHFLattice}.
As the first step, one uses
N\'emethi's theorem stating that if $G$ is negative definite, then $X$ is a rational singularity if and only if
the reduced lattice homology of $Y_G$ is zero; see \cite{NemethiAR,NemethiLattice}.
Next, one refers to \cite{ZemHFLattice} to show the reduced lattice homology of $Y_G$ vanishes
if and only if $Y_G$ is an L-space.

\begin{rem}
   Theorem~\ref{thm:nemethiLO} characterizes only graphs representing an L-space among graphs with negative definite incidence matrix $\Inc_G$. 
   There are indefinite graphs representing L-spaces. For example,
   if $G$ is a linear plumbing such that $\Inc_G$ is non-degenerate, then $\YG$ is a lens space, regardless of
  whether $\Inc_G$ is definite or not.
   To the best of our knowledge, there is not a generalization of Theorem~\ref{thm:nemethiLO} for indefinite graphs.
\end{rem}

\subsection{Algebraic links and L-space links}\label{sub:algebraic}

Let us recall the following definition.
\begin{define}\label{def:rai} Let $L$ be a link in a rational homology $3$-sphere. We say
    $L$ is an \emph{$L$-space link} if all sufficiently large positive  surgeries on $L$ are $L$-spaces. 
\end{define}

Though usually one considers L-space links in $S^3$ or an integer homology sphere (see e.g. \cite{LiuSurgeries,GorskyNemethiAlgebraicLinks}), we note that the same definition can be applied to links in rational homology 3-spheres. Suppose $K$ is a rationally null-homologous knot in $Y$, then Morse framings on $K$ can be identified with an affine $\Z$ subspace of $\Q$ by taking the intersection number of the framing (viewed as a parallel longitude of $K$) with a rational Seifert surface. Hence, large surgeries on rationally null-homologous links are surgeries with Morse framings which are sufficiently large in $\Q^{\ell}$ with respect to this identification.

In \cite{GorskyNemethiAlgebraicLinks} Gorsky and N\'emethi studied which plumbed links in plumbed manifolds
are L-space links. Their main result is that an algebraic link in $S^3$ is an L-space link, \cite{GorskyNemethiAlgebraicLinks}*{Theorem 2}. 
Their proof 
works in a more general setting, leading
to the following statement, which is given in \cite{GorskyNemethiAlgebraicLinks}*{Theorem~12} and the ensuing discussion.

\begin{prop}\label{prop:GN_algebraic}
  Suppose $\Gamma$ is a graph such that $\Inc_G$ is negative definite and 
  $\YG$ is a link of a rational singularity. Then $L_\uparrow$ is an $L$-space link.
\end{prop}
We stress that the assumption on multiplicities as in Proposition~\ref{prop:Winter}
is never used in the proof of Proposition~\ref{prop:GN_algebraic}. 
That is, Proposition~\ref{prop:GN_algebraic} does not require that the link $L_\uparrow$ be a link
of an analytic singularity; it only makes a restriction on the graph $G$ being the graph representing a rational singularity.

To see this, we quickly sketch the argument of \cite{GorskyNemethiAlgebraicLinks}*{proof of Theorem~2} proving Proposition~\ref{prop:GN_algebraic}.
One first extends the graph $G$ to another graph, $\wt{G}_0$,
which replaces all arrowhead vertices of $\Gamma$ by chains of $-2$  weighted vertices ended by a $-1$ weighted vertex.  By construction,
$\wt{G}_0$ can be contracted to $G$ by successive blow-downs, that is $\wt{G}_0$ represents the manifold $\YG$. Hence, it is a rational graph. Next, a sufficiently large positive surgery on $L_\uparrow$
can be presented as a subgraph of $\wt{G}_0$ as long as the chains of $-2$ weighted vertices are long enough. Since any subgraph of a rational graph is rational by Laufer's criterion, large positive
surgeries on $L_\uparrow$ are represented by rational graphs.  Rational graphs represent L-spaces by \cite{NemethiAR}*{Theorems 6.3 and 8.3}.
That is to say,
a sufficiently large positive surgery on $L_\uparrow$ is an L-space. This means, that $L_\uparrow$ is an L-space link.


It is well-known that an algebraic knot is determined by its Alexander polynomial. A natural question is whether this result generalizes to
 plumbed L-space links.
The following well-known fact is due to Yamamoto.
\begin{prop}[see \cite{Yamamoto}]
  Two algebraic links in $S^3$ with the same Alexander polynomial are equal.
\end{prop}
The result of Yamamoto relies on the classification of algebraic links in $S^3$, due to Zariski \cite{Zariski};  we refer to \cite{EisenbudNeumannGraphLinks}
for this characterization. In particular, this result does not admit direct generalizations to links in other 3-manifolds. 
We give now a few counterexamples for some naive attempts to generalize the result. Proposition~\ref{prop:non_isotopic} is not used
in the present paper. Rather it indicates that algebraic links cannot be distinguished by Alexander polynomials, hence, they cannot
be distinguished by Heegaard Floer homology.
\begin{prop}\ \label{prop:non_isotopic}
  \begin{itemize}
    \item[(a)] There exist non-isotopic plumbed L-space links in $S^3$ with the same Alexander polynomial;
    \item[(b)] There exist non-isotopic knots that are links of embedded analytic surface singularities with the same Alexander polynomial.
  \end{itemize}
\end{prop}
\begin{proof}
  Item (a) is classical. We know that the $(2,3)$-cable on the positive trefoil is a plumbed knot, and it is an L-space knot by \cite{HeddencablingII}*{Theorem~1.10}. 
  However, its Alexander
  polynomial is the same as that of the $T(3,4)$ torus knot.

  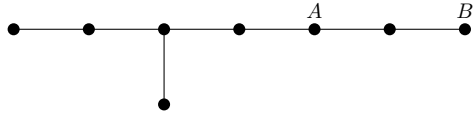
\begin{figure}
    \begin{tikzpicture}
      \draw (-2,0) -- (4,0);
      \draw (0,0) -- (0,-1);
      \foreach \x in {-2,-1,0,1,2,3,4} \fill[black] (\x,0) circle (0.08);
      \fill[black] (0,-1) circle (0.08);
      \draw (2,0.25) node [scale=0.8] {$A$};
      \draw (4,0.25) node [scale=0.8] {$B$};
    \end{tikzpicture}
    \caption{Resolution graph of the $E_8$ singularity. All vertices correspond to spheres with self-intersection $-2$. The meaning
    of components $A$ and $B$ is explained in the text.}\label{fig:plumbing}
  \end{figure}
  Item (b) expands on results of Campillo, Delgado and Gussein-Zade \cite{CDG}. 
  In fact, in \cite{CDG}*{Section 3, Example 2} there are two knots in the Poincar\'e sphere with the same Alexander polynomial, represented by two plumbing diagrams. We quickly recall their construction. The starting point is the resolution of the $E_8$ singularity given by the plumbing
  graph in Figure~\ref{fig:plumbing}. The first knot is obtained by taking the plumbing diagram of the $E_8$ singularity and adding to it
  an arrowhead vertex at the component marked $A$ in Figure~\ref{fig:plumbing}. 

  The second knot is obtained by drawing an $A_4$-singularity
  (i.e. with local equation $x^5-y^2=0$) transversally to a point at the $B$ component. The resolution of that singularity yields the plumbing graph drawn in Figure~\ref{fig:plumbing2}, compare~\cite{CDG}*{Figure 2}.
  \begin{figure}
    \begin{tikzpicture}
      \draw (-2,0) -- (8,0);
      \draw (0,0) -- (0,-1);
      \draw[->] (7,0) -- ++(0.8,0.8);
      \foreach \x in {-2,-1,...,8} \fill[black] (\x,0) circle (0.08);
      \fill[black] (0,-1) circle (0.08);
      \draw (4,-0.3) node [scale=0.7] {$-3$};
      \draw (7,-0.3) node [scale=0.7] {$-1$};
      \draw (8,-0.3) node [scale=0.7] {$-3$};
    \end{tikzpicture}
    \caption{Plumbing graph of the second knot in Proposition~\ref{prop:non_isotopic}(b). All weights that are not explicitly marked have
    value $-2$.}\label{fig:plumbing2}
  \end{figure}
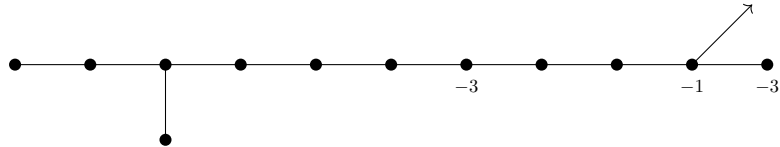

  An explicit
  algorithm described in \cite{EisenbudNeumannGraphLinks}*{Chapter 20} transforms these two plumbing graphs into
  graph links, which are presented in Figure~\ref{fig:two_links}.
  Using the algorithm of \cite{NeumannInvariants}, we compute the signature functions of these links, and we present them in Figure~\ref{fig:first_graph}, omitting straightforward calculations. The signatures are different, so the knots are different.
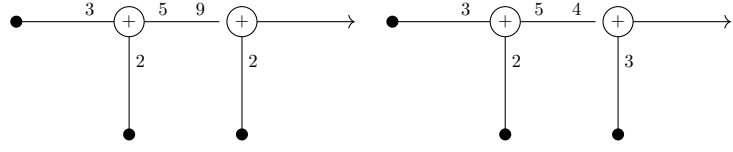
\begin{figure}
  \begin{tikzpicture}
\def\plusnode#1#2{\draw(#1,#2) circle (0.2); \draw(#1,#2) node[scale=0.7] {$+$};}
\def\vertex#1#2{\fill[black] (#1,#2) circle(0.08);}
    \begin{scope}[xshift=-2.5cm]
      \plusnode{0}{0}
      \plusnode{1.5}{0}
      \vertex{-1.5}{0}
      \vertex{0}{-1.5}
      \vertex{1.5}{-1.5}
      \draw(-1.5,0) -- node[near end,above,scale=0.7] {$3$} (-0.2,0);
      \draw(0,-1.5) -- node[near end,right,scale=0.7] {$2$} (0,-0.2);
      \draw(1.5,-1.5) -- node[near end,right,scale=0.7] {$2$} (1.5,-0.2);
      \draw[->] (1.7,0) -- (3,0);
      \draw(0.2,0) -- node[near start,above,scale=0.7] {$5$} node[near end,above,scale=0.7] {$9$} (1.2,0);
    \end{scope}
    \begin{scope}[xshift=2.5cm]
      \plusnode{0}{0}
      \plusnode{1.5}{0}
      \vertex{-1.5}{0}
      \vertex{0}{-1.5}
      \vertex{1.5}{-1.5}
      \draw(-1.5,0) -- node[near end,above,scale=0.7] {$3$} (-0.2,0);
      \draw(0,-1.5) -- node[near end,right,scale=0.7] {$2$} (0,-0.2);
      \draw[->] (1.7,0) -- (3,0);
      \draw(1.5,-1.5) -- node[near end,right,scale=0.7] {$3$} (1.5,-0.2);
      \draw(0.2,0) -- node[near start,above,scale=0.7] {$5$} node[near end,above,scale=0.7] {$4$} (1.2,0);
    \end{scope}
  \end{tikzpicture}
  \caption{The two links of \cite{CDG} represented as graph links of \cite{EisenbudNeumannGraphLinks}.}\label{fig:two_links}
\end{figure}
\begin{figure}
  \begin{tikzpicture}
    \draw[->](-0.5,0) -- (5.3,0);
    \draw[->](0,-1.5) -- (0,0.5);
    \draw[thin] (0.4,-0.1) -- (0.4,0.1); \draw (0.4,0.3) node[scale=0.6] {$\frac{1}{18}$};
    \draw[thin] (0.6,-0.1) -- (0.6,0.1); \draw (0.6,0.3) node[scale=0.6] {$\frac{1}{12}$};
    \draw[thin] (1.2,-0.1) -- (1.2,0.1); \draw (1.2,0.3) node[scale=0.6] {$\frac{1}{6}$};
    \draw[thin] (2,-0.1) -- (2,0.1); \draw (2,0.3) node[scale=0.6] {$\frac{5}{18}$};
    \draw[thin] (2.8,-0.1) -- (2.8,0.1); \draw (2.8,0.3) node[scale=0.6] {$\frac{7}{18}$};
    \draw[thin] (3,-0.1) -- (3,0.1); \draw (3,0.3) node[scale=0.6] {$\frac{5}{12}$};
    \draw[thin] (5,-0.1) -- (5,0.1); \draw (5,0.3) node[scale=0.6] {$\frac12$};
    \draw[thin,dotted] (-0.5,-0.3) -- (5,-0.3); \draw(-0.2,-0.2) node[scale=0.6] {$-2$};
    \draw[thin,dotted] (-0.5,-0.6) -- (5,-0.6); \draw(-0.2,-0.5) node[scale=0.6] {$-4$};
    \draw[thin,dotted] (-0.5,-0.9) -- (5,-0.9); \draw(-0.2,-0.8) node[scale=0.6] {$-6$};
    \draw[thin,dotted] (-0.5,-1.2) -- (5,-1.2); \draw(-0.2,-1.1) node[scale=0.6] {$-8$};
    \draw[thin,dotted] (-0.5,-1.5) -- (5,-1.5); \draw(-0.2,-1.4) node[scale=0.6] {$-10$};
    \draw[ultra thick,blue](0,0) -- (0.4,0);
    \draw[ultra thick,blue](0.4,-0.3) -- (0.6,-0.3);
    \draw[ultra thick,blue](0.6,0) -- (1.2,0);
    \draw[ultra thick,blue](1.2,-0.3) -- (2,-0.3);
    \draw[ultra thick,blue](2,-0.6) -- (2.8,-0.6);
    \draw[ultra thick,blue](2.8,-0.9) -- (3,-0.9);
    \draw[ultra thick,blue](3,-1.2) -- (5,-1.2);
    \begin{scope}[xshift=6cm]
    \draw[->](-0.5,0) -- (5.3,0);
    \draw[->](0,-1.2) -- (0,0.9);
    \draw[thin] (0.4,-0.1) -- (0.4,0.1); \draw (0.4,0.3) node[scale=0.6] {$\frac{1}{18}$};
    \draw[thin] (0.6,-0.1) -- (0.6,0.1); \draw (0.6,0.3) node[scale=0.6] {$\frac{1}{12}$};
    \draw[thin] (1.2,-0.1) -- (1.2,0.1); \draw (1.2,0.3) node[scale=0.6] {$\frac{1}{6}$};
    \draw[thin] (2,-0.1) -- (2,0.1); \draw (2,0.3) node[scale=0.6] {$\frac{5}{18}$};
    \draw[thin] (2.8,-0.1) -- (2.8,0.1); \draw (2.8,0.3) node[scale=0.6] {$\frac{7}{18}$};
    \draw[thin] (3,-0.1) -- (3,0.1); \draw (3,0.3) node[scale=0.6] {$\frac{5}{12}$};
    \draw[thin] (5,-0.1) -- (5,0.1); \draw (5,0.3) node[scale=0.6] {$\frac12$};
    \draw[thin,dotted] (-0.5,0.6) -- (5,0.6); \draw(-0.2,0.7) node[scale=0.6] {$2$};
    \draw[thin,dotted] (-0.5,-0.3) -- (5,-0.3); \draw(-0.2,-0.2) node[scale=0.6] {$-2$};
    \draw[thin,dotted] (-0.5,-0.6) -- (5,-0.6); \draw(-0.2,-0.5) node[scale=0.6] {$-4$};
    \draw[thin,dotted] (-0.5,-0.9) -- (5,-0.9); \draw(-0.2,-0.8) node[scale=0.6] {$-6$};
    \draw[thin,dotted] (-0.5,-1.2) -- (5,-1.2); \draw(-0.2,-1.1) node[scale=0.6] {$-8$};
    \draw[ultra thick,blue](0,0) -- (0.4,0);
    \draw[ultra thick,blue](0.4,0.6) -- (0.6,0.6);
    \draw[ultra thick,blue](0.6,0) -- (1.2,0);
    \draw[ultra thick,blue](1.2,-0.3) -- (2,-0.3);
    \draw[ultra thick,blue](2,-0.6) -- (2.8,-0.6);
    \draw[ultra thick,blue](2.8,-0.3) -- (3,-0.3);
    \draw[ultra thick,blue](3,-0.6) -- (5,-0.6);
  \end{scope}
  \end{tikzpicture}
  \caption{The signature functions $x\to \sigma(e^{2\pi i x})$, $x\in[0,1/2]$ for links of Figure~\ref{fig:two_links}.}\label{fig:first_graph}
\end{figure}
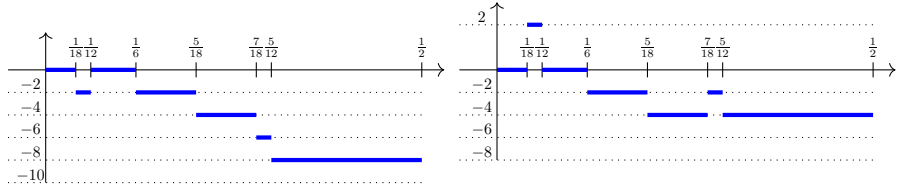
\end{proof}
\begin{rem}
  The signatures might be different, but the discontinuities of the signature function appear at the same places. This is
 consistent with the fact that the Alexander polynomials of the two knots are equal.  (It is well-known that the jumps of the signature functions occur only at roots of the Alexander polynomial.)
  The fact that the signatures of the two knots in Proposition~\ref{prop:non_isotopic}(b) are different means not only that the knots
  are not isotopic, but also that they are not concordant.
\end{rem}

\section{Link lattice homology}\label{sec:lattice}

In this section, we recall some basics about Heegaard Floer homology, and subsequently define our link lattice complex.

\subsection{Background on link Floer homology}
To fix the notation and terminology, we give some necessary
background on link Floer homology. We assume some familiarity
with basics of Heegaard Floer homology and its refinements for
knots, see \cite{OSDisks,OSKnots,OSLinks,RasmussenKnots}.

Let $L$ be an $\ell$-component link in a 3-manifold $Y$.
Recall \cite{OSLinks}*{Section 3.5} 
that the pair $(Y,L)$ can be encoded in a multi-pointed
Heegaard link diagram $(\Sigma,\ve{\alpha},\ve{\beta},\ve{w},\ve{z})$, as follows:
\begin{enumerate}
  \item $\Sigma$ is a closed oriented genus $g$ surface;
  \item $\ve{\alpha}=\{\alpha_1,\dots,\alpha_{g+\ell-1}\}$ and $\ve{\beta}=\{\beta_1,\dots,\beta_{g+\ell-1}\}$ are collections of simple closed curves on $\Sigma$. The curves $\alpha_i$ are pairwise non-intersecting. Also, the curves $\beta_i$ are pairwise non-intersecting. Moreover,
    $\ve{\alpha}$ and $\ve{\beta}$ each span a $g$-dimensional
    subspace of $H_1(\Sigma;\Z)$;
  \item $\ve{w}=\{w_1,\dots,w_\ell\}$, $\ve{z}=\{z_1,\dots,z_\ell\}$. Each component of $\Sigma\setminus\ve{\alpha}$ (respectively of $\Sigma\setminus\ve{\beta}$) contains a single point of $\ve{w}$ and a single point of $\ve{z}$.
  \end{enumerate}
 It is not hard to see there exists a Heegaard link diagram for any pair $(Y,L)$.  Furthermore, any two diagrams can be connected by sequence of Heegaard moves for link diagrams. See \cite{OSLinks}*{Theorem~4.7}.

Given a Heegaard link diagram, we consider Lagrangian tori
\[\bT_\alpha=\alpha_1\times\dots\times\alpha_{g+\ell-1},\quad 
\bT_\beta=\beta_1\times\dots\times\beta_{g+\ell-1}\]
in the symmetric product $\Sym^{g+\ell-1}(\Sigma)$. 
The link Floer chain complex, $\cCFL(Y,L)$, is a free chain complex
over
\[
\scR_\ell=\bF[\scU_1,\scV_1,\dots,\scU_\ell,\scV_\ell]
\]
generated
by intersection
points $\ve{x}\in\bT_\alpha\cap\bT_\beta$ with the differential
counting pseudo-holomorphic curves in $\Sym^{g+\ell-1}(\Sigma)$
via:
\[\partial \ve{x}=\sum_{\ve{y}\in\bT_\alpha\cap\bT_\beta}
\sum_{\substack{\phi\in\pi_2(\ve{x},\ve{y})\\ \mu(\phi)=1}} (\#\cM(\phi)/\R) \scU_1^{n_{w_1}(\phi)}\cdots\scU_\ell^{n_{w_\ell}(\phi)}\scV_1^{n_{z_1}(\phi)}\cdots\scV_\ell^{n_{z_\ell}(\phi)}\ve{y}.\]
Here the sum is taken over all homotopy classes $\pi_2(\ve{x},\ve{y})$
of maps $\phi$ of a unit disk $\bD\subset \bC$ to $\Sym^{g+\ell-1}(\Sigma)$,
where $\phi(-1)=\ve{x}$, $\phi(1)=\ve{y}$, $\phi(\partial \bD \cap \{\im(z)\le 0\})\subset\bT_\alpha$, $\phi(\d \bD \cap \{\im(z)\ge 0\})\subset\bT_\beta$. Here, $\mu(\phi)$ denotes the Maslov index of the class $\phi$.  The
space $\cM(\phi)$ consists of all pseudo-holomorphic curves
representing the class $\phi$, for a generic 1-parameter family of almost complex structures on $\Sym^{g+\ell-1}(\Sigma)$. 
 For $x\in\Sigma\setminus(\ve{\alpha}\cup\ve{\beta})$, we denote by 
$n_x(\phi)$ the intersection number of $\{x\}\times\Sym^{g+\ell-2}(\Sigma)\subset \Sym^{g+\ell-1}(\Sigma)$ with $\phi(\bD)$. We refer to
\cite{OSDisks}
for more details.

There is a map $\frs_w$ from $\bT_\alpha\cap\bT_\beta$ to
the set of $\Spin^{c}$ structures on $Y$. The component of the map $\partial$ from $\ve{x}$ to $\ve{y}$ can be non-trivial only if $\frs_w(\ve{x})=\frs_w(\ve{y})$. That is to say, the chain complex $\cCFL(Y,L)$ splits as a direct sum over complexes $\cCFL(Y,L,\frs)$, for $\frs\in\Spin^{c}(Y)$.

There is completed version of $\cCFL(Y,L)$ regarded as a module over the ring of power series
\[
\ve{\scR}_\ell:=\bF[[\scU_1,\scV_1,\dots, \scU_\ell,\scV_\ell]],
\]
namely we set
\[\ve{\cCFL}(Y,L):=\cCFL(Y,L)\otimes_{\scR_\ell}\ve{\scR}_\ell.\]
In other words, $\ve{\cCFL}(Y,L)$ has the same generators as $\cCFL(Y,L)$ and the same differential, except that  we work over a larger ring. The
completed version appears in the surgery formula.

When $Y$ is a rational homology 3-sphere, the link Floer homology groups have several gradings. Firstly, there is a $\Q\times \Q$-valued Maslov bigrading, denoted $(\gr_w,\gr_z)$, as well as a $\Q^\ell$-valued Alexander grading $A$. Furthermore
\[
(\gr_w,\gr_z)(\scU_i)=(-2,0)\quad (\gr_w,\gr_z)(\scV_i)=(0,-2) \quad \text{and} \quad A(\scV_i)=-A(\scU_i)=e_i,
\]
where $e_i$ is the standard $i$-th coordinate vector in $\Q^\ell$. Note also that
\[
\gr_w-\gr_z=2\sum_{i=1}^\ell A_i.
\]

\begin{rem} \label{rem:grading}In this paper,  we normalize $\gr_w$ so that the isomorphism 
\[
H_*(\cCFL(Y,L)/(\scV_1-1,\dots, \scV_\ell-1))\iso \HF^-(Y)
\]
 is grading preserving. In the above, we are writing $\HF^-(Y)$ for the Heegaard Floer homology computed with a singly pointed Heegaard diagram for $Y$. We make a similar normalization for $\gr_z$. Equivalently, our grading convention is that $\HF^-$ of a 3-manifold is invariant under adding extra basepoints as a graded module; compare \cite{OSLinks}*{Section~6.1}. We note that some authors normalize the Maslov gradings so that $H_*(\cCFL(Y,L)/(\scV_1-1,\dots, \scV_\ell-1)$ is isomorphic instead to $\HF^-(Y)[(\ell-1)/2]$.
\end{rem}

\subsection{Lattice homology}

We recall the definition of lattice homology \cite{NemethiLattice}. We use the notation of Ozsv\'{a}th, Stipsicz and Szab\'{o} \cite {OSSLattice}, since our construction of link lattice homology is slightly easier to describe using their notation. Let $G$ be a plumbing tree, and write $V_G$ for the vertices of $G$. Write $\mathbb{P}(V_G)$ for the power set of $V_G$ (i.e. the set of all subsets of $V_G$). The lattice complex is the $\bF[[U]]$ module
\[
\bC\bF(G):=\prod_{[K,E]\in \Char(G)\times \bP(V_G)} \bF[[U]]\otimes \langle [K,E]\rangle,
\]
where $\Char(G)\subset H^{2}(\XG, \mathbb{Z})$ denotes the set of characteristic elements of $H^2(\XG,\Z)$ on the 4-manifold $\XG$. Recall that $K\in H^2(\XG,\mathbb{Z})$ is \emph{characteristic}, if $K(x)\equiv x\cdot x\bmod 2$ for all $x\in H_2(\XG,\Z)$.

We now define the differential on $\mathbb{CF}(G)$. Note that each vertex $v\in V_G$ determines an element of $H_2(\XG)$, which is class in $H_2(\XG)$ of the base space
of the disc bundle $T_v$ used in the plumbing construction. For $I\subset E$, one defines
\[
2 f(K,I)=\left( \sum_{v\in I} K(v)\right)+\left(\sum_{v\in I} v \right)\cdot \left( \sum_{v\in I} v\right).
\]
Note that the right-hand side of the above equation is always an even integer, because $K$ is characteristic. 
In particular, $f(K,I)$ is an integer.
We set $g(K,E)=\min\{f(K,I): I\subset E\}$.  Next, one defines
\[
A_v(K,E)=g(K,E-v)\qquad B_v(K,E)=\min\{f(K,I):v\in I\subset E\}.
\]
Set
\[
a_v(K,E)=A_v(K,E)-g(K,E)\quad \text{and} \quad b_v(K,E)=B_v(K,E)-g(K,E).
\]
By the definition of $g(K, E)$, one can see that $g(K, E-v)\geq g(K, E)$. Similarly, $B_{v}(K, E)\geq g(K, E)$. Hence, $a_v(K, E)$ and $b_v(K, E)$ are both nonnegative integers. 

The differential on $\mathbb{CF}(G)$ is defined by the formula,
\begin{equation}\label{eq:differential_on_CF}
\d [K,E]=\sum_{v\in E} U^{a_v(K,E)}\otimes [K,E-v]+\sum_{v\in E} U^{b_v(K,E)} \otimes [K+2v^*, E-v],
\end{equation}
where $v^*$ is the Poincar\'e dual to $v$. Note that because of the factor $2$, $K+2v^*$ is characteristic if and only if $K$ is.
Equation~\eqref{eq:differential_on_CF} is
extended linearly over $\bF[[U]]$. We will refer to the first summand in \eqref{eq:differential_on_CF} as the \emph{A-terms} in the differential, and we will refer to the second summand as the \emph{B-terms} of the differential.

\subsection{The link lattice complex}

We now suppose that $\Gamma$ is a plumbing tree, whose vertex set is partitioned into two sets:
\[
V_\Gamma=V_G\cup V_\uparrow.
\]
Recall that the components of $V_G$ are equipped with a framing, while those of $V_\uparrow$ are not.

The vertices $V_\uparrow$ determine a link $L_\uparrow$ in the 3-manifold $\YG$. We assume that each component of $L_\uparrow$ is rationally null-homologous in $\YG$. This occurs, for example, when the incidence matrix $\Inc_G$ is non-singular.

To define the link lattice complex, we first pick a framing on the components of $V_\uparrow$ arbitrarily. In Proposition~\ref{prop:ind-framing}, we will show that the choice of framing on $V_\uparrow$ does not affect the link lattice complex.

We define the link lattice complex $\mathbb{CFL}(\Gamma,V_\uparrow)$ as the quotient of $\mathbb{CF}(\Gamma)$ by the subspace generated over $\bF[[U]]$ by tuples $[K,E]$ where $V_\uparrow\not\subset E$. Equivalently, we may view $\mathbb{CFL}(\Gamma,V_\uparrow)$ as being generated by $[K,E]$ where $V_\uparrow\subset E$, equipped with quotient complex differential. We think of the differential on $\mathbb{CFL}(\Gamma,V_\uparrow)$ as being given by the same formula as Equation~\eqref{eq:differential_on_CF}, except with the sums being taken over only $v\in E\cap V_G$.

\subsection{The module structure}
\label{sec:module-structure}

Recall that
\[
\scR_\ell=\bF[\scU_1,\scV_1,\dots, \scU_\ell,\scV_\ell].
\]
We now describe the action of $\scR_\ell$ on link lattice homology, where $\ell=|V_\uparrow|$. Write $V_\uparrow=\{v_1,\dots, v_\ell\}$. For each $i\in \{1,\dots, \ell\}$, there is an induced element $\mu_i^*\in H_2(X_{\Gamma};Y_{\Gamma})\iso H^2(X_{\Gamma})$. This element is dual to the class $v_i\in H_2(X_\Gamma)$  in the sense that $\mu^*_i(v_i)=1$, and $\mu^*_i(w)=0$ for all $w\in V_\Gamma\setminus \{v_i\}$. The class $\mu^*_i$ is represented by the co-core disk of the 2-handle corresponding to $v_i$.

Define the quantities:
\[
\delta^+_i(K,E)=g(K+2\mu_i^*,E)-g(K,E)\quad \text{and} \quad \delta_i^-(K,E)=g(K-2\mu_i^*,E)-g(K,E).
\]
An easy computation shows that
\[
f(K\pm 2\mu^*_i,I)=f(K,I)\pm 1
\]
if $v_i\in I$, and $f(K\pm 2\mu^*_i,I)=f(K,I)$ if $v_i\not \in I$. 
In particular, we have that
\[
\delta^{+}_i(K,E)\in \{0, 1\}\quad \text{and} \quad \delta^{-}_i(K,E)\in \{0, -1\}
\]
for all $i$.

For $i\in \{1,\dots, \ell\}$, we define
\begin{equation}
\scU_i\cdot [K,E]=
\begin{cases}
U[K-2\mu^*_i,E]&\text{ if } \delta^-_i(K,E)=0\\
[K-2\mu^*_i,E]& \text{ if } \delta^-_i(K,E)=-1,
\end{cases}
\label{eq:def:U}
\end{equation}
and
\begin{equation}
\scV_i\cdot [K,E]=
\begin{cases}
U[K+2\mu^*_i,E]&\text{ if } \delta^+_i(K,E)=1 \\
[K+2\mu^*_i,E]& \text{ if } \delta^+_i(K,E)=0.
\end{cases}
\label{eq:def:V}
\end{equation}
We extend $\scU_i$ and $\scV_i$ to the entire link lattice complex by declaring them to be $\bF[U]$ equivariant. 
Equivalently, we set
\begin{equation}
\begin{split}\scU_i \cdot [K,E]&=U^{g(K-2\mu^*_i,E)-g(K,E)+1}[K-2\mu^*_i,E]
\quad \text{and} \\
\scV_i\cdot [K,E]&= U^{g(K+2\mu^*_i,E)-g(K,E)}[K+2\mu^*_i,E].
\end{split}
\label{eq:UV-def}
\end{equation}

\begin{lem}
 If $v_i\in V_\uparrow$, then the endomorphisms $\scU_i$ and $\scV_i$ are chain maps.
\end{lem}
\begin{proof}
  The differential of $\mathbb{CFL}(\Gamma,V_\uparrow)$ is given by modifying ~\eqref{eq:differential_on_CF} to sum over only $v\in V_G$. 
 Clearly the summands of $\scV_i\d [K,E]$ are in bijection with the summands of $\d \scV_i[K,E]$. It remains to show that the powers of $U$ coincide. We consider the $A$-terms of the differential first. The power of $U$ from $\scV_i \d[K,E]$ of $[K+2\mu^*_i,E-v]$ is
 \[
 g(K+2\mu^*_i,E-v)-g(K,E-v)+g(K,E-v)-g(K,E)
 \]
 whereas the power of $U$ from $\d (\scV_i [K,E])$ is
 \[
 g(K+2\mu^*_i,E-v)-g(K+2\mu^*_i,E)+g(K+2\mu^*_i,E)-g(K,E).
 \]
 These are obviously equal. Similarly, for the $B$-terms of the differential we use the equality
 \[
 B_v(K,E)=(K(v)+v\cdot v)/2+g(K+2v^*,E-v).
 \]
 From here, the argument is similar to the case of type-$A$ terms. We note that 
 \begin{equation}\label{eq:Kvv}
(K+2\mu^*_i)(v)+v\cdot v=K(v)+v\cdot v
 \end{equation}
 if $v\neq v_i$.  For the $B$-terms of the differential, the $U$ power of the term from $\d \scV_i [K,E]$ is
 \[
 B_v(K+2\mu^*_i,E)-g(K+2\mu^*_i,E)+g(K+2\mu^*_i,E)-g(K,E).
 \]
 The $U$-power from $\scV_i \d [K,E]$ is
 \[
 g(K+2v^*+2\mu^*_i, E-v)-g(K+2v^*,E-v)+B_v(K,E)-g(K,E).
 \]
 The difference between these terms is $((K+2\mu^*_i)(v)+v\cdot v)/2-(K(v)+v\cdot v)/2$ which vanishes by \eqref{eq:Kvv}.
 
 The claim about the map $\scU_i$ follows from essentially the same logic.
\end{proof}

\begin{lem}\,
\begin{enumerate}
\item For each $i$, we have $\scU_i\scV_i=\scV_i\scU_i=U$.
\item For all $i,j$, the commutators $[\scU_i,\scV_j]$, $[\scU_i,\scU_j]$ and $[\scV_i,\scV_j]$ vanish. 
\end{enumerate}
\end{lem}
\begin{proof}
All of the stated relations are easily derived from Equation~\eqref{eq:UV-def}.
\end{proof}

\begin{lem}
\label{lem:power-series-rings}
  The action of $\scR_\ell$ on $\mathbb{CFL}(\Gamma,V_\uparrow)$ extends to an action of
  the ring of power series 
  \[\ve{\scR}_\ell= \bF[[\scU_1,\scV_1,\dots, \scU_\ell,\scV_\ell]].\]
\end{lem}
\begin{proof}We use the following fact about direct products. Suppose that $(A_i)_{i\in I}$ and $(B_j)_{j\in J}$ are families of vector spaces. Suppose that we have a function on indices $\g\colon I\times J\to J$ as well as a family of maps
\[
f_{i,j}\colon A_i\otimes B_j\to B_{\g(i,j)}.
\]
If $f_{i,j}$ and $\g$ have the property that for each $j'\in J$, there are only finitely many $i$ and $j$ so that $\g(i,j)=j'$ and $f_{i,j}\neq 0$, then there is a well-defined map 
\[
\sum_{i,j} f_{i,j}\colon \prod_{i\in I} A_i\to \prod_{j\in J} B_j,
\]
whose component functions are $f_{i,j}$.

Note that both the lattice complex and the power series ring $\bF[[\scU_1,\scV_1,\dots, \scU_\ell, \scV_{\ell}]]$ may be viewed as infinite direct products of copies of $\bF$. Hence, it is sufficient to show that for each $y=U^n\otimes [K,E]$, there are at most finitely many monomials $a\in \scR_\ell$ and generators $x=U^m\otimes [K',E]$ such that
$a\cdot x=y$. If $a=\scU_1^{i_1}\dots\scU_{\ell}^{i_\ell}\scV_1^{j_1}\dots \scV_\ell^{j_\ell}\in \scR_\ell$ is a monomial, we define $\nu(a)\in H^2(X_\Gamma)$  by the formula
\[
\nu(a)=(j_1-i_1) \mu^*_{v_1}+\cdots+ (j_\ell-i_\ell)\mu^*_{v_\ell}.
\]
 Monomials in $\scR_\ell$ are equipped with a $U$-weight
\[
w_U(a)=\min(i_1,j_1)+\cdots+\min(i_\ell,j_\ell).
\]
A generator $U^n\otimes [K,E]$ also has a $U$-weight $w_U(U^n\otimes [K,E])=n$. It is straightforward to verify that $w_U(a\cdot x)\ge w_U(a)+w_U(x)$ for any $a$ and $x$. In particular, if $y=U^n\otimes [K,E]$ is fixed, then there are only finitely many possible $U$ weights of $a$ and $x$ such that $y=a\cdot x$. 

Monomials in $\scR_\ell$ also have an Alexander grading
$A(a)=(j_1-i_1,\dots, j_\ell-i_{\ell})$.
Since $H^2(X_\Gamma)$ is torsion free and of rank $|L_\Gamma|$, we have that $\nu(a)=\nu(a')$ if and only if $A(a)=A(a')$. Also, if $a\cdot [K',E]=U^n \otimes[K,E]$, then it must be the case that $K'=K-2\nu(a)$.

Let $s\in \{1,\dots, \ell\}$.  Note that if $A_s(a)=j_s-i_s$ is sufficiently negative, then $\scV_s[K-2\nu(a),E]=U[K-2\nu(a)+2\mu^*_s,E]$ by the observation that  $(K-2\nu(a))(v_s)<0$ and the minima involving $g(K-2\nu(a),E)$ and $g(K-2\nu(a)+2\mu^*_s,E)$ will both be attained at some $I\subset E$ containing $v_s$. A similar argument shows that if $A_s(a)$ is sufficiently positive, then $\scU_s[K-2\nu(a),E]=U [K-2\nu(a)-2\mu^*_s,E]$.

In particular, it follows from the above reasoning that the set of monomials $a$ such that there is an $i$ satisfying $a\cdot U^i\otimes  [K-2\nu(a),E]=U^n\otimes [K,E]$ is bounded in $U$-weight and Alexander grading. However, it is easy to verify that the set of elements in $\scR_\ell$ in bounded Alexander grading and $U$-weight is finite, completing the proof.
\end{proof}

\subsection{Maslov gradings}
\label{sec:maslov-gradings}

If the intersection form of $X_\Gamma$ is non-singular, 
then the lattice complex $\mathbb{CFL}(\Gamma,V_\uparrow)$ inherits a Maslov grading from $\mathbb{CF}(\Gamma)$. We recall the formula
\begin{equation}
\gr(U^i\otimes [K,E])=-2i+2g(K,E)+|E|+\frac{1}{4}(K^2-3\sigma(X_{\Gamma})-2\chi(X_{\Gamma})).
\label{eq:Maslov-grading}
\end{equation}
Compare \cite{OSSKnotLatticeHomology}*{Section~2.3}. Here, $\sigma(X_\Gamma)$ and $\chi(X_\Gamma)$ are the signature and Euler characteristic, respectively. The number $K^2$ is obtained by factoring $K$ from $H^2(X_\Gamma)$ to $H^2(X_\Gamma, \d X_\Gamma)$, squaring using the cup product, and then evaluating on the fundamental class $[X_\Gamma, \d X_\Gamma]\in H_4(X_\Gamma, \d X_\Gamma)$.

In the setting of link lattice homology, it is more natural to 
define the Maslov grading via the formula
\begin{equation}
\gr_{w}(U^i\otimes [K,E])=-2i+2g(K,E)+|E|-|V_\uparrow|+\frac{1}{4}(K|_{\XG}^2-3\sigma(\XG)-2\chi(\XG)). \label{eq:grading-gr-w-lattice-link}
\end{equation}
This grading is defined when the intersection form $Q_G$ of $X_G$ is non-singular. More generally, this grading may also be defined when $\Inc_G$ is singular, as long as we restrict to torsion $\Spin^c$ structures on $\YG$.

\begin{lem}\label{lem:gr-w-properties} The differential $\d$ on $\mathbb{CFL}(\Gamma,V_\uparrow)$ decreases $\gr_w$ by 1. Furthermore, if $v_i\in V_\uparrow$, then $\scV_i$ preserves $\gr_{w}$ and $\scU_i$ decreases $\gr_w$ by $2$.
\end{lem}
\begin{proof}The proof that $\gr_w(\d)=-1$ is essentially identical to the proof in \cite{OSSLattice}*{Lemma~3.1}  (cf. \cite{NemethiLattice}*{}) so we will not repeat it.

We now consider the actions of $\scV_i$ and $\scU_i$. Note that $(K\pm 2\mu^*_i)|_{\XG}^2=K|_{\XG}^2$, since $\mu^*_i$ has trivial restriction to $\XG$. Hence the grading changes are entirely due to the powers of $U$ and the $g(K,E)$ terms. With this in mind, the stated grading changes follow immediately from the formulas defining the action of $\scV_i$ and $\scU_i$ in Equations~\eqref{eq:def:V} and~\eqref{eq:def:U}.
\end{proof}

\begin{rem} 
Note that $\gr_w$ and $\gr$ do not in general differ by a constant. Also $\d$ is only homogeneously graded on $\mathbb{CFL}(\Gamma,V_\uparrow)$, and not on the entire complex $\mathbb{CF}(\Gamma)$. The reader should compare the lattice complex to the Heegaard Floer mapping cone complex of Ozsv\'{a}th and Szab\'{o} \cite{OSIntegerSurgeries}. For $K\subset S^3$, with integer framing $n$, this takes the form of a mapping cone $\bX_n(K)=\Cone(v+h_n\colon \bA(K)\to \bB(K))$, which is homotopy equivalent to $\ve{\CF}^-(S^3_n(K))$. The complex $\bA(K)$ is isomorphic to the full knot Floer complex $\cCFK(K)$ (a finitely generated chain complex over $\bF[\scU,\scV]$), and hence admits a grading $\gr_w$ on $\cCFK(K)$. This does not coincide with the Maslov grading on $\bX_n(S^3_n(K))$ (see \cite{OSIntegerSurgeries}*{Section~4}) induced by the isomorphism with $\ve{\CF}^-(S^3_n(K))$.
\end{rem}

\subsection{Alexander gradings}
\label{sec:Alexander-gradings}

We now define the Alexander multi-grading. If $|V_\uparrow|=\ell$, our Alexander grading will take values in $\Q^{\ell}$. We assume that $\YG$ is a rational homology 3-sphere. 

If $v\in V_\uparrow$, we may only view $v$ as a class in $H_2(\XG,\d \XG)$, and not necessarily $H_2(\XG)$. Concretely, $[v]\in H_2(\XG,\d \XG)$ is obtained by capping one end of the link cobordism $[0,1]\times L_v\subset \XG$ with a disk in $S^3$. Since $\YG$ is a rational homology 3-sphere, the knot $L_v\subset \YG$ is rationally null-homologous, so we may lift
the class $[v]\in H_2(\XG,\d \XG)$ to a class  $[\hat v]\in H_2(\XG;\Q)$. Such a lift is given by capping with a rational Seifert surface for $L_v\subset \YG$.

\begin{define}
\label{def:Alexander-grading} Let $\Gamma$ be an arrow decorated plumbing tree, as above, and suppose $U^i \cdot[K,E]$ is a generator of $\mathbb{CFL}(\Gamma,V_\uparrow)$. Write $V_\uparrow=\{v_1,\dots, v_\ell\}$. We define the \emph{Alexander multi-grading} of $U^i\cdot [K,E]$ to be
\[
A\left(U^i [K,E]\right)=
\left(
 \frac{K(v_1-\hat v_1)+\sum_{v\in V_\uparrow}v\cdot (v_1-\hat v_1)}{2},\dots, \frac{K(v_\ell-\hat v_\ell)+\sum_{v\in V_\uparrow}v\cdot(v_\ell-\hat v_\ell)}{2}
 \right),
\]
which lies in $\Q^\ell$. 
\end{define}
Note that the Alexander multi-grading of $U^i [K,E]$ depends on $K$, but not on $E$ or $i$.

\begin{rem} More generally, if $\YG$ is not a rational homology 3-sphere, we can construct the Alexander grading $A_i$ when the corresponding component of $L_\uparrow$ is rationally null-homologous, as long as a rational Seifert surface is chosen. For our purposes, such a choice is equivalent to a choice of lift $\hat{v}_i$ of $v_i$ under the map $H_2(\XG;\Q)\to H_2(\XG,\d \XG;\Q)$. Our grading may be defined in this context as well.
\end{rem}

\begin{lem}
\label{lem:alexandergrading}
 The differential on $\mathbb{CFL}(\Gamma,V_\uparrow)$ preserves the Alexander multi-grading. Suppose that $v_i,v_j\in V_\uparrow$. The action of $\scU_i$ drops the Alexander grading $A_j$ by $\delta_{i,j}$ (Kronecker delta) and the action of $\scV_i$ increases the Alexander grading $A_j$ by $\delta_{i,j}$.
\end{lem}
\begin{proof} 
The $A$-terms of the differential obviously preserve the Alexander multi-grading since they do not change $K$. The $B$-terms send $[K,E]$ to a sum of $U$-multiples of $[K+2w^*,E-w]$, where $w$ ranges over the non-arrow components. Note that since $w\in V_G$, $w^*$ is the Poincar\'{e} dual of the element $[w]\in H_2(\XG;\Z)$, so
\[
(K+2w^*)(v_j-\hat v_j)=K(v_j-\hat v_j).
\]
Hence the $B$-terms also preserve Alexander grading.

Next, we consider the actions of $\scU_i$ and $\scV_i$. We note that $(K+2 \mu^*_i)(\hat v_j)=K(\hat v_j)$, since $\hat v_j$ and $\mu^*_i$ are represented by disjoint rational 2-chains.  On the other hand, we have
\[
(K+2 \mu^*_i)(v_j)=K(v_j)+2\delta_{i,j}.
\]
Hence, the conclusion follows from these observations. 
\end{proof}

\subsection{Conjugation symmetry}
As with the original construction of lattice homology \cite{NemethiLattice}*{Remark 3.2.7}, the link lattice complex admits a conjugation symmetry. Compare \cite{OSSKnotLatticeL}*{Section~2.2}. This takes the form of a map
\[
J(U^i[K,E])=U^i[-K-\sum_{v\in E} 2 v^*, E].
\]
We now observe that the $J$-map is skew Alexander graded:
\begin{lem}
\label{lem:Alex-J} The $J$ map satisfies
\[
A(J([K,E]))=-A([K,E]),
\]
\end{lem}
\begin{proof}
By construction of the link lattice complex, $V_\uparrow\subset E$. We observe that if $v\in E\setminus V_\uparrow$ and $v_i\in V_\uparrow$, then $2v^*(v_i-\hat v_i)=0$
since $v^*$ is Poincar\'{e} dual to a class in $H_2(\XG; \mathbb{Z})$. Hence, if $v_i\in V_\uparrow$, then
\[
\begin{split}(-K-\sum_{v\in E} 2v^*)(v_i-\hat v_i)+\sum_{v\in V_\uparrow} v\cdot(v_i-\hat v_i)
&=(-K-\sum_{v\in V_\uparrow} 2v^*)(v_i-\hat v_i)+\sum_{v\in V_\uparrow} v\cdot(v_i-\hat v_i)\\
&=-K(v_i-\hat v_i)-\sum_{v\in V_\uparrow} v\cdot (v_i-\hat v_i).
\end{split}
\]
Therefore, $A(J([K,E]))=-A([K,E])$.
\end{proof}

We recall from \cite{OSSKnotLatticeHomology}*{Section~2.3} that $J$ preserves the grading $\gr$ from Equation~\eqref{eq:Maslov-grading}. On the other, its interaction with $\gr_w$ is more interesting. Define
\[
\gr_z=\gr_w-2(A_1+\cdots+A_\ell),
\]
where $(A_1,\dots, A_\ell)$ is the Alexander grading. Note that by Lemmas~\ref{lem:gr-w-properties} and~\ref{lem:alexandergrading},
the action of $\scU_i$ preserves the $\gr_z$-grading, while the action of $\scV_i$ drops the $\gr_z$-grading by $2$.

\begin{lem} The map $J$ interchanges the $\gr_w$ and $\gr_z$ gradings. 
\end{lem}
\begin{proof} We recall the useful identity (see \cite{OSSKnotLatticeHomology}*{Equation~2.2}) that
  \begin{equation}\label{eq:gJ}
    g(K,E)-g(-K-\sum_{v\in E} 2v^*,E)=f(K,E).
  \end{equation}
Let us write
\[
f([K,E])=\frac{1}{2}(K(v_E)+v_E^2)
\]
where $v_E=\sum_{v\in E} v$. 
Therefore, 
\[
\begin{split}
\gr_w([K,E])-\gr_w(J[K,E])&=2f([K,E])+\frac{1}{4}(K|_{X_G}^2-(-K-2v_E^*)|_{X_G}^2)\\
&=K(v_E)+v_E^2-(K\cup v_E^*+v_E^*\cup v_E^*)[X_G,Y_G]
\end{split}
\]
Let us write $\hat{v}_E^*$, for the restriction of $v_E^*$ to $X_G$, pulled back from $H^2(X_G, Y_G)\to H^2(X_G)$. Use similar notation for individual vertices. Note that if $v_i\in v_\uparrow$, then $v_i^*|_{X_G}=\hat{v}_i^*$, where $\hat{v}_i^*\in H_2(X_G;\Q)$ is the 2-chain appearing in the definition of the Alexander grading.  Observe that if $v\in E\setminus V_{\uparrow}$, then $\hat{v}^*$ is still the Poincar\'{e} dual of the 2-sphere represented by $v$, which is contained in $X_G$. Let us write $v_\uparrow$ for $\sum_{v\in V_{\uparrow}} v$, and define $\hat{v}_{\uparrow}$ similarly. We compute that
\[
\begin{split}
f([K,E])+\frac{1}{4}(K|_{X_G}^2-(-K-2v_E^*)|_{X_G}^2)&=K(v_E)+v_E^2-(K\cup \hat{v}_E^*+\hat{v}_E^*\cup \hat{v}_E^*)[X_G,Y_G]\\
&=K(v_\uparrow-\hat{v}_\uparrow)+v_\uparrow^2-\hat{v}_\uparrow^2\\
&=K(v_\uparrow-\hat{v}_\uparrow)+v_{\uparrow}(v_\uparrow-\hat{v}_\uparrow). 
\end{split}
\]
Combining the above with Lemma~\ref{lem:Alex-J} we see that
\[
\gr_w([K,E])=\gr_w(J([K,E]))+2(A_1+\cdots +A_\ell)([K,E])=\gr_z(J[K,E]).
\]
The same argument shows $\gr_z([K,E])=\gr_w(J[K,E])$. 
\end{proof}

\begin{lem}
  The map $J$ skew-commutes with $\scU_i$ and $\scV_i$, i.e.
\[
J\circ \scV_i=\scU_i\circ J \quad \text{and} \quad J\circ\scU_i=\scV_i\circ J.
\]
\end{lem}
\begin{proof}
  We compute from \eqref{eq:def:U} and~\eqref{eq:def:V}:
\[
\begin{split}
\scU_i \cdot J([K,E])=& U^{1+\delta_i^-(J [K,E])} \cdot [-K-\sum_{v\in E} 2 v^*-2\mu_i^*, E]\\
J( \scV_i\cdot [K,E])=&U^{\delta_i^+([K,E])} [-K-\sum_{v\in E} 2 v^* -2\mu_i^*, E]. 
\end{split}
\]
Therefore, it suffices to show that
\[
1+\delta_i^-(J[K,E])= \delta_i^+([K,E]). 
\]
Rewriting in terms of the $g$ function, the above is equivalent to
\[
1= g(K+2\mu_i^*, E)-g(K,E)+g(-K-\sum_{v\in E} 2v^*)-g(-K-2\mu_i^*-\sum_{v\in E} 2v^*, E)
\]
By Equation~\eqref{eq:gJ}, we see that the right-hand side of the above equation is
\[
f(K+2\mu_i^*,E)-f(K,E)=\mu_i^*(\sum_{v\in E} v)=1,
\]
as claimed. (Recall that $v_i\in E$ by definition of the lattice link complex).  An analogous argument shows that $\scV_i\circ J=J\circ\scU_i$.
\end{proof}

%
%
%
%

\subsection{\texorpdfstring{$\Spin^c$}{Spinc}-structures}
\label{sec:spin^c}
We describe how the link lattice complex $\mathbb{CFL}(\Gamma,V_\uparrow)$ naturally splits over $\Spin^c(\YG)$ as a module  over $\scR_\ell$. Recall the isomorphism
\[
\Spin^c(\YG)\iso \Spin^c(\XG)/H_2(\XG).
\]
Also, the Chern class map $c_1\colon \Spin^c(\XG)\to \Char(\XG)$ is an isomorphism of affine $H^2(\XG)$-sets (where $C\in H^2(\XG)$ acts on $\Char(\XG)$ by $K\mapsto K+2C$).

In the link lattice complex, we associate the generator $[K,E]\otimes U^i$ with the $\Spin^c$ structure $[K|_{\XG}]$, viewed as an element of $\Char(\XG)/H_2(\XG)\iso \Spin^c(Y_G)$. This gives the decomposition of the $\F[[U]]$-modules
\begin{equation}\label{eq:decompose}
  \mathbb{CFL}(\Gamma,V_\uparrow)=\bigoplus_{\mathfrak{s}\in \Spin^c(Y_G)}\mathbb{CFL}(\Gamma,V_\uparrow,\mathfrak{s}).
\end{equation} 

Since the differential on $\mathbb{CFL}(\Gamma,V_\uparrow)$ is constructed by modifying \eqref{eq:differential_on_CF} to sum over only $v\in V_G$, the decomposition \eqref{eq:decompose} is preserved by $\d$.
The actions of $\scU_i$ and $\scV_i$ also preserve this decomposition, because they change $K$ to $K\pm 2\mu^*_i$, and $\mu^*_i\in H^2(X_\Gamma)$ has trivial restriction to $H^2(\XG)$. That is to say, \eqref{eq:decompose} yields the decomposition of chain complexes
of $\F[[\scU_1,\scV_1,\dots,\scU_\ell,\scV_\ell]]$-modules over the $\Spin^c$ structures of $Y_{G}$.

\subsection{Independence from the framing on arrow components}

We now show that our chain complex $\mathbb{CFL}(\Gamma,V_\uparrow)$ is independent of the choice of framing on the arrow components, up to canonical isomorphism.

Suppose that $\cG$ is a weighted plumbing tree obtained by weighting the arrow vertices of $\Gamma$ by (any) integral weights, and using the weights from $\Gamma$ on $V_G$. We obtained a model of the link lattice complex in the previous section, which we denote by $\mathbb{CFL}_{\cG}(\Gamma,V_\uparrow)$. In this section, we describe a canonical isomorphism
\[
F_{\cG,\cG'}\colon \mathbb{CFL}_{\cG}(\Gamma,V_\uparrow)\to \mathbb{CFL}_{\cG'}(\Gamma,V_\uparrow)
\]
for any two extensions $\cG$ and $\cG'$.

Let $L_{\Gamma}\subset S^3$ denote the link associated to $\Gamma$ as in Subsection~\ref{sub:plumbing_calculus}. 
 We write $L_\cG$ to denote $L_\Gamma$, equipped with the framing from $\cG$.  
 Write $n$ for $|V_\Gamma|$. Following the notation of Manolescu and Ozsv\'{a}th \cite{MOIntegerSurgery}, we define the \emph{linking lattice} $\bH(L_{\cG})$ to be the affine $\Z^{n}$ subspace of $\Q^{n}$ consisting of vectors $\ve{s}=(s_1,\dots, s_n)$ such that $s_i\in \Z+\lk(L_i,L_\Gamma-L_i)/2$. As sets, we clearly have
\[
\bH(L_\cG)=\bH(L_{\cG'}).
\]
The lattices $\bH(L_\cG)$ and $\bH(L_{\cG'})$ are distinguished by their natural actions of $H_2(X_\cG)\iso H_2(X_{\cG'})\iso \Z^n$.
The action of $H_2(X_\cG)$ on $\bH(L_\cG)$ is as follows. Given $v\in V_\Gamma$, write $\lambda_v$ for the longitude of $K_v$ determined by the framing of $K_v$. By writing  $ H_1(S^3\setminus L_\Gamma)\iso\Z^n$, we can identify $\lambda_v$ as an element in $\Z^n$, and the action of $v$ on $\bH(L_{\cG})$ can be identified as a translation by this corresponding element in $\Z^n$.

Next, there is a canonical isomorphism
\[
\Phi_{\cG}\colon \Char(X_{\cG})\to \bH(L_{\cG})
\]
given by the formula
\begin{equation}
\Phi_{\cG}(K)= \left( \frac{K(v_1)+v_\cG\cdot v_1}{2},\dots,\frac{K(v_n)+v_\cG \cdot v_n}{2} \right).\label{eq:Spin-c-lattice-iso}
\end{equation}
In the above, we write $v_\cG=v_1+\cdots+v_n\in H_2(X_\cG)$. Note that $\Phi_{\cG}$ is equivariant with respect to the action of $H_2(X_\cG)$. 

Given two weight-extensions $\cG$ and $\cG'$ of $\Gamma$, we define the group isomorphism 
\[
F_{\cG , \cG'}\colon \mathbb{CFL}_{\cG}(\Gamma,V_\uparrow)\to \mathbb{CFL}_{\cG'}(\Gamma,V_\uparrow),
\]
via the formula
\[
F_{\cG,\cG'}([K,E]\otimes U^i)=[(\Phi_{\cG'}^{-1}\circ \Phi_{\cG})(K),E]\otimes U^i.
\]
The map $F_{\cG,\cG'}$ is clearly an isomorphism of $\bF[[U]]$-modules. In fact, we have the following:

\begin{prop}\label{prop:ind-framing} The map $F_{\cG,\cG'}$ is a $(\gr_w,A)$-grading preserving chain isomorphism between $\mathbb{CFL}_{\cG}(\Gamma,V_\uparrow)$ and $\mathbb{CFL}_{\cG'}(\Gamma,V_\uparrow)$.
\end{prop}
\begin{proof} The differentials on $\mathbb{CFL}_{\cG}(\Gamma,V_\uparrow)$ and $\mathbb{CFL}_{\cG'}(\Gamma,V_\uparrow)$ are similar to Equation~\eqref{eq:differential_on_CF}, except that we take the sum only over the vertices $v\in V_G$. Note that there are no summands in the differential for $v\in V_\uparrow$.  We consider the restricted action of $H_2(\XG)\subset H_2(X_\Gamma)$ on $\Char(X_{\cG})$, $\Char(X_{\cG'})$, $\bH(L_{ \cG})$ and $\bH(L_{\cG'})$. Since the framings of the vertices of $V_G$ coincide on $X_\cG$ and $X_{\cG'}$, the identification $\bH(L_\cG)\iso \bH(L_{\cG'})$ is equivariant with respect to the action of $H_2(\XG)$.

  It is straightforward to verify that the map $\Phi_{\cG}$ is equivariant with respect to the action of $H_2(\XG)$. The same argument applies to show that $\Phi^{-1}_{\cG'}$ is equivariant as well.

Let $K\in \Char(X_{\cG})$ and write $K'=(\Phi_{ \cG'}^{-1}\circ \Phi_{\cG})(K)$. The remainder of the proof follows from the following two claims:
\begin{enumerate}
\item $K|_{\XG}=K'|_{\XG}$.
\item For all $E$, we have $f(K,E)=f(K',E). $
\end{enumerate}
We verify these two claims presently. 

The first claim follows from the proof of \cite{OSSLattice}*{Lemma~4.6}, which we repeat for the benefit of the reader using our present notation. Write $\ve{s}=\Phi_{\cG}(K)=\Phi_{\cG'}(K')$. Since $H^2(\XG)$ is torsion free, it suffices to show that $K|_{\XG}(v)=K'|_{\XG}(v)$ for each $v\in V_G$. Note that if $v_i\in V_G$, then $K(v_i)=2s_i-v_\cG\cdot v_i$. This quantity only depends on $\ve{s}$, on the framing of $K_i$ (the link component represented by the vertex $v_i$), and on the linking numbers of $K_i$ with  other link components. 
 In particular, $K(v_i)=K'(v_i)$, completing the proof of the first claim.

We now consider the second claim. By definition,
\[
2f(K,E)=K(v_E)+v_E\cdot v_E
\]
where $v_E$ is the sum of $v$ for $v\in E$. We may rearrange the above expression to obtain
\[
2f(K,E)=(v_E-v_\cG)\cdot v_E+\sum_{v\in E} (K(v)+v_\cG\cdot v)=(v_E-v_\cG)\cdot v_E+\sum_{v_i\in E} s_i.
\]
The right-hand side depends only on $\ve{s}$, the framing of $L_G$, and the linking numbers of the components of $L_\cG$, but not on the framings of $V_\uparrow$. This establishes the second claim.

From these considerations, it follows that $F_{\cG,\cG'}$ preserves the $\gr_w$-grading, and is also a chain map.

We now establish that $F_{\cG,\cG'}$ preserves the Alexander grading. Suppose that $v_i\in V_\uparrow$. The corresponding component of the Alexander grading is half of
\[
\begin{split}
K(v_i-\hat v_i)+\sum_{v\in V_\uparrow}v\cdot (v_i-\hat v_i)&=K(v_i)+\sum_{v\in V_\uparrow} v\cdot v_i-K(\hat v_i)-\sum_{v\in V_\uparrow} v\cdot \hat v_i\\
&=K(v_i)+\sum_{v\in V_\Gamma} v \cdot v_i-\sum_{v\in V_G} v\cdot v_i-K(\hat v_i)-\sum_{v\in V_\uparrow} v\cdot \hat v_i.
\end{split}
\]
The first two terms above sum to $2s_i$. The last three terms may be rewritten as follows:
\[
-\sum_{v\in V_G} \PD[v](v_i)-K(\hat v_i)-\sum_{v\in V_\uparrow} \PD[v](\hat v_i).
\]
In particular each term is the evaluation of an element of $H^2(X_\Gamma)$ on an element of $H_2(\XG;\Q)$. We observe that $\PD[v]|_{X_G}$, for $v\in \Gamma$, is independent of the framing on $V_\uparrow$. Furthermore, $K|_{X_G}=K'|_{X_G}$. Hence  $A_i([K,E])=A_i([K',E])$.

\end{proof}

\subsection{Freeness of the lattice complex}

Since $\scU_i\scV_i=U$ for all $i$, for $\ell>1$ the link lattice complex is not free over $\bF[[\scU_1,\scV_1,\dots, \scU_{\ell},\scV_{\ell}]]$. Nonetheless, we prove that the link lattice complex is free over $\bF[[\scU_i,\scV_i]]$, for each index $i$. 

\begin{prop}\label{prop:free}
Suppose that $\Gamma$ is an arrow decorated plumbing graph with a chosen vertex $v_i\in V_\uparrow$. Write $\scU_i$ and $\scV_i$ for the variables associated to $v_i$. Then the module $\mathbb{CFL}(\Gamma,V_\uparrow)$ is a completion of a free module over $\bF[[\scU_i,\scV_i]]$.
\end{prop}

\begin{proof}
We consider the set of generators $[K,E]$ modulo the equivalence relation generated by  $[K,E]\sim [K+2\mu^*_i,E]$ for all $K$ and $E$. Equivalence classes may be identified with elements of $(\Char(\Gamma)\times \bP(V_G))/\Z$, where $1\in\Z$ acts on $\Char(\Gamma)$ by $2\mu_i^*$.

Fix an equivalence class, and let $W$ denote the $\bF[U]$-span of the generators in this class. We may write $W\iso\bigoplus_{s\in \Z} W_s$ where each $W_s\iso \bF[U]$. We order the $W_s$  so that $\scV_i\cdot W_s\subset W_{s+1}$ and $\scU_i\cdot W_s\subset W_{s-1}$. 
We make the following claims, from which the result will follow fairly easily:
\begin{enumerate}[label=($f$-\arabic*), ref=$f$-\arabic*]
\item\label{free-1} For each $s$, exactly one of $\scV_i\colon W_s\to W_{s+1}$ and $\scU_i\colon W_{s+1}\to W_s$ will be multiplication by 1, and the other will be multiplication by $U$. This follows from the fact that $\scU_i\cdot \scV_i$ acts by $U$.
\item\label{free-2} If $[K,E]$ is fixed, then $[K,E]$ is not in the image of $\scU^n_i$ or $\scV^n_i$ for arbitrarily large $n$. This follows immediately from Lemma~\ref{lem:power-series-rings}.
\item  \label{free-3}
If  $\scV_i\cdot [K,E]=[K+2\mu^*_i,E]$, then $ \scV_i\cdot [K+2\mu^*_i,E]=[K+4\mu^*_i,E].$  Similarly,  if $\scU_i\cdot [K,E]=[K-2\mu_i^*,E]$ then $\scU_i\cdot [K-2\mu_i^*,E]= [K-4\mu^*_{i},E]$.
\end{enumerate}

We now prove  claim~\eqref{free-3}, focusing on the argument for $\scV_i$ since the claim about $\scU_i$ is similar.
We recall from Section~\ref{sec:module-structure} that
\[
f(K+2\mu_i^*,E)=\begin{cases} f(K,E) & \text{ if } v_i\not \in E\\
f(K,E)+1& \text{ if } v_i\in E.
\end{cases}
\]
Hence, $\delta_i^+(K,E)=0$ if and only if there is a $J\subset E$ such that $f(K,J)=g(K,E)$ and $v_i\not \in J$. In particular, if $\scV_i\cdot [K,E]=[K+2\mu_i^*,E]$, then there exists such a $J$. Hence 
\[
g(K,E)=f(K,J)=f(K+4\mu^*_i,J)\ge g(K+4\mu^*_i,E)\ge g(K,E),
\]
so we have equality throughout. It follows that $\scV_i\cdot[K+2\mu_i^*,E]=[K+4\mu_i^*,E]$. 

Note that the claim~\eqref{free-3} implies there is no $[K,E]$ which is in the image of both $\scU_i$ and $\scV_i$, since if $[K,E]$ were in the image of both, then the above claims show that 
 \[
 \scU_i\scV_i [K-2\mu_i^*,E]=\scU_i[K,E]=[K-2\mu_i^*,E],
 \]
  which contradicts $\scU_i\scV_i=U$.
  
  Claims~\eqref{free-1} and~\eqref{free-2} imply that there exist generators $[K,E]$ in $W$ which are in the image of $\scU_i$, and there also exist generators which are in the image of $\scV_i$. (This rules out the module $\bF[\scU_i,\scV_i,\scV_i^{-1}]$ and $\bF[\scU_i,\scU_i^{-1}, \scV_i]$). 

From the above considerations, we obtain that there is a unique generator $[K,E]$ such that $\scU_i [K,E]=[K-2\mu^*_i,E]$ and $\scV_i[K,E]=[K+2\mu_i^*,E]$. By \eqref{free-3} this $[K,E]$ must be a free generator of $W$ over $\bF[\scU_i,\scV_i]$. 
\end{proof}

\subsection{Type-$D$ modules over $\cK$}

\label{sec:bordered-lattice}

If $G$ is a weighted plumbing tree (without arrow vertices) and $v$ is a distinguished vertex, we now describe how to view the lattice complex as a type-$D$ module over the algebra $\cK$, described by the third author \cite{ZemBordered}.

We first recall the algebra $\cK$  from \cite{ZemBordered}. It is an algebra over the idempotent ring $\ve{I}\iso \ve{I}_0\oplus \ve{I}_1$, where each $\ve{I}_\veps\iso \bF$. We define
\[
\ve{I}_0\cdot \cK \cdot \ve{I}_0\iso \bF[\scU,\scV],\quad \ve{I}_0\cdot \cK \cdot \ve{I}_1=0\quad \text{and} \quad \ve{I}_1\cdot \cK \cdot \ve{I}_1\iso \bF[\scU,\scV,\scV^{-1}].
\]
Finally, $\ve{I}_1\otimes \cK\otimes \ve{I}_0$ is isomorphic to the direct sum of two copies of $\bF[\scU,\scV,\scV^{-1}]$, viewed as being generated by two distinguished elements $\sigma$ and $\tau$. These elements satisfy
\[
\sigma \cdot \scU =U \scV^{-1}\cdot \sigma, \quad \sigma\cdot \scV=\scV\cdot \sigma
\]
\[
\tau\cdot \scU=\scV^{-1} \cdot\tau\quad \text{and} \quad \tau\cdot \scV=U\scV\cdot \tau.
\]
where $U=\scU\scV$.

In this section, we describe how to construct a type-$D$ module $\cX(G)^{\cK}$ from the data of $\mathbb{CF}(G)$. The construction of $\cX(G)^{\cK}$ may be equivalently  described as a tensor product of the Hopf, merge and solid torus modules from \cite{ZemBordered}, though we presently give a direct construction in terms of lattices. We define $\cX(G)^{\cK}$ at the end of the section, after we prove several properties about the lattice complex.

Write $\mathbb{CF}_0(G)$ for the codimension 1 subcube of $\mathbb{CF}(G)$ generated by tuples $[K,E]$ where $v\in E$. Write $\mathbb{CF}_1(G)$ for the codimension 1 subcube generated by tuples $[K,E]$ where $v\not \in E$. We view $\mathbb{CF}(G)$ as a mapping cone
\[
\mathbb{CF}(G)\iso \Cone\left(F_v^A+F_v^B\colon \mathbb{CF}_0(G)\to \mathbb{CF}_1(G)\right),
\]
where $F_v^A$ and $F_v^B$ are the summands of the differential which are weighted by $U^{a_v(K,E)}$ and $U^{b_v(K,E)}$, respectively.

We observe that $\mathbb{CF}_0(G)$ is exactly the link lattice complex if we designate the special vertex $v$ as the sole arrow vertex. In particular, Proposition~\ref{prop:free} implies that it is a completion of a free $\bF[[\scU,\scV]]$ module (where $\scU$ and $\scV$ are the actions for $v$). 

We may define actions of $\scU$ and $\scV$ also on $\mathbb{CF}_1(G)$, using the same formulas as in Section~\ref{sec:module-structure}. We first observe that the formulas have a comparatively easier description than on $\mathbb{CF}_0(G)$:

\begin{lem}\label{lem:actions-on-CF-1} On $\mathbb{CF}_1(G)$, we have
\[
\scU\cdot [K,E]=U[K-2\mu^*,E]\quad \text{and} \quad \scV\cdot[K,E]=[K+2\mu^*,E]
\]
for all $K$ and $E$ such that $v\not \in E$.
\end{lem}
\begin{proof}If $v\not \in E$, then $f(K,E)=f(K\pm 2\mu^*,E)$, and hence 
\begin{equation}
g(K,E)=g(K\pm 2\mu^*,E).
\label{eq:g-unchanged-adding-mu}
\end{equation} Both equations follow by applying this fact to the definition of $\scU$ and $\scV$ from Section~\ref{sec:module-structure}.
\end{proof}

As a consequence of the above, we may define an action of $\scV^{-1}$ on $\mathbb{CF}_1(G)$ via the formula $\scV^{-1}\cdot [K,E]=[K-2\mu^*,E]$. As an additional consequence of Lemma~\ref{lem:actions-on-CF-1}, we have the following easy analog to Proposition~\ref{prop:free}:

\begin{cor}\label{cor:free-idempotent-1}
 The complex $\mathbb{CF}_{1}(G)$ is the completion of a free module over $\bF[\scU,\scV,\scV^{-1}]$.
\end{cor}

Let $F_v^A$ denote the $A$-term of the differential which increments $v$. Write $F_v^B$ for the $B$-term of the differential which increments $v$.
\begin{lem} The map $F_v^A$ satisfies
\[
F_v^A(\scU\cdot \ve{x})=\scU\cdot F_v^A(\ve{x})\quad \text{and} \quad F_v^A(\scV\cdot \ve{x})=\scV\cdot F_v^A(\ve{x})
\]
for all $\ve{x}\in \mathbb{CF}_0(G)$. Similarly,
\[
F_v^B(\scU\cdot \ve{x})=\scV^{-1} \cdot F_v^B(\ve{x})\quad \text{and}\quad F_v^B(\scV\cdot \ve{x})=\scV U \cdot F_v^B(\ve{x}).
\]
\end{lem}
\begin{proof} We begin with the map $F_v^A$. From direct computation,  
\[
\begin{split}
\scV\cdot  F_v^A([K,E])=&U^{g(K+2\mu^*,E-v)-g(K,E)}\cdot [K+2\mu^*,E-v]\\
=&F_v^A(\scV\cdot [K,E]).
\end{split}
\]
An entirely analogous argument shows that $\left[F_v^A,\scU\right]=0$.

Next, we consider the commutation of $F_v^B$ with $\scU$. We compute that
\[
\begin{split}
F_v^B(\scU\cdot [K,E])
&=U^{g(K+2v^*-2\mu^*,E-v)-g(K,E)+1+(K-2\mu^*)(v)/2+v^2/2}\cdot [K-2\mu^*+2v^*,E-v]\\
&=\scV^{-1}\cdot U^{g(K+2v^*,E-v)-g(K,E)+K(v)/2+v^2/2}\cdot [K+2v^*,E-v] \\
&=\scV^{-1}\cdot F_v^B\left([K,E]\right).
\end{split}
\]
Going from the first line to the second, we are using Equation~\eqref{eq:g-unchanged-adding-mu}.

Next, we consider the commutation of $F_v^B$ with $\scV$. We compute
\[
\begin{split}
F_v^B\left(\scV\cdot [K,E]\right)
&=U^{g(K+2\mu^*+2v^*,E-v)-g(K,E)+(K+2\mu^*)(v)/2+v\cdot v/2}\cdot[K+2\mu^*+2v^*,E-v] \\
&=U\scV\cdot  \left(U^{g(K+2v^*,E-v)-g(K,E)+K(v)/2+v\cdot v/2}\cdot [K+2v^*,E-v]\right)\\
&=U\scV\cdot F_v^B\left([K,E]\right).
\end{split}
\]
Going from the first line to the second, we use Equation~\eqref{eq:g-unchanged-adding-mu}. The proof is complete.
\end{proof}

We are now able to define the type-$D$ module $\cX(G)^{\cK}$. As a right $\ve{I}$-module, we write $\cX(G)=\bX_0\oplus \bX_1$, where each $\bX_\veps$ is a vector space over $\bF$. We define $\bX_0$ to be the $\bF$ vector space generated by a free $\bF[\scU,\scV]$-basis of $\mathbb{CF}_0(G)$ from Proposition~\ref{prop:free}. We define $\bX_1$ to be the $\bF$ vector space generated by a free $\bF[\scU,\scV,\scV^{-1}]$-basis of $\mathbb{CF}_1(G)$ from Corollary~\ref{cor:free-idempotent-1}.

We now define the structure map 
\[
\delta^1\colon \cX(G)\to \cX(G)\otimes_{\ve{I}} \cK.
\]
The construction is entirely analogous to the setting of the link surgery formula. See \cite{ZemBordered}*{Section~8.5} for a parallel construction. If $\xs$ is a basis element of $\mathbb{CF}_0(G)$ and $\d(\xs)$ has a summand of $\scU^i\scV^j\cdot \ys,$ where $\ys$ is a basis element, then we define $\delta^1(\xs)$ to have a summand of $\ys\otimes \scU^i\scV^j$. We make a similar definition for basis elements of $\mathbb{CF}_1(G)$. Next, if $F_v^A(\xs)=\scU^i\scV^j\cdot \ys$, we declare $\delta^1(\xs)$ to also have a summand of $\ys\otimes \scU^i\scV^j\sigma$. Similarly, if $F_v^B(\xs)=\scU^i\scV^j\cdot \ys$, then we declare $\delta^1(\xs)$ to have a summand of $\ys\otimes \scU^i\scV\tau$.  It is straightforward to verify that $\cX(G)^{\cK}$ satisfies the type-$D$ structure relations.
 
 We note that the underlying vector space of $\cX(G)^{\cK}$ is infinite dimensional, so completions play a subtle yet important role in the theory. We leave it to the reader to verify that the modules satisfy the \emph{Alexander module} condition described in \cite{ZemBordered}*{Section~6}.

\subsection{An example}
\label{sec:linklatticehomologyofHopflink}

In this subsection, we compute the link lattice homology of $T(2, 2)$, the positive Hopf link. 
The Hopf link in $S^{3}$ can be presented by the plumbing graph $\Gamma$ with one solid and two arrow vertices, together with two edges connecting the solid vertex with two arrow vertices respectively. The solid vertex has weight $-1$, we assign $-3$ and $-2$ to the arrow vertices.
 Denote the solid vertex by $v_0$ and the two arrow vertices by $v_1, v_2$. The link lattice complex $\mathbb{CFL}(\Gamma, V_{\uparrow})$ is generated by the elements $[K, E_1]$ and $[K, E_2]$ as a $\F[[U]]$-module where $K\in\Char(\Gamma)$ and
\[E_1=\{v_1, v_2\} \textup{ and } E_2=\{v_0, v_1, v_2\}.\]
It is not hard to see that
\[\Char(\Gamma)=\{K=[2n+1,2m_1+3,2m_2+2]\colon n,m_1,m_2\in\Z\},\]
where writing $K=[2n+1,2m_1+3,2m_2+2]$ means that
$K(v_0)=2n+1, K(v_1)=2m_1+3, K(v_2)=2m_2+2$.
We compute the differentials:
\begin{align*}
  \partial [K, E_1]&=0\\
  \partial [K, E_2]&=U^{a_{v_{0}}[K, E_2]}\otimes [K, E_1]+U^{b_{v_{0}}[K,E_2]} \otimes [K+2v_{0}^{\ast}, E_1].
\end{align*}
By direct computations, 
\begin{align*}
  g(K, E_2)&=\min\{0, n, m_1, m_2, m_1+n+1, m_2+n+1, m_1+m_2, m_1+m_2+n+2\}\\
  A_{v_{0}}(K, E_2)&=\min \{0, m_1, m_2, m_1+m_2\}\\
  B_{v_{0}}(K, E_2)&=\min\{n, m_1+n+1, m_2+n+1, m_1+m_2+n+2\}.
\end{align*}
Hence 
\begin{equation}
\partial [K,E_2]=
\begin{cases}
[K, E_1]+U^n [K+2v_{0}^{\ast}, E_1]&\text{ if } n\geq 0, m_1\geq 0, m_2\geq 0, \\
[K, E_1]+U^{n+1}[K+2v_{0}^{\ast}, E_1]& \text{ if } n\geq 0, m_1 <0, m_2\geq 0 \textup{ or } m_1\geq 0, m_2<0, \\
[K, E_1]+U^{n+2}[K+2v_{0}^{\ast}, E_1]& \text{ if } n\geq 0, m_1<0, m_2< 0, \\
U^{-n}[K, E_1]+[K+2v_{0}^{\ast}, E_1]& \text{ if } n< 0, m_1\geq 0, m_2\geq 0, \\
U^{-n-1}[K, E_1]+[K+2v_{0}^{\ast}, E_1]& \text{ if } n< 0, m_1 <0, m_2\geq 0 \textup{ or } m_1\geq 0, m_2<0, \\
U^{-n-2}[K, E_1]+[K+2v_{0}^{\ast}, E_1]& \text{ if } n< -1, m_1< 0, m_2< 0.\\
[K, E_1]+U[K+2v_{0}^{\ast}, E_1]& \text{ if } n=-1, m_1<0, m_2<0.
\end{cases}
\label{ex:toruslink}
\end{equation}
Therefore, the link lattice homology $\mathbb{HFL}(\Gamma, V_{\uparrow})$ is concentrated on the elements $[K, E_1]$ where $K=[-3, 2m_1+3, 2m_2+2]$ for $m_1<0, m_2<0$ and $K=[-1, 2m_1+3, 2m_2+2]$ otherwise. We now consider the module structure, that is the actions of $\scU_1, \scU_2, \scV_1, \scV_2$ on these elements.
By \eqref{eq:UV-def}, 
$$\scU_i\cdot [K, E_1]=U^{g(K-2\mu_i^\ast, E_1)-g(K, E_1)+1}[K-2\mu_i^\ast, E_1],$$
$$\scV_i\cdot [K, E_1]=U^{g(K+2\mu_i^\ast, E_1)-g(K, E_1)}[K+2\mu_i^\ast, E_1].$$ 
Suppose $K=[2n+1,2m_1+3,2m_2+2]$. Then 
$$g(K, E_1)=\min\{0, m_1, m_2, m_1+m_2\},$$
$$g(K\pm 2\mu_1^\ast, E_1)=\min\{0, m_1\pm 1, m_2, m_1\pm 1+m_2\}.$$
Similarly,
$$g(K\pm 2\mu_2^\ast, E_1)=\min\{0, m_1, m_2\pm 1, m_1+m_2\pm 1\}.$$
Hence, for $i=1$ or $2$, we have
\[
\scU_i [K,E_1]=
\begin{cases}
U[K-2\mu_i^\ast, E_1]& \text{ if } m_i> 0\\
[K-2\mu_i^\ast, E_1] & \text{ if } m_i\leq 0.
\end{cases}
\]
\[
\scV_i [K,E_1]=
\begin{cases}
U[K+2\mu_i^\ast, E_1]& \text{ if } m_i<0 \\
[K+2\mu_i^\ast, E_1] & \text{ if } m_i\geq  0.
\end{cases}
\]
It follows that when $m_1=m_2=0$ and $m_1=m_2=-1$ the corresponding generators $[K, E_1]$ are in the image neither of $\scU_i$ nor of $\scV_i$.  In particular, the lattice link homology of $T(2,2)$ is generated over $\F[\scU_1, \scU_2, \scV_1, \scV_2]$ by the elements
\[
\ve{X}:=[K_1,E_1]\quad \text{and} \quad \ve{Y}:=[K_2,E_1],
\]
where $K_1:=[-1,3,2]$ and $K_2:=[-3,1,0]$. 

Using Definition~\ref{def:Alexander-grading}, one easily computes
\[
A(\ve{X})=(0,0)\quad \text{and} \quad A(\ve{Y})=(-1,-1). 
\]
Additionally, since $\chi(X_{G})=2$ and $\sigma(X_{G})=-1$,  we compute using Equation~\eqref{eq:grading-gr-w-lattice-link} that
\[
\begin{split}
\gr_{w}(\ve{X})&=\frac{1}{4}(K|_{\XG}^2-3\sigma(\XG)-2\chi(\XG))=\frac{1}{4}(K|_{\XG}^2-1)=0\\
\gr_{w}(\ve{Y})&=-4+\frac{1}{4}(K|_{\XG}^2-1)=-4+2=-2.
\end{split}
\]

Finally,  if $i\neq j$, we have the following equalities as elements of homology:
\[
\scU_i\cdot \ve{X}=U[K_1-2\mu_i^\ast, E_1]=[K_1+2v_0^\ast-2\mu_i^\ast, E_1]=\scV_j\cdot \ve{Y}.
\]
 Comparing the gradings we see that these two relations (corresponding to $(i,j)=(1,2),(2,1)$)
generate all relations between the homology classes of $\ve{X}$ and $\ve{Y}$. 
In Theorem \ref{thm:toruslink}, we compute the link Floer homology of $T(n, n)$. Readers can verify that the above description coincides with the link Floer homology of $T(2, 2)$ (see Figure~\ref{fig:free-resolutions} for the link Floer complex).

Using these techniques, it is also possible to compute the $H$-function of $T(2,2)$. By using \eqref{eq:grading-gr-w-lattice-link}, one can compute $\gr_{w}([K_{m_1,m_2}, E_1])$ for $m_1,m_2\in\Z$ and
$K_{m_1,m_2}=[-3,2m_1+3,2m_2+2]$ if $m_1,m_2< 0$, $K_{m_1,m_2}=[-1,2m_1+3,2m_2+2]$ otherwise.  The function 
\[
(a,b)\mapsto -\frac12\gr_w([K_{a-1/2,b-1/2},E_1])
\]
agrees with the $H$-function for the Hopf link; see Figure~\ref{h-function}. 


\section{The equivalence with link Floer homology}\label{sec:equivalence}

In this section, we prove that link lattice homology and link Floer homology are isomorphic. The argument follows from similar logic to the case of 3-manifolds \cite{OSSLattice} \cite{ZemHFLattice}.

\begin{thm}\label{thm:lattice=HFL} Suppose that $\Gamma$ is an arrow-decorated plumbing tree  with vertex set $V_G\cup V_\uparrow$, such that $\YG$ is a rational homology 3-sphere. For each $\frs\in \Spin^c(\YG)$, there is an absolutely $(\gr_w,A)$-graded isomorphism of $A_\infty$-modules over $\ve{\scR}_\ell$: 
\[
\mathbb{CFL}(\Gamma,V_\uparrow,\frs)\simeq \ve{\cCFL}(\YG,L_\uparrow,\frs).
\]
Here, both $\mathbb{CFL}$ and $\ve{\cCFL}$ are equipped with the natural $A_\infty$-module structures which have only $m_1$ and $m_2$ non-trivial.
\end{thm}
The proof of Theorem~\ref{thm:lattice=HFL} is completed in Subsection~\ref{sub:lattice=HFL}. We now provide a sketch of the proof. We will use a relative version of the Manolescu--Ozsv\'{a}th link surgery formula \cite{MOIntegerSurgery}, which computes link Floer homology as a subcube of the full link surgery hypercube. This is stated in Theorem~\ref{thm:Manolescu-Ozsvath-subcube}. From here, we follow the approach of \cite{ZemHFLattice} and view $L_\Gamma$ as a connected sum of Hopf links. Using a tensor product formula from \cite{ZemBordered} for the link surgery complex, one obtains a combinatorial model for the link surgery complex of $L_\Gamma$.

Following the approach of Ozsv\'{a}th, Stipsicz and Szab\'{o} \cite{OSSLattice}, one may identify the lattice complex with a simplified version of the link surgery hypercube obtained by taking the homology of the link surgery complex at each vertex of the cube $\{0,1\}^\ell$, and using only the length 1 maps of the link surgery hypercube. In \cite{ZemHFLattice}, the third author shows directly using the connected sum formula for the link surgery formula that the link surgery complex for $L_\Gamma$ is chain homotopy equivalent to this simplified model of the link surgery complex. When $b_1(\YG,\Q)=0$, we show that this homotopy equivalence induces a homotopy equivalence between the link lattice complex  and the corresponding quotient complex of the link surgery complex of $L_\Gamma$. We show that the morphisms in this homotopy equivalence are well-behaved with respect to the actions of $\scU_i$ and $\scV_i$, and give a homotopy equivalence of $A_\infty$-modules.

\subsection{The link surgery complex and sublinks}

As a first step, we describe a refinement of the Manolescu and Ozsv\'{a}th link surgery formula. If $L\subset S^3$ is a link equipped with integral framing $\Lambda$, then Manolescu and Ozsv\'{a}th construct a chain complex $\cC_{\Lambda}(L)$ over $\bF[[U]]$ (defined in terms of the link Floer complex $\cCFL(S^3,L)$, equipped with additional data)
and prove that
\[
H_*(\cC_{\Lambda}(L))\iso \ve{\HF}^-(S^3_{\Lambda}(L)).
\]

There is a refinement of this result which can be used to compute link Floer homology, as follows. Suppose that $M=J\cup L\subset S^3$ is a partitioned link with $|M|=n$ and $|L|=\ell$. Equip $J$ with a framing $\Lambda$. We may extend $\Lambda$ arbitrarily to a framing $\Lambda'$ on all of $M$ to obtain a link surgery complex $\cC_{ \Lambda'}(M)$ whose homology is $\ve{\HF}^-(S^3_{\Lambda'}(M))$.

The link surgery complex $\cC_{ \Lambda'}(M)$ is a \emph{hypercube} of chain complexes, which means that it admits a natural filtration by the integral points of the cube $\{0,1\}^{n}$, where $n=|M|$.
While more details are given in \cite{MOIntegerSurgery}*{Section~5},
we give some necessary background and introduce the notation. For any $\varepsilon\in\{0,1\}^n$, we consider the
multiplicatively closed subset $S_\veps\subset \scR_n$ generated by $\scV_i$ such that $\veps(i)=1$. The complex
$\cC_\varepsilon$ is defined as the algebraic completion of the localization $S_\veps^{-1}\cdot \cCFL(M)$.
We remark that the original definition in \cite{MOIntegerSurgery} is seemingly different, though the equivalence with the above description follows from
\cite{ZemBordered}*{Lemma~5.7}.

If $\veps,\veps'\in\{0,1\}^n$, we write $\veps\le \veps'$ if $\veps_i\le \veps_i'$ for all $i$. We write $\veps<\veps'$ if $\veps\le \veps'$ and $\veps\neq \veps'$. If $\veps,\veps'\in \{0,1\}^n$ and $\veps<\veps'$, Manolescu and Ozsv\'{a}th construct a map
$D_{\veps,\veps'}\colon\cC_\veps\to\cC_{\veps'}$; see \cite{MOIntegerSurgery}*{Section 5}. The chain complex $\cC_{\Lambda'}(M)$ is the direct
sum of complexes $\cC_\veps$ and the differential is the sum of the internal differentials in $\cC_\veps$ and of the maps $D_{\veps,\veps'}$.

We now describe a relative complex $\cC_{\Lambda}(J,L)$. Note that
each axis direction in $\{0,1\}^n$ corresponds to a component of $M$. Hence, we may consider the quotient complex $\cC_{\Lambda}(J,L)$ obtained by quotienting the subcomplex of $\cC_{ \Lambda'}(M)$ consisting of those $\cC_\varepsilon$ such that $\veps(i)=1$ 
for at least one index $i$ corresponding
to $L$.
Examining Manolescu and Ozsv\'{a}th's construction, it is evident that $\cC_{\Lambda}(J, L)$ is independent of the framings of the $L$ components.

Furthermore, the module $\cC_{\Lambda}(J,L)$ has a natural action of the ring $\ve{\scR}_\ell$, corresponding to the variables for $L$. The underlying spaces $\cC_\varepsilon$ are preserved by this $\ve{\scR}_\ell$-module structure. 
It follows from \cite{ZemBordered}*{Lemma~5.9} that the hypercube maps of $\cC_{\Lambda}(J,L)$ commute with the action of $\ve{\scR}_\ell$, i.e. the action of the variables from $L$. Note that in general the differential on $\cC_{\Lambda}(J,L)$ will \emph{not} commute with the actions of the variables from $J$. The next result is important for our purposes:

\begin{thm}\label{thm:Manolescu-Ozsvath-subcube} Suppose that $M\subset S^3$ is a link which is partitioned into two sublinks $M=J\cup L$. Let $\Lambda$ be an integral framing on $J$ and write $\ell=|L|$. Then there is a homotopy equivalence of chain complexes over $\ve{\scR}_\ell$:
\[
\ve{\cCFL}(S^3_\Lambda(J),L)\simeq \cC_{\Lambda}(J,L). 
\]
Furthermore, this isomorphism respects $\Spin^c$ structures under an isomorphism
\[
\Spin^c(S^3_\Lambda(J))\iso \bH(M)/(\Span(\mu^*_1,\dots, \mu^*_\ell)+H_2(W_\Lambda(J))),
\]
where $W_{\Lambda(J)}$ is the cobordism from $S^{3}$ to the surgery on the link $J$. 
\end{thm}
The above result is a folklore result. We believe the techniques of \cite{MOIntegerSurgery} can be adapted in a straightforward manner to prove this theorem. Nonetheless, experts in the Heegaard Floer surgery formulas may recognize that although conceptually simple, a rigorous proof requires a substantial amount of bookkeeping because of the role of algebraic truncations in the surgery formula. A conceptually simple proof, avoiding truncations, can be found in \cite{ZemBorderedProperties}*{Corollary~9.2}.

\subsection{Link lattice homology and the link surgery formula}

We now describe how to recast the link lattice complex in terms of the link surgery complex. This is an adaptation of \cite{OSSLattice}*{Proposition~4.4} to our present context of links.

Construct an $|L_G|$-dimensional hypercube of chain complexes as follows: For $\veps\in \{0,1\}^{|L_G|}$, define $Z_{\veps}:=H_*(\cC_\veps)$, where $\cC_{\veps}$ is the corresponding 
submodule of $\cC_{\Lambda}(L_G,L_\uparrow)$. If $\veps<\veps'$, construct a hypercube map $\delta_{\veps,\veps'}\colon Z_\veps\to Z_{\veps'}$ via the formula:
\[
\delta_{\veps,\veps'}:=\begin{cases}(D_{\veps,\veps'})_*& \text{ if } |\veps'-\veps|_{L^1}=1\\
0& \text{ otherwise}.
\end{cases}
\]
Write $\cZ=(Z_{\veps}, \delta_{\veps,\veps'})_{\veps\in \{0,1\}^{|L_G|}}$. Clearly $\cZ$ is a hypercube of chain complexes over $\ve{\scR}_{\ell}$.

Compare the following to \cite{OSSLattice}*{Proposition~4.4}:

\begin{prop}\label{prop:chain-isomorphism=H-lattice-CFL} Let $\Gamma$ be an arrow-decorated plumbing tree. The hypercube $\cZ=(Z_{\veps},\delta_{\veps,\veps'})_{\veps\in \{0,1\}^{|L_G|}}$ is isomorphic to the link lattice complex $\mathbb{CFL}(\Gamma,V_\uparrow)$. 
\end{prop}

Before proving Proposition~\ref{prop:chain-isomorphism=H-lattice-CFL}, we prove a  technical lemma which is helpful for relating the  $\ve{\scR}_\ell$-actions on $\mathbb{CFL}(\Gamma,V_\uparrow)$ and $\cZ$.
We note that the following lemma is essentially implicit in the definition of lattice homology and also Ozsv\'{a}th, Stipsicz, and Szab\'{o}'s construction of the spectral sequence (cf. \cite{OSSLattice}*{Proposition~4.4}), though we have been unable to find an exact reference which is suitable for our purposes. If $L\subset S^3$ and $\ve{s}\in \bH(L)$, write 
\[
d(L,\ve{s})=\max\left\{ \gr_w(x): x\in H_*\frA^-(L,\ve{s}), x \text{ is $U$-nontorsion}\right\}.
\] Here, $\gr_w$ denotes the internal Maslov grading from link Floer homology, and $\frA^{-}(L, \ve{s})\subset\cCFL(S^{3}, L)$ is a subcomplex corresponding to the Alexander grading $\ve{s}$. 

\begin{rem}
  For an oriented link $L\subset S^{3}$ in the 3-sphere, Gorsky and N\'{e}methi \cite{GorskyNemethiLattice} defined a link invariant, the so-called \emph{H-function},  by declaring $-2H_{L}(\ve{s})=d(L, \ve{s})$. The $H$-function is a generalization of the $h_k$-invariant considered by Rasmussen \cite{RasmussenKnots}*{Section 7}. For algebraic links, it can be related to the semigroup counting function \cite[Section 3.5]{GorskyNemethiLattice}. We later generalized the $H$-function for links in rational homology spheres, see Section \ref{sub:Alexander_and_Floer}.

\end{rem}

\begin{lem}\label{lem:relate-gradings} Let $G$ be a  forest of plumbing trees. If $E\subset V_G$, write $L_E\subset S^3$ for the sublink of $L_G$ containing exactly the components corresponding to vertices of $E$. Let $\phi_E\colon \Char(\XG)\to \bH(L_E)$ be the composition of the restriction map from $\Char(\XG)$ to $\Char(X_E)$ and the isomorphism from $\Char(X_E)$ to $\bH(L_E)$ in Equation~\eqref{eq:Spin-c-lattice-iso}. Then
\[
2g(K,E)=d(L_E,\phi_{E}(K)).
\]

\end{lem}
\begin{proof}
Note that $g(K,E)$ (computed in $\XG$) is the same as $g(K|_{X_E},E)$ (computed in $X_E$). Hence, we may assume without loss of generality that $E=V_G$.

 Next, we observe that the framings on the components of $L_G$ play no role in the statement. Indeed, if $K\in \Char(\XG)$ and $\Phi_G(K)=(s_1,\dots, s_n)$, then Equation~\eqref{eq:Spin-c-lattice-iso} implies that
\begin{equation}
2g(K,V_G)=\min_{I\subset V_G}\left(\sum_{i\in I} 2s_i-v_{G-I}\cdot v_I\right) \label{eq:g-interms-H-lattice}
\end{equation}
where $v_{G-I}$ is the sum of $v_i$ for $i\in V_G\setminus I$ and $v_I$ is similar. 
Our proof will be by induction on the number of vertices. If $\ve{s}\in \bH(L_G)$, we will write $2g(L_G,\ve{s})$ for the quantity on the right-hand side of Equation~\eqref{eq:g-interms-H-lattice}. Similarly, if $I\subset V_G$ we write 
\begin{equation}
\label{def:f}
2f(\ve{s}, I)=\sum_{i\in I} 2 s_i -v_{G-I}\cdot v_I.
\end{equation} 
 We claim that $2g(L_G,\ve{s})=d(L_G, \ve{s})$.

We begin with the case that $L_G$ is an $n$-component unlink $\bU_n$. In this case, the homology $\cHFL(\bU_n)\iso \oplus_{\ve{s}\in \bH(L_G)} \frA^-(\bU_n, \ve{s})$ is well known to be $
\bF[\scU_1,\scV_1,\dots, \scU_n,\scV_n]/(\scU_i\scV_i-\scU_j\scV_j, i,j\in \{1,\dots n\}),$ with the class of $1$ having $(\gr_w,\gr_z)$-bigrading $(0,0)$. The plumbing diagram of the unlink $\bU_n$ consists of $n$ disjoint arrow vertices. After we assign the weight $-1$ to each arrow vertex,  it is straightforward to see that the link lattice homology is also isomorphic to $\bF[\scU_1,\scV_1,\dots, \scU_n,\scV_n]/(\scU_i\scV_i-\scU_j\scV_j, i,j\in \{1,\dots n\})$ as an $\bF[U]$-module with  the generator as $[K, E]$ where $K=[2s_1+1, \cdots, 2s_n+1]$ with $s_i\in \mathbb{Z}$ and $E$ consists of all arrow vertices equipped with the same gradings $(\gr_w, \gr_z)=(0, 0)$. That is, 
$$d(\bU_n, \ve{s})=2g(\bU_n, \ve{s})=\sum_{i=1}^{n} 2\min\lbrace 0, s_i\rbrace.$$

We now assume that claim is true for some forest $G$ of plumbing trees. We will prove the claim also holds for $L_G\# H$, where $H$ is the positive Hopf link and the connected sum operation is between one knot component in $L_G$ and one knot component of $H$. We recall that the complex of the positive Hopf link takes the following form:
\begin{equation}
\cCFL(H)\iso \begin{tikzcd}[labels=description,row sep=.8cm, column sep=1.3cm] \ve{a} \ar[d, "\scV_{n}"]\ar[r, "\scU_{n+1}"]& \ve{b}\\
\ve{c}& \ve{d} \ar[u, "\scU_{n}"] \ar[l, "\scV_{n+1}"]
\end{tikzcd}.
\label{eq:positive-Hopf-link}
\end{equation}
The Alexander bigrading of $\ve{a},\ve{b},\ve{c},\ve{d}$ are $(\tfrac{1}{2},-\tfrac{1}{2})$, $(\tfrac{1}{2},\tfrac{1}{2})$, $(-\tfrac{1}{2},-\tfrac{1}{2})$ and $(-\tfrac{1}{2},\tfrac{1}{2})$, respectively. The $(\gr_w,\gr_z)$-bigradings are $(-1,-1)$, $(0, -2)$, $(-2,0)$ and $(-1,-1)$, respectively.

We assume that $L_n\subset H$ is the component where the connected sum is taken, and $L_{n+1}\subset H$ for the remaining component.

Topologically, the link $L_G$ is obtained by taking an unlink, and iteratively taking the connected sum with Hopf links. We recall that that complex for an $\ell$-component unlink is obtained by tensoring the $\ell-1$ Koszul complexes
\[
C_i=\begin{tikzcd}[column sep=2cm] \xs \ar[r, "\scU_1\scV_1+\scU_i\scV_i"]& \ys
\end{tikzcd}
\]
for $i=2,\dots, \ell$. The $\ell=2$ case of this computation is verified using a genus 0 Heegaard diagram for a 2-component unlink, and the general case is proven by iteratively using the connected sum formula \cite{OSKnots}*{Section~7}. In Equation~\eqref{eq:positive-Hopf-link}, we observe that the Hopf link also has a similar 2-step filtration.

In particular, since a general $L_G$ is obtained by tensoring an unlink with a  collection of Hopf links, we may write $\cCFL(L_G)$ as 
\[
F_s\to F_{s-1}\to \cdots \to F_0.
\]
Call the index $i$ for $F_i$ the \emph{Hopf grading},
where each $F_i$ is a free $\scR_n$-module. By \cite{OSSLattice}*{Lemma~4.2} (cf. \cite{OSPlumbed}*{Lemma~2.6}), $L_G$ is an L-space link, so it follows that $\cHFL(L_G)$ is supported in just one Hopf grading.  We observe that the homology must be supported in $F_0$, since the map $F_1\to F_0$ is not surjective. To see that the map is not surjective, define a map from $F_0$ to $\bF$ which sends all $\scU_i$ and $\scV_i$ to $1$ and sends each basis element of $F_0$ to 1. Then the composition $F_1\to F_0\to \bF$ is zero, whereas the map $F_0\to \bF$ is non-zero.


Applying the above argument to $\cCFL(L_G\# H)$, we see that any homogeneously graded  cycle in $\cCFL(H)\otimes_{\bF[\scU_{n},\scV_{n}]}\cCFL(L_G)$ which represents an $\bF[U]$-non-torsion element of homology  may be written as a sum of an odd number of terms of the form $\a\cdot \ve{b}\otimes \ve{z}$ or $\b \cdot \ve{c}\otimes \ve{z}$, where  $\ve{z}\in \cCFL(L_G)$ is an $\bF[U]$-non-torsion cycle, and $\a,\b \in \bF[\scU_{n+1},\scV_{n+1}]$. 

Consider $\ve{s}=(s_1,\dots, s_{n+1})\in \bH(L_G\# H)$, where $s_{n+1}\in \tfrac{1}{2}+\Z$ corresponds to the component $L_{n+1}\subset H$. We break the proof into two cases: $s_{n+1}>0$, and $s_{n+1}<0$.

We consider first the case that $s_{n+1}>0$. Recall that $L_G\# H$ is an L-space link \cite{OSSLattice}*{Lemma~4.2}. In this case, $H_*\frA^-(L_G\# H,\ve{s})$ is generated over $\bF[U]$ by the elements $\scV_{n+1}^{s_{n+1}-1/2}\cdot \ve{b}\otimes \ve{z}$ and $\scV_{n+1}^{s_{n+1}+1/2}\cdot \ve{c}\otimes\zs$, where $\zs$ is $\bF[U]$-non-torsion. Write $\ve{s}'=(s_1,\dots, s_n)$. We obtain from the above argument that
\begin{equation}
\begin{split}
d(L_G\#H,\ve{s})=&\max\{ \gr_w(\ve{b})+d(L_G, \ve{s}'-\tfrac{1}{2}e_n), \gr_w(\ve{c})+d(L_G, \ve{s}'+\tfrac{1}{2}e_n) \}\\
=&\max\{ d(L_G, \ve{s}'-\tfrac{1}{2}e_n),-2+ d(L_G, \ve{s}'+\tfrac{1}{2}e_n)\}\\
=&d(L_G, \ve{s}'-\tfrac{1}{2}e_n).
\end{split}
\label{eq:d-invariant-comp-induction}
\end{equation}
Here the last equality follows from the facts that $d(L_G, \ve{s}'-\tfrac{1}{2}e_n)=-2H_{L_{G}}(\ve{s}'-\tfrac{1}{2}e_n)$, and $H_{L_{G}}(\ve{s}'-\tfrac{1}{2}e_n)\leq H_{L_{G}}(\ve{s}'+\tfrac{1}{2}e_n)+1$ \cite[Proposition 3.4]{BorodzikGorskyImmersed}. 
Let $I'\subset V_G$ be any subset.  We may view $I'$ also as subset of $V_G\cup \{v_{n+1}\}$. We compute easily that
\[
2f(\ve{s}, I')= 2f(\ve{s}'-\tfrac{1}{2}e_n,I'), \quad \text{and}
\]
\[
2f(\ve{s}, I'\cup \{v_{n+1}\})=\begin{cases} 2f(\ve{s}'-\tfrac{1}{2}e_n,I')+2s_{n+1}+1& \text{ if } v_n\in I'\\
2f(\ve{s}'-\tfrac{1}{2}e_n,I')+2s_{n+1}-1& \text{ if } v_n\not \in I',
 \end{cases}
\]
Since $s_{n+1}>0$, we observe from these equations that
\[
f(\ve{s},I')=f(\ve{s}'-\tfrac{1}{2}e_n,I')\le f(\ve{s},I'\cup \{v_{n+1}\}).
\]
for all $I'\subset V_G$ and $\ve{s}\in \bH(L_G)$. It follows easily that
\[
g(L_G\#H,\ve{s})=g(L_G,\ve{s}'-\tfrac{1}{2}e_n).
\]
By induction $2g(L_G,\ve{s}'-\tfrac{1}{2}e_n)=d(L_G,\ve{s}'-\tfrac{1}{2}e_n),$ so from Equation~\eqref{eq:d-invariant-comp-induction} we obtain $2g(L_G\# H,\ve{s})=d(L_G\# H, \ve{s})$ when $s_{n+1}>0$.

  We now consider the case $s_{n+1}<0$. In this case, $H_*\frA^-(L_G\# H,\ve{s})$ is generated by elements of the form $\scU_{n+1}^{-s_{n+1}+1/2}\cdot \ve{b}\otimes \zs$ and $\scU_{n+1}^{-s_{n+1}-1/2}\cdot \ve{c}\otimes \zs$. A similar argument to $s_{n+1}>0$ case yields that
\begin{equation}
d(L_G\# H,\ve{s})= d(\ve{s}'+\tfrac{1}{2}e_n)+2s_{n+1}-1. \label{eq:d-s_n+1<0}
\end{equation}
One computes directly that
\begin{equation}
\label{eq:s+}
2f(\ve{s}, I')= \begin{cases}
2f(\ve{s}'+\tfrac{1}{2}e_n,I')-2& \text{ if } v_n\in I'\\
2f(\ve{s}'+\tfrac{1}{2}e_n,I')& \text{ if } v_n\not \in I',
\end{cases}\qquad
\text{and}
\end{equation}
\[
2f(\ve{s}, I'\cup \{v_{n+1}\})=
2f(\ve{s}'+\tfrac{1}{2}e_n,I')+2s_{n+1}-1,
\]
for all $I'\subset V_{G}$.
Since $s_{n+1}<0$, we have
\[
2f(\ve{s},I'\cup \{v_{n+1}\})=2 f(\ve{s}'+\tfrac{1}{2} e_n,I')+2s_{n+1}-1\le 2f(\ve{s},I'),
\]
where the last inequality comes from \eqref{eq:s+}, so $2g(L_G\# H,\ve{s})=2 g(L_G,\ve{s}'+\tfrac{1}{2} e_n)+2s_{n+1}-1$. Combined with Equation~\eqref{eq:d-s_n+1<0}, we conclude  that $2g(L_G\# H,\ve{s})=d(L_G\# H,\ve{s})$, completing the proof of Lemma~\ref{lem:relate-gradings}.
\end{proof}

\begin{proof}[Proof of Proposition~\ref{prop:chain-isomorphism=H-lattice-CFL}] Note that \cite{OSSLattice}*{Proposition~4.4} proves the identification on the level of chain complexes of $\bF[[U]]$-modules. It suffices to show that the decomposition respects the refined actions of the ring $\ve{\scR}_\ell$. We recall the basics of their isomorphism. The lattice $\bH(L_\Gamma)$ represents the set of Alexander gradings supported by the link Floer complex $\cCFL(L_\Gamma)$. If $L\subset L_\Gamma$, and $\veps\in \{0,1\}^n$ is the indicator function for the components of $L$, then we may write $\frA^-_{\veps}(L_\Gamma,\ve{s})$ for the subcomplex of $S_\veps^{-1}\cdot \cCFL(L_\Gamma)$ in Alexander grading $\ve{s}$. According to \cite{OSSLattice}*{Lemma~4.2}, there is an isomorphism
\[
H_*(\frA^-_{\veps}(L_\Gamma,\ve{s}))\iso \bF[[U]].
\]
Additionally, there is an isomorphism $\bH(L_\Gamma)\to \Spin^c(X_\Gamma)$ (stated in Equation~\eqref{eq:Spin-c-lattice-iso}). This gives an isomorphism between $\cZ$ and $\mathbb{CFL}(\Gamma,V_\uparrow)$ as $\bF[[U]]$-modules. The Maslov grading of the generator of $H_*(\frA^-_{\veps}(L_\Gamma,\ve{s}))$ is $d(L_{\Gamma}, \ve{s})$ and the maps $D_{\epsilon, \epsilon'}$ for $|\epsilon'-\epsilon|=1$ are determined by the Maslov grading of the generators of the domain and target. Following Lemma \ref{lem:relate-gradings} and the same argument in 
\cite{OSSLattice}*{Proposition~4.4}, the differentials also coincide. It remains to show that the isomorphism respects the $\ve{\scR}_\ell$-module structure.

 To disambiguate the actions, let us write $
 \bF[[\scU_1,\scV_1,\dots, \scU_\ell,\scV_\ell]]
 $ for the action we have described on $\mathbb{CFL}(\Gamma,V_\uparrow)$, and let us write
 $\bF[[\cU_1,\cV_1,\dots, \cU_{\ell},\cV_\ell]]$ for the action induced by the identification $\mathbb{CFL}(\Gamma,V_\uparrow)\iso \cZ $. As a first step, note that by definition $\cU_i$ changes the Alexander grading $\ve{s}\in \bH(L_\Gamma)$ by $-e_i\in \Z^n$ (where $e_i$ is the unit vector with $i$-th component $1$, and other components $0$). It follows from Lemma~\ref{lem:relate-gradings}, that $\cU_i$ has the same $\gr_w$-grading as $\scU_i$, and similarly $\cV_i$ has the same $\gr_w$-gradings as $\scV_i$.
  Clearly, if $[K,E]$ is a generator, then $\scU_i\cdot [K,E]=U^\epsilon [K-2\mu^*_i,E]$ and $\cU_i\cdot [K,E]=U^{\epsilon'} [K-2\mu^*_i,E]$ for some $\epsilon,\epsilon'\in \{0,1\}$. 
  Since $\cU_i$ and $\scU_i$ have the same $\gr_w$-grading as endomorphisms, we must have that $\epsilon=\epsilon'$, so that $\scU_i$ and $\cU_i$ have the same action. The same argument implies $\scV_i=\cV_i$.
\end{proof}

\subsection{Proof of Theorem~\ref{thm:lattice=HFL}}\label{sub:lattice=HFL}

In this subsection we give the proof of Theorem~\ref{thm:lattice=HFL}, though we delay the discussion about absolute gradings until Subsection~\ref{sec:absolute-gradings}. The main steps of the proof follow from \cite{ZemHFLattice}*{Section~5.2}. We provide a summary and highlight the necessary changes for our present setting.

By Theorem~\ref{thm:Manolescu-Ozsvath-subcube}, we have a homotopy equivalence of chain complexes over $\ve{\cR}_{\ell}$: 
\[
\ve{\cCFL}(\YG,L_\uparrow)\simeq \cC_{\Lambda}(L_G,L_\uparrow).
\] 
In particular, the above chain homotopy equivalence may be viewed as a homotopy equivalence of $A_\infty$-modules over $\ve{\cR}_{\ell}$, where each module has $m_j=0$ for $j>2$. 
Above, we also constructed a hypercube $\cZ=(Z_{\veps},\delta_{\veps,\veps'})$ of $\ve{\scR}_\ell$-modules by taking the homology of $\cC_{\Lambda}(L_G,L_\uparrow)$ at each cube point, and using only the length 1 maps from $\cC_{\Lambda}(L_G,L_\uparrow)$. By Proposition~\ref{prop:chain-isomorphism=H-lattice-CFL}, we have a chain isomorphism
\[
\cZ\simeq \mathbb{CFL}(\Gamma,V_\uparrow).
\]
Hence, it suffices to construct an $A_\infty$-homotopy equivalence
\begin{equation}
\cZ\simeq \cC_{\Lambda}(L_G,L_\uparrow).\label{eq:H-C_lambda-equiv}
\end{equation}
This homotopy equivalence follows from the same logic as \cite{ZemHFLattice}*{Section~6.2}, which we summarize for the benefit of the reader.

We will write $\cC=(\cC_{\veps}, \Phi_{\veps,\veps'})$ for $\cC_{\Lambda}(L_G,L_\uparrow)$. The underlying complex $\cC_\veps$ at each vertex of $\cC_{\Lambda}(L_G,L_\uparrow)$ is obtained from $\cCFL(L_\Gamma)$ by localizing at the variables $\scV_i$ such that $\veps(i)=1$, and then taking an appropriate completion.  
The complex $\cCFL(L_\Gamma)$ is a tensor product of Hopf link complexes.
 The Hopf link has a 2-step filtration (see Equation~\eqref{eq:positive-Hopf-link}), and hence $\cCFL(L_\Gamma)$ has a description as
 \begin{equation}
\cCFL(S^3,L_G\cup L_\uparrow)\iso 
\left(\begin{tikzcd}
\cF^m \ar[r]& \cF^{m-1}\ar[r]&\cdots \ar[r]& \cF^0.
\end{tikzcd}\right)
 \label{eq:tree-Hopf-links}
 \end{equation}
 where each $\cF^i$ is a free $\scR_n$-module, $n=|V_\Gamma|$ and $m$ is the number of edges in $\Gamma$ (i.e. Hopf link components).
 Each $\cC_{\veps}$ has a similar filtration, which we denote by $\cF_\veps^i$. We call the superscript $i$ in $\cF^i$ the \emph{Hopf grading}. 

Since $L_\Gamma$ is an L-space link, each $Z_\veps$ is a direct product of copies of $\bF[[U]]$. Following \cite{ZemHFLattice}*{Proposition~6.2}, there is a natural way to construct a homotopy equivalence between each $\cC_{\veps}$ and $Z_{\veps}$, for each $\veps$. This is because the homology of $\cC_{\veps}$ is supported in $\cF^0_\veps$ so the projection map of $\cF^0_\veps$ onto $Z_{\veps}$ gives a quasi-isomorphism. Since in each Alexander grading (i.e. each $\ve{s}\in \bH(L_\Gamma)$) the homology of $Z_{\veps}$ is $\bF[[U]]$ (in particular, a projective $\bF[[U]]$ module), it is straightforward to construct a splitting over $\bF[[U]]$ of the sequence in Equation~\eqref{eq:tree-Hopf-links} in each Alexander grading. This gives us maps
\[
\pi_{\veps}\colon \cC_{\veps}\to Z_{\veps},\quad i_\veps\colon Z_{\veps}\to \cC_{\veps}\quad \text{and} \quad h_{\veps}\colon \cC_{\veps}\to \cC_{\veps},
\]
which satisfy 
\[
\pi_{\veps}\circ i_{\veps}=\id,\quad  i_{\veps}\circ \pi_{\veps}=\id+[\d, h_{\veps}],\quad h_{\veps}\circ h_{\veps}=0, \quad \pi_{\veps}\circ h_{\veps}=0,\quad h_{\veps}\circ i_{\veps}=0.
\]
\begin{rem}\label{rem:pi_e_is_equivariant}
Note that these maps are usually only $\bF[[U]]$-equivariant; but not necessarily $\ve{\scR}_\ell$-equivariant. The only exception is $\pi_{\veps}$,  because it is the canonical projection of $\cF^0_{\veps}$ to $\cF^0_{\veps}/\im \cF^1\iso H_*(Z_{\veps})$.
\end{rem}

The homological perturbation lemma for hypercubes (see \cite{HHSZDuals}*{Lemma~2.10}) induces hypercube structure maps $d_{\veps,\veps'}$ on $\bigoplus_{\veps\in\{0,1\}^n} Z_{\veps}$, which is homotopy equivalent to the hypercube $\cC_{\Lambda}(L_G,L_\uparrow)$. Let us write $\cW$ for the hypercube $(Z_{\veps},d_{\veps,\veps'})$. 
The structure maps $d_{\veps,\veps'}$ are given by the concrete formula
\begin{equation}
d_{\veps,\veps'}:=\sum_{\veps=\veps_1<\cdots <\veps_n=\veps'} \pi_{\veps_n}\circ \Phi_{\veps_{n-1},\veps_n}\circ h_{\veps_{n-1}}\circ \cdots \circ h_{\veps_2} \circ \Phi_{\veps_1,\veps_2}\circ i_{\veps_1}.
\label{eq:perturbed-hypercube-differential}
\end{equation}
There are also homotopy equivalences $\Pi\colon \cC\to \cW$, $I\colon \cW\to \cC$ and $H\colon \cC\to \cC$, given by similar formulas, which satisfy 
\[
\Pi\circ I=\id, \quad I\circ \Pi=\id+\d(H),\quad H\circ H=0,\quad \Pi\circ H=0,\quad \text{and}\quad H\circ I=0.
\]
 We will also need to understand the map $\Pi$, which is given by
\begin{equation}
\Pi_{\veps,\veps'}=\sum_{\veps=\veps_1<\dots<\veps_n=\veps'} \pi_{\veps_n}\circ \Phi_{\veps_{n-1},\veps_n}\circ h_{\veps_{n-1}}\circ \cdots \circ \Phi_{\veps_1,\veps_2}\circ h_{\veps_1}.
\label{eq:Pi-map-hypercubes}
\end{equation}

A natural strategy is to show that $\cW=\cZ$. Note that the underlying groups are identical, so it suffices to understand the structure maps. We observe that $\delta_{\veps,\veps'}=d_{\veps,\veps'}$ whenever $|\veps-\veps'|_{L^1}=1$. In general, it is not the case that $d_{\veps,\veps'}=\delta_{\veps,\veps'}$ when $|\veps-\veps'|_{L^1}>1$. Instead the main argument of \cite{ZemHFLattice} is to show that the analog of $\cC_{\Lambda}(L_G,L_\uparrow)$ is homotopy equivalent to a hypercube which has the same underlying internal chain complexes as $\cC_{\Lambda}(L_G,L_\uparrow)$ and for which the induced hypercube structure on the analog of $\cW$ coincides with that of $\cZ$. The argument in our present setting is essentially identical to the one in \cite{ZemHFLattice}*{Section~6.2}. In fact, we may take the modified hypercube structure on $\cC_{\Lambda}(L_G,L_\uparrow)$ to be the restriction of the one constructed in \cite{ZemHFLattice}, viewing $\cC_{\Lambda}(L_G,L_\uparrow)$ as a subcube of the full link surgery hypercube for $L_\Gamma$. We write $\cC'_{\Lambda}(L_G,L_\uparrow)$ for the resulting hypercube. We will also write 
\[
\cC'=(\cC_{\veps}, \Phi_{\veps,\veps'}')
\]
for $\cC'_{\Lambda}(L_G, L_\uparrow)$, and $\cW'$ for the hypercube constructed via homological perturbation. There are maps
\[
\Pi'\colon \cC'\to \cW',\quad I'\colon \cW'\to \cC' \quad \text{and} \quad H'\colon \cC'\to \cC'
\]
constructed similarly to the maps $\Pi$, $I$ and $H$, using formulas as in~\eqref{eq:Pi-map-hypercubes}.

 Concretely, $\cC_{\Lambda}(L_G,L_\uparrow)$ is constructed by tensoring the link surgery complex of Hopf links using the tensor product formula from \cite{ZemHFLattice}*{Equation~3.2}. The complex $\cC_{\Lambda}'(L_G,L_\uparrow)$ is constructed by tensoring the link surgery complexes of Hopf links together using an algebraically simplified model of the tensor product, where several terms have been deleted. (This simplified model appears in \cite{ZemHFLattice}*{Theorem~3.4}). The key property of the hypercube maps appearing in $\cC_{\Lambda}'(L_G,L_\uparrow)$ are the maps $\Phi_{\veps,\veps'}'$ are non-increasing in the Hopf grading from~\eqref{eq:tree-Hopf-links} only when $|\veps-\veps'|_{L^1}\le 1$. Furthermore, when $|\veps-\veps'|_{L^1}=1$, the map $\Phi_{\veps,\veps'}'$ preserves the Hopf grading. Since $h_{\veps}$ strictly increases the Hopf grading and $\pi_{\veps}$ is non-vanishing only on the lowest Hopf grading, the composition in~\eqref{eq:perturbed-hypercube-differential} will only be non-trivial when $|\veps-\veps'|_{L^1}=1$. Hence
 \[
 \cW'=\cZ.
 \]

 In particular, composing these homotopy equivalences, we obtain a homotopy equivalence of chain complexes
 \begin{equation}
\ve{\cCFL}(\YG, L_\uparrow)\simeq \cZ. \label{eq:homotopy-equivalence-refining}
\end{equation}

 It remains to show that the homotopy equivalence in~\eqref{eq:homotopy-equivalence-refining} may be extended to a homotopy equivalence of $A_\infty$-modules over $\ve{\cR}_\ell$. There are two subclaims:
\begin{enumerate}[label=($R$-\arabic*), ref=$R$-\arabic*]
\item\label{equivariance-1}  The homotopy equivalence between $\cC_\Lambda(L_G,L_\uparrow)$ and $\cC'_{\Lambda}(L_G,L_\uparrow)$ may be taken to be $\ve{\cR}_{\ell}$-equivariant.
\item\label{equivariance-2}  The homotopy equivalence $\cC'_{\Lambda}(L_G,L_\uparrow)\simeq \mathbb{CFL}(\Gamma,V_\uparrow)$ (from the homological perturbation lemma of hypercubes) extends to a homotopy equivalence of $A_\infty$-modules over $\ve{\cR}_{\ell}$. 
\end{enumerate} 
We address~\eqref{equivariance-1} first. In \cite{ZemHFLattice}*{Corollary~4.8}, it is shown that the simplified connected sum formula yielding $\cC_{\Lambda}'(L_G,L_\uparrow)$ is valid as long as in forming $L_\Gamma$ by an iterated connected sum, we never take the connected sum of two knot components which are both homologically essential after we surger on the other components of $L_G$. In the case that $b_1(\YG)=0$, we always avoid this configuration (cf. \cite{ZemHFLattice}*{Lemma~6.5}). The homotopy equivalence between $\cC_{\Lambda}(L_G,L_\uparrow)$ and $\cC_{\Lambda}'(L_G,L_\uparrow)$ is concrete and obtained by relating the connected sum formula in \cite{ZemHFLattice}*{Equations~(3.2)} and the simplified connected sum formula in \cite{ZemHFLattice}*{Theorem~3.4}. As described in \cite{ZemHFLattice}*{Corollary~4.8} relating these two models amounts to constructing a null-homotopy of an algebraically defined homology action on the link surgery formula. It is observed \cite{ZemHFLattice}*{Remark~4.3} that this null-homotopy may be taken to be $\ve{\cR}_{\ell}$-equivariant.

We now address~\eqref{equivariance-2}. Proposition~\ref{prop:chain-isomorphism=H-lattice-CFL} shows that $\cZ$ is chain isomorphic as an $\ve{\cR}_{\ell}$-module to $\mathbb{CFL}(\Gamma,V_\uparrow)$.  It is sufficient to show that the homotopy equivalence $\cC'\simeq \cZ$ of chain complexes is in fact a homotopy equivalence of $A_\infty$-modules over $\ve{\cR}_{\ell}$. To see this, it is in fact sufficient to show that the map $\Pi'\colon\cC'\to \cZ$ defined as in \eqref{eq:Pi-map-hypercubes},  commutes with the $\ve{\cR}_\ell$ action. This is sufficient because in the category of $A_\infty$-modules, quasi-isomorphisms are always invertible as $A_\infty$-morphisms. To establish the $\ve{\cR}_\ell$-equivariance, we observe that the projection maps $\pi_{\veps}$ are themselves equivariant since they are merely quotient maps; compare Remark~\ref{rem:pi_e_is_equivariant}. 
Next, we examine the expression for $\Pi_{\veps,\veps'}'$, as in~\eqref{eq:Pi-map-hypercubes}. By considering the Hopf grading similarly to how we did with $d_{\veps,\veps'}$ above, we observe that the maps $\Pi_{\veps,\veps'}'$ are non-trivial only when $\veps=\veps'$. In this case, the only non-vanishing contribution is from $\pi_{\veps}$, so $\ve{\cR}_\ell$-equivariance is established, and the proof is complete.
\qed

\begin{rem} By using the homological perturbation lemma, stated in Lemma~\ref{lem:homological-perturbation-modules}, a concrete homotopy equivalence of $A_\infty$-modules between $\cC'_{\Lambda}(L_G,L_\uparrow)$ and $\cZ$ may be constructed. Indeed the hypercube maps $\Pi'$, $I'$ and $H'$, constructed via the homological perturbation lemma for hypercubes, also satisfy the assumptions of the homological perturbation lemma for $A_\infty$-modules. These maps then induce an $A_\infty$-module structure on $\cZ$ over the ring $\ve{\cR}_{\ell}$, which is homotopy equivalent to $\ve{\cCFL}(Y_G,L_\uparrow)$. By considering the Hopf grading, similarly to the above, we obtain only a non-trivial $m_1$ and $m_2$ on $\cZ$. We observe also that the morphisms $\Pi'$, $H'$ and $I'$ extend to $A_\infty$-morphisms $\Pi_j'$, $H_j'$ and $I_j'$. Hopf grading considerations show that $\Pi_j'=0$ unless $j=1$, in which case the only contribution is from $\pi_{\veps}$. We observe that $I_j'$ may be non-trivial for $1\le j\le |L_\Gamma|$.
\end{rem}

\subsection{Absolute gradings}
\label{sec:absolute-gradings}

In this section, we prove the subclaim of Theorem~\ref{thm:lattice=HFL} concerning the absolute Maslov and Alexander gradings. Compare \cite{OSSLattice}*{Proposition~4.8}.

We begin by stating formulas for the absolute gradings on the link surgery complex, and its subcube refinement for sublinks. Although likely known to experts, these formulas have not appeared in the literature except in special cases.  For example, in the case of knots, the result is due to Ozsv\'{a}th and Szab\'{o} \cite{OSKnots}*{Section~4}. Detailed proofs of the absolute grading formula can be found in \cite{ZemBorderedProperties}*{Section 10}.

If $L\subset S^3$ is a link with framing $\Lambda$, let $W_{\Lambda}(L)$ denote the standard 2-handle cobordism from $S^3$ to $S^3_{\Lambda}(L)$. If $\veps\in \{0,1\}^n$, where $n=|L|$, and $\ve{s}\in \bH(L)$,  write $\cC_{\veps}(\ve{s})\subset \cC_{\Lambda}(L)$ for the subspace of $\cC_{\veps}$ which lies in internal Alexander grading $\ve{s}$. 
 Finally, if $\ve{s}\in \bH(L)$, write $\frz_{\ve{s}}\in \Spin^c(W_{\Lambda}(L))$ for the $\Spin^c$ structure which satisfies
 \begin{equation}
 \frac{\langle c_1(\frz_{\ve{s}}),\Sigma_i\rangle -\Sigma\cdot \Sigma_i}{2}=-s_i\label{eq:grading-z-s}
 \end{equation}
 for all $i\in \{1,\dots, n\}$. In the above, $\Sigma_i$ denotes the core of the 2-handle attached along component $K_i\subset L$, and $\Sigma$ denotes the sum of all $\Sigma_i$.

 \begin{lem}[\cite{ZemBorderedProperties}*{Theorem 10.2}]
\label{lem:Maslov-gradings-absolute}
 Suppose that $L\subset S^3$ is a link with framing $\Lambda$, and that $b_1(S^3_{\Lambda}(L))=0$. The homotopy equivalence $\ve{\CF}^-(S^3_{\Lambda}(L))\simeq  \cC_{\Lambda}(L)$ is absolutely graded if we equip $\cC_\veps(\ve{s})\subset \cC_{\Lambda}(L)$ with the Maslov grading
 \[
 \tilde{\gr}:=\gr_{w}+\frac{c_1(\frz_{\ve{s}})^2-2\chi(W_{\Lambda}(L))-3\sigma(W_{\Lambda}(L))}{4}+|L|-|\veps|,
 \]
 where $\gr_w$ is the Maslov grading from $\cCFL(L)$.
\end{lem}

In \cite{OSIntegerSurgeries}, Ozsv\'{a}th and Szab\'{o} prove this formula in the context of
knot surgery formula. Their main tool is computing the grading change
of a surgery cobordism map, which they denote by $f_3^+$.  A similar strategy to Ozsv\'{a}th and Szab\'{o}'s proof of the grading formula is likely possible in the context of the link surgery formula. Nonetheless, algebraic truncations make writing a simple proof challenging.

There is also a relative version of the statement. Suppose that we have a partitioned link $J\sqcup L\subset S^3$ and that $J$ is equipped with an integral framing $\Lambda$.
 If $K_i$ is a link component of $J\sqcup L$, there is a class $\Sigma_i\in H_2(W_{\Lambda}(J),\d W_{\Lambda}(J))$. If $K_i$ is in $J$, then $\Sigma_i$ is the core of the corresponding 2-handle. If $K_i$ is in $L$, then $\Sigma_i$ is an annulus, with boundary on the images of $K_i$ in $S^3$ and $S^3_{\Lambda}(J)$. We write $\Sigma$ for the sum of the classes for all components of $J\sqcup L$. If the component $K_i\subset L$  becomes rationally null-homologous in $S^3_{\Lambda}(J)$, we write $\hat \Sigma_i$ for the class in $H_2(W_{\Lambda}(J);\Q)$ obtained by capping with a rational Seifert surface. Finally, if $\ve{s}\in \bH(L)$,  we write $\frz_{\ve{s}}^{J}\in \Spin^c(W_{\Lambda}(J))$ for the $\Spin^c$ structure which satisfies Equation~\eqref{eq:grading-z-s} for all link components $K_i$ in $J$.

 \begin{lem}[\cite{ZemBorderedProperties}*{Theorem 10.8}] \label{lem:grading-shift-links}
Suppose that $J\cup L\subset S^3$ is a  partitioned link and $\Lambda$ is an integral framing on $J$ such that $b_1(S^3_{\Lambda}(J))=0$. Then the isomorphism $\ve{\cCFL}(S^3_{\Lambda}(J),L)\simeq \cC_{\Lambda}(J,L)$ from Theorem~\ref{thm:Manolescu-Ozsvath-subcube} is absolutely $\gr_w$-graded if we equip $\cC_{\veps}(\ve{s})\subset \cC_{\Lambda}(J,L)$ with the Maslov grading
\[
\tilde{\gr}_w=\gr_w+\frac{c_1\left(\frz_{\ve{s}}^{J}\right)^2-2\chi(W_{\Lambda}(J))-3\sigma(W_{\Lambda}(J))}{4}+|J|-|\veps|,
\]
where $\gr_w$ is the internal $\gr_w$-grading on $\cCFL(J\cup L)$. The isomorphism  is absolutely graded with respect to the Alexander grading $A=(A_1,\dots, A_{|L|})$ if we define $A_i$ on $\cC_{\veps}(\ve{s})$ via the formula 
\[
A_i=s_i+\frac{\langle c_1\left(\frz_{\ve{s}}^{J}\right),\hat{\Sigma}_i\rangle -\Sigma \cdot \hat \Sigma_i}{2}.
\]
\end{lem}

\begin{prop} If $b_1(\YG)=0$, the isomorphism from Theorem~\ref{thm:lattice=HFL} respects the absolute Maslov and Alexander gradings.
\end{prop}
\begin{proof}
As a first step, we consider the absolute case when there are no arrow vertices.  We recall that we already constructed an isomorphism $\Phi_G\colon \Char(\XG)\to \bH(L)$ in Equation~\eqref{eq:Spin-c-lattice-iso}. In the present case, it is straightforward to verify from the definitions that
$
c_1(\frz_{\ve{s}})=-\Phi_G^{-1}(\ve{s}).
$
If $\ve{s}\in \bH(L_G)$ and $\ve{s}=\Phi_{G}(K)$, then this is equivalent to
\begin{equation}
c_1(\frz_{\ve{s}})=-K. \label{eq:c_1-z-s-and-K}
\end{equation}
 Next, we recall that in Lemma~\ref{lem:relate-gradings}, we identified the quantity $g([K,E])$ with the $\gr_w$-grading of the generator of the tower in the $\ve{s}$-graded subspace of the homology of $\cCFL(L_G)$, localized at the $\scV_i$ variables for vertices in $E$. Noting that $(-K)^2=K^2$, we obtain Equation~\eqref{eq:Maslov-grading}.

We now consider the case that there are arrow vertices. Suppose that $\Gamma$ is a tree with $V_\Gamma=V_G\cup V_\uparrow$. Lemma~\ref{lem:grading-shift-links} computes the Maslov grading shift. If $K\in \Char(X_{\Gamma})$ and $\ve{s}=\Phi_{\Gamma}(K)$, then we have similarly to Equation~\eqref{eq:c_1-z-s-and-K} that
\begin{equation}
c_1\left(\frz_{\ve{s}}^{L_G}\right)=-K|_{\XG}.
\label{eq:c_1-z-s-K-restriction}
\end{equation}
In particular, the statement from Lemma~\ref{lem:relate-gradings} implies that the Maslov grading on $\ve{\cCFL}(S^3_{\Lambda}(L_G),L_\uparrow)$ coincides with the one defined in Equation~\eqref{eq:grading-gr-w-lattice-link}.

We now consider the Alexander grading. We note that given $\ve{s}\in \bH(L_\Gamma)$, there are two $\Spin^c$ structures of interest: $\frz_{\ve{s}}^{L_G}\in \Spin^c(\XG)$ and $\frz_{\ve{s}}^{L_\Gamma}\in \Spin^c(X_\Gamma)$. It is straightforward to verify that $\frz_{\ve{s}}^{L_G}=\frz_{\ve{s}}^{L_\Gamma}|_{\XG}$.

If $\ve{s}=(s_1,\dots, s_{|L_\Gamma|})\in\bH(L_\Gamma)$, and $K_i\in L_\uparrow$, we view $\cC_{\Lambda}(L_G,L_\uparrow,\ve{s})$ as having internal Alexander grading $A_i$ equal to $s_i$.  Lemma~\ref{lem:grading-shift-links} implies that the isomorphism $\ve{\cCFL}(S^3_{\Lambda}(L_G),L_\uparrow)\simeq \cC_{\Lambda}(L_G,L_\uparrow)$ is Alexander graded if we shift the internal Alexander grading $s_i$ (for a component $K_i\in L_\uparrow$) of $\cC_{\Lambda}(L_G,L_\uparrow,\ve{s})$ by $(\langle c_1(\frz_{\ve{s}}^{L_{G}}),\hat \Sigma_i\rangle-\Sigma \cdot \hat \Sigma_i)/2$. 

By the definition of $\Phi_\Gamma$, if  $\Phi_{\Gamma}(K)=\ve{s}$, then
\[
s_i=\frac{\langle K,\Sigma_i\rangle+\Sigma\cdot \Sigma_i}{2},
\]
where the pairings occur in $X_\Gamma$. In particular, the generator $[K,E]$ in the lattice complex will be given $i^{\mathrm{th}}$ Alexander grading
\[
s_i+\frac{\langle c_1(\frz_{\ve{s}}^{L_G}), \hat{\Sigma}_i \rangle -\Sigma\cdot \hat \Sigma_i}{2}.
\]
By Equation~\eqref{eq:c_1-z-s-K-restriction}, we may write the above as
\[
\frac{K(v_i-\hat{v}_i)+\sum_{v\in V_\Gamma} v\cdot (v_i-\hat{v}_i)}{2}.
\]
We note that for $v\in V_G$, the pairing $v\cdot (v_i-\hat{v}_i)$ will vanish, and hence we can replace the sum in the above equation with a sum over only $v\in V_\uparrow$. This recovers the formula for the Alexander grading stated in Section~\ref{sec:Alexander-gradings}, so the proof is complete.
\end{proof}

\section{Plumbed L-space links}\label{sec:plumbed_l_space}

In this section, we compute the link Floer complexes of plumbed L-space links. In Section~\ref{sec:free-resolution-proof}, we prove that their complexes are formal (i.e. $\cCFL(\YG,L_\uparrow)$ is homotopy equivalent as an $A_\infty$-module to $(\cHFL(\YG,L_\uparrow),m_j)$, where $m_j=0$ unless $j=2$).  Consequently, the chain complexes are also homotopy equivalent to free resolutions of their homology. In Section~\ref{sub:Alexander_and_Floer} we recall work of Gorsky and N\'{e}methi \cite{GorskyNemethiLattice} which allows one to compute the module $\cHFL(\YG,L_\uparrow)$ when $\YG=S^3$ and $L_\uparrow$ is an L-space link. We extend their description to the case where $\YG$ is a rational homology L-space. Finally, in Section~\ref{sec:comparison}, we prove that our model of link lattice homology recovers the version of Gorsky and N\'{e}methi \cite{GorskyNemethiLattice} in the case of plumbed L-space links.

\subsection{Plumbed L-space links and free resolutions}
\label{sec:free-resolution-proof}
Suppose $\Gamma$ is an arrow-decorated plumbing tree.
Our goal is to show that the chain complex $\cCFL(\YG,L_\uparrow)$ is a free
resolution of its homology, in particular, that the chain complex is determined by the homology
up to chain homotopy equivalence. 

In this section, we consider plumbings where $L_\uparrow$ is an L-space link. We observe that this implies that $Y_G$ is a rational homology 3-sphere, and furthermore is itself an L-space.
In particular, by Proposition~\ref{prop:GN_algebraic},
the results in this subsection hold for links of embedded analytic singularities, as long as the underlying surface
singularity is rational.

\begin{thm}\label{thm:algebraic-free resolution}
Suppose that $\Gamma$ is an arrow-decorated plumbing tree, with $V_\Gamma=V_G\cup V_\uparrow$. Let $L_\uparrow\subset \YG$ be the associated link and assume that $\YG$ is a rational homology 3-sphere. If $L_\uparrow$ is an L-space link, then $\cCFL(\YG,L_\uparrow)$ is a free resolution over $\scR_\ell$ of $\cHFL(\YG,L_\uparrow)$. 
\end{thm}
\begin{proof}
  To simplify the notation, we assume that $\YG$ is an integer homology 3-sphere. For rational homology 3-spheres, one may apply the same argument to each $\Spin^c$ structure.
  
  Next, we observe that it is sufficient to show  that $\ve{\cCFL}(Y_G,L_\uparrow)$ is homotopy equivalent to a free resolution of $\ve{\cHFL}(Y_G,L_\uparrow)$. This may be seen as follows: The $\cR_\ell$-module $\cHFL(Y_G,L_\uparrow)$ is finitely generated and hence admits a finitely generated free resolution over $\cR_\ell$ by Hilbert's syzygy theorem \cite{HilbertSyzygy}. See e.g. \cite[Theorem 15.2]{PeevaBook} for a modern exposition. Furthermore, it is straightforward to see that if $C$ and $C'$ are two free, finitely generated chain complexes over $\cR_\ell$ which are both $(\gr_{w},\gr_{z})$-graded, then $C$ and $C'$ are homotopy equivalent over $\cR_{\ell}$ if and only if $C\otimes_{\cR_{\ell}} \ve{\cR}_{\ell}$ and $C'\otimes_{\cR_\ell} \ve{\cR}_\ell$ are homotopy equivalent. Moreover, the completion of a free resolution of $\cHFL(Y_G,L_\uparrow)$ will be a free resolution  of $\ve{\cHFL}(Y_G,L_\uparrow)$, by similar reasoning. Hence, $\cCFL(Y_G,L_\uparrow)$ will be homotopy equivalent to a free resolution of $\cHFL(Y_G,L_\uparrow)$ if and only if $\ve{\cCFL}(Y_G,L_\uparrow)$ is homotopy equivalent to a free resolution of $\ve{\cHFL}(Y_G,L_\uparrow)$.

Since $L_\uparrow\subset \YG$ is an L-space link, $\cHFL(\YG,L_\uparrow)$ is a torsion-free $\bF[U]$-module. By Theorem~\ref{thm:lattice=HFL}, $\ve{\cCFL}(\YG, L_{\uparrow})$ is homotopy equivalent to $\mathbb{CFL}(\Gamma,V_\uparrow)$ as an $A_\infty$-module over $\ve{\cR}_{\ell}$.

The link lattice complex has a cube grading:
\[
\mathbb{CFL}(\Gamma,V_\uparrow)=(
\begin{tikzcd}
Z_n\ar[r, "f_{n,n-1}"]& Z_{n-1} \ar[r]& \cdots \ar[r]& Z_1\ar[r, "f_{1,0}"]& Z_0,
\end{tikzcd})
\]
where $Z_q$ is spanned by $U^p[K,E]$ where $|E|=q$, $p\ge 0$.
Furthermore, each $Z_i$ is itself an $\ve{\cR}_{\ell}$-module. (In particular, the action of $\ve{\cR}_{\ell}$ preserves the cube grading).

We claim that the homology of $\mathbb{CFL}(\Gamma,V_\uparrow)$ is supported in a single $Z_{i}$. We observe that if there are two $Z_i$ which support the homology, then $\ve{\cHFL}(Y_G,L_\uparrow)$ will split as a direct sum of $\ve{\cR}_\ell$-modules.  We claim that this is impossible. 

 To see this, we argue by considering the localization at the multiplicatively closed subset $S\subset \ve{\cR}_{\ell}$ spanned by monomials. Since  $L_\uparrow\subset Y_G$ is an L-space link, it follows that $s\cdot x\neq 0$ for all $s\in S$ and non-zero $x\in \ve{\cHFL}(Y_G,L_\uparrow)$. Therefore, the localization map is injective, so it suffices to show that $S^{-1} \ve{\cHFL}(Y_G,L_\uparrow)$ does not split as the direct sum of two $\ve{\cR}_\ell$-modules.

 We claim that  $S^{-1}\mathbb{CFL}(\Gamma,V_\uparrow)$ is isomorphic to $\ve{\cHFL}^\infty(S^3, \bU_\ell)$, which is spanned by a single generator under the $\ve{\cR}_\ell$-action. To see this, we observe that localization is an exact functor, so it suffices to consider the localization of the chain complex. We pick an arbitrary $(s_1,\dots, s_n)\in \bH(Y,L)$. We observe that after localizing, we may perform a change of basis and replace each basis element $\xs$ with a basis element
\[
\scV_1^{-A_1(\xs)+s_1}\cdots \scV_\ell^{-A_\ell(\xs)+s_\ell} \xs. 
\]
(We use $\ve{s}$ so that the powers of $\scV_i$ are integral). With this choice of basis, all generators of $\ve{\cCFL}(Y_G,L_\uparrow)$ are now concentrated in Alexander grading $\ve{s}$. Hence, each component of the differential is weighted by powers of $\scU_i\scV_i$. Setting $\scV_i$ equal to 1 recovers $\ve{\CF}^\infty(Y_G,L_\uparrow)$, so we conclude that
\[
\ve{\cCFL}^\infty(Y_G,L_\uparrow)\iso \ve{\CF}^\infty(Y_G,w_1,\dots, w_\ell)\otimes_{\bF[U_1,\dots, U_\ell]} \ve{\cR}_\ell,
\]
where $U_i$ acts on $\ve{\cR}_\ell$ by $\scU_i\scV_i$.  Here, $\ve{\CF}^-(Y_G,w_1,\dots, w_\ell)$ denotes the $\ell$-pointed Floer complex for $Y_G$. It follows from \cite{OSProperties}*{Theorem~10.1} that the singly pointed Floer complex $\ve{\CF}^-(Y_G,w_1)$ is homotopy equivalent to $\bF[[U,U^{-1}]$, 
and by \cite{OSLinks}*{Proposition~6.5}, adding the basepoint $w_i$ has the effect of tensoring with the Koszul complex $\begin{tikzcd}\xs\ar[r, "U_1+U_i"] & \ys\end{tikzcd}$. In particular,
it follows that $\ve{\cHFL}^\infty(Y_G,L_\uparrow)$ is isomorphic to 
\[
S^{-1} \ve{\cR}_{\ell}/(\scU_i\scV_i-\scU_j\scV_j: i,j\in \{1,\dots, \ell\})
\]
 This does not decompose as a direct sum of $\ve{\cR}_\ell$-modules, since it has a single generator (the image of $1\in \ve{\cR}_{\ell}$) over the ring $S^{-1} \ve{\cR}_\ell$.

  Consider first the case that the homology is supported in cube grading $i=0$. In this case, we may pick a splitting over $\bF$ of the complex $\mathbb{CFL}(\Gamma,V_\uparrow)$.  Such a splitting determines a homotopy equivalence of $\mathbb{CFL}(\Gamma,V_\uparrow)$ with its homology. We may apply the homological perturbation lemma to obtain an induced $A_\infty$-module structure over $\ve{\cR}_{\ell}$ on the homology $ \mathbb{HFL}(\Gamma,V_\uparrow)$. A filtration argument like the one in the proof of Theorem~\ref{thm:lattice=HFL}  implies that the induced $A_\infty$-module structure on $Z_{0}/\im Z_{1}\iso H_* \mathbb{CFL}(\Gamma,V_\uparrow)$
  has $m_j=0$ unless $j=2$. By Corollary~\ref{cor:free resolution=simplest-action}, this implies that $\ve{\cCFL}(\YG,L_\uparrow)$ is homotopy equivalent over $\ve{\cR}_{\ell}$ to a free resolution of its homology.

  We now consider the case that $\mathbb{CFL}(\Gamma,V_\uparrow)$ is supported at $Z_{i}$ for some $i\neq 0$.    Our argument proceeds by induction, with the base case $i=0$ covered above.
 We will use techniques described in Subsection~\ref{sub:free}.

As $i>0$, in particular $H_*(Z_{0})=0$, so $f_{1,0}$ is surjective. We may pick a splitting $i_{0,1}$ of $f_{1,0}$ as a map of vector spaces, which induces a splitting of $Z_{1}$ as $Z_{1}=Z_{1}^l\oplus Z_{1}^r$ (where the direct sum is of $\bF$-vector spaces), and $Z_{1}^l=\ker(f_{1,0})$ and $Z_{1}^r=\im(i_{0,1})$. Note that $Z_{1}^r$ is not in general an $\ve{\cR}_{\ell}$-module, as it is the image of an $\bF$-linear map, however $\ker(f_{1,0})$ is always an $\ve{\cR}_{\ell}$-submodule since $f_{1,0}$ is $\ve{\cR}_{\ell}$-equivariant.

 There is a chain complex $\cZ'$ obtained by deleting $Z_{1}^r$ and $Z_{0}$, and the above maps determine a chain homotopy equivalence between $\mathbb{CFL}(\Gamma,V_\uparrow)$ and $\cZ'$ as chain complexes over $\bF$. Via the homological perturbation lemma for $A_\infty$-modules, we may equip $\cZ'$ with an $A_\infty$-module structure over $\ve{\cR}_{\ell}$ which is homotopy equivalent to $\mathbb{CFL}(\Gamma,V_\uparrow)$. The map $h$ appearing in the homological perturbation lemma is the map $i_{0,1}$. The inclusion and projection maps are the obvious ones. Compare Section~\ref{sub:free}. Since $\ker(f_{1,0})$ is closed under the action of $\ve{\cR}_{\ell}$, the action on $\cZ'$ from the homological perturbation lemma is the standard one.
 This reduces the index at which the homology of $\mathbb{CFL}(\Gamma,V_\uparrow)$ is supported.
  Proceeding by induction we reduce to the base case $i=0$, completing the proof.
\end{proof}

\subsection{Computing the Floer chain complex from Alexander polynomials}
\label{sub:Alexander_and_Floer}

In this subsection, we consider the $H$-function for oriented $\ell$-component links $L$ in rational homology spheres $Y$. We begin by defining a lattice $\bH(Y,L)$, which is an affine space over $H_1(Y\setminus L,\Z)$. The set $\bH(Y,L)$ is a subspace of $\Q^{\ell}\times \Spin^c(Y)$. The simplest definition of $\bH(Y,L)$ is that it is the set of $(\ve{s},\frt)\in \Q^{\ell}\times \Spin^c(Y)$ such that $\widehat{\HFK}(Y,L,\frt)$ is non-trivial in some Alexander grading $\ve{s}'\in \Q^\ell$ satisfying $\ve{s}-\ve{s}'\in \Z^\ell$. 
Since $\widehat{\HFK}(Y,L,\frt)\neq 0$ for each $\frt$ when $Y$ is a rational homology 3-sphere, this construction gives a well-defined set $\bH(Y,L)$. The action of an element $\g\in H_1(Y\setminus L)$ is given by 
\[
\g\cdot (s_1,\dots, s_\ell,\frt)=(s_1-\lk(\g,K_1),\dots, s_\ell-\lk(\g,K_\ell),\frt+\PD[i_*\g]) 
\]
where $\lk(\g,K_i)\in \Q$ is the rational linking number, and $i\colon Y\setminus L\to Y$ is inclusion.

A more topological description may be obtained by presenting $Y\setminus L$ as surgery on a link $L'$ in the complement of an $\ell$-component unlink in $S^3$. Such a presentation induces  link cobordism $(W,\Sigma)$ from the complement of an $\ell$-component unlink to $Y\setminus L$ such that $W$ is a 2-handlebody, and $\Sigma$ consists of $\ell$ annuli,  each of which  cobounds an unknot component in $S^3$ and a knot component of $L$. If $\frt\in \Spin^c(Y)$, the fiber over $\frt$ under the map $\bH(Y,L)\to \Spin^c(Y)$ consists of $(\ve{s},\frt)$ where
\begin{equation}
\ve{s}\in \left(\frac{\langle c_1(\frt'), \hat{\Sigma}_1\rangle-[\hat {\Sigma}_1]\cdot [\Sigma]}{2},\cdots, \frac{\langle c_1(\frt'),\hat{\Sigma}_\ell\rangle-[\hat {\Sigma}_\ell]\cdot [\Sigma]}{2}\right)+\Z^\ell, 
\label{eq:Alexander-grading-link}
\end{equation}
and $\frt'\in \Spin^c(W)$ is any lift of $\frt\in \Spin^c(Y)$. Here, we view $\Sigma$ as the union of $\Sigma_1,\dots, \Sigma_\ell$. We write $[\Sigma]$ and $[\Sigma_i]$ for the induced classes in $H_2(W,\d W;\Z)$, and we write $[\hat {\Sigma}_i]$ for the lifts under the map $H_2(W;\Q)\to H_2(W,\d W;\Q)$. 

This may be seen to coincide with the definition in terms of Alexander gradings on $\widehat{\HFL}(Y,L)$ using a small modification of the cobordism argument from \cite[Section~5.5]{ZemAbsoluteGradings} (which is stated for integrally null-homologous links). See \cite[Section~3.2]{HHSZDuals} for an exposition in the setting of rationally null-homologous knots. It is an easy consequence of the cobordism description of the Alexander grading that $\widehat{\CFL}(Y,L)$ is supported on $\bH(Y,L)$.

\begin{lem} As affine spaces over $H_1(Y\setminus L;\Z)$, there is an isomorphism $\bH(Y,L)\iso H_1(Y\setminus L;\Z)$.
\end{lem}
\begin{proof} This follows from the short exact sequence of affine spaces
  \begin{equation}\label{eq:short_spin}
0\to \Z^\ell\to \bH(Y,L)\to \Spin^c(Y)\to 0
\end{equation}
which is parallel to the short exact sequence of homology groups
\[
0\to \Z^\ell\to H_1(Y\setminus L;\Z)\to H_1(Y;\Z)\to 0.
\]
If we pick a base element $(\ve{s}, \frt)\in \bH(Y,L)$, we obtain a map of affine spaces from $H_1(Y\setminus L;\Z)$ to $\bH(Y,L)$ which makes the natural diagram commute. By the five-lemma, we obtain that $\bH(Y,L)$ and $H_1(Y\setminus L;\Z)$ are isomorphic as affine spaces.
\end{proof}

We are now ready to define the $H$-function of a link in a rational homology sphere.
\begin{define}
  For an oriented link $L\subset Y$ in a rational homology sphere $Y$ and $(\ve{s},\frt)\in\bH(Y,L)$, we define the 
  $H_{L}\colon\bH(Y,L)\to\Q$
  by saying that $-2H_{L}(\ve{s}, \frt)$ is the maximal $\gr_w$-grading of a non-zero element in the free part of 
  $H_{\ast}(\frA^{-}(L, \ve{s},\frt))$ where $\frA^{-}(L, \ve{s},\frt)$ is the subcomplex of $\cCFL(Y,L,\frt)$ lying in Alexander grading $\ve{s}$.
\end{define}

The $H$-function of L-space links in the 3-sphere can be computed from Alexander polynomials of the link and all sublinks \cite{GorskyNemethiLattice, BorodzikGorskyImmersed}. In order to generalize the result to links in rational homology sphere, we first recall \emph{generalized Alexander polynomials}. Friedl, Juh\'{a}sz and Rasmussen  \cite{FJRDecat}*{Theorem~1} prove that $\widehat{\HFL}(Y,L)$ categorifies the Turaev torsion $\Delta(Y,L)$ of $Y\setminus L$, which we view as an element in $\bF[H_1(Y\setminus L)]$, well-defined up to multiplication by monomials. When $Y$ is a rational homology 3-sphere, we refer to $\Delta(Y,L)$ as the \emph{generalized Alexander polynomial}.

 In our present setting, it is helpful to view $\Delta(Y,L)$ as taking values in the $\bF[H_1(Y\setminus L)]$-module $\bF[\bH(Y,L)]$ instead of $\bF[H_1(Y\setminus L)]$ itself, via the following formula:
\[
\Delta(Y,L)=\chi(\widehat{\HFL}(Y,L)):=\sum_{(\ve{s},\frt)\in \bH(Y,L)} t^{(\ve{s},\frt)}\cdot \chi(\widehat{\HFL}(Y,L,\ve{s},\frt))\in \bF[\bH(Y,L)].
\]
We may think of $\chi(\widehat{\HFL}(Y,L))$ as a normalized version of the Turaev torsion. Similarly to \cite{OSLinks}*{Proposition~8.1}, it is not hard to see that the Euler characteristic of $\widehat{\HFL}(Y,L)$ is symmetric with respect to the involution of $\bH(Y,L)$ given by $(\ve{s},\frt)\mapsto (-\ve{s},\bar{\frt}+\PD[L])$ (cf. \cite{ZemAbsoluteGradings}*{Proposition~8.3}).

To make more transparent connections with integer homology spheres, we refine the definition of $\Delta(Y,L)$. For $\frt\in\Spin^c(Y)$, we set:
\begin{equation}
\label{eq:Alexander}
  \Delta(Y,L,\frt)=\chi(\widehat{\HFL}(Y,L,\frt)):=\sum_{\ve{s}\colon (\ve{s},\frt)\in \bH(Y,L)} t^{\ve{s}}\cdot \chi(\widehat{\HFL}(Y,L,\ve{s},\frt))\in \bF[\bH(Y,L)].
\end{equation}

We recall that the graded Euler characteristics of $\widehat{\HFL}$ and $\HFL^-$ are related by the formula
\[
\chi(\HFL^-(Y,L))=\frac{1}{\prod_{i=1}^\ell (1-t_i)} \chi(\widehat{\HFL}(Y,L)).
\]
See \cite{OSLinks}*{Proposition~9.2}.

For any sublink $L'\subset L$, define the natural forgetful map:
  \[
\pi_{L,L-L'}\colon \bH(Y,L)\to \bH(Y,L- L')
\]
as follows.  
In the case that $L'=K_1$, the map $\pi_{L,L-L'}$ is given by the formula
\[
\pi_{L,L-L'}(s_1,\dots, s_\ell, \frt)=\left(s_2-\dfrac{\lk(K_1,K_2)}{2},\dots, s_\ell-\dfrac{\lk(K_1,K_\ell)}{2}, \frt \right),
\]
where $L=K_1\cup\cdots \cup K_\ell$. For general $L'$, the formula is a composition of several of these maps. We refer reader to \cite[Section 3.7]{MOIntegerSurgery} for explicit formulas of the forgetful map for links in $S^{3}$. Compare  \cite[Section 3]{BorodzikGorskyImmersed}.

Given $(\ve{s}, \frt), (\ve{s}', \frt')\in \bH(Y, L)$, we say $(\ve{s}, \frt)\geq (\ve{s}', \frt')$ if and only if $\frt=\frt'$ and $\ve{s}\geq \ve{s}'$. That is, $s_i\geq s'_i$ for all $i$, where $\ve{s}=(s_1, \dots, s_\ell)$ and $\ve{s}'=(s'_1, \dots, s'_\ell)$. 

\begin{lem}\label{lem:is_determined}
  For an oriented $L$-space link $L\subset Y$ in a rational homology sphere with $\frt\in \Spin^c(Y)$, the $H$-function $H_{L}$ satisfies: 
  \begin{equation}
  \label{eq:H-functionformula}
  H_{L}(\ve{s}, \frt)=\sum_{L'\subset L} (-1)^{| L'|-1} \sum_{\substack{(\ve{s}', \frt)\in \bH(Y, L')\\ (\ve{s}', \frt)\ge \pi_{L, L'}(\ve{s}+\ve{1}, \frt)}} \chi(\HFL^{-}(Y, L', \ve{s}',\frt)).
  \end{equation}
  where $\ve{1}=(1, \dots, 1)$.
  In particular, the $H$-function is determined by the generalized Alexander polynomials $\Delta(Y,L',\frt)$ of all sublinks $L'\subset L$. 
\end{lem}

\begin{proof}
  The arguments of \cite{GorskyNemethiLattice}*{Theorem~2.10}  relating $H_L$ to $\chi(\HFL^{-}(Y, L))$ in the case of links in $S^3$ can
  be repeated verbatim for the case of $H_{L}$ and $\chi(\HFL^{-}(Y, L))$ to get \eqref{eq:H-functionformula}, though here we follow the convention in \cite[Theorem 3.15]{BorodzikGorskyImmersed} and assume $\HFL^{-}(Y, \emptyset)=0$. The right-hand side of \eqref{eq:H-functionformula} is determined by $\chi(\widehat{\HFL}(Y, L'))$ for all sublinks $L'$, which can be computed from  Alexander polynomials of $L'$ by \eqref{eq:Alexander}. Therefore, the $H$-function is determined by the Alexander polynomials of all sublinks $L'\subset L$. 

\end{proof}

\begin{example}
We normalize the multivariable Alexander polynomial of the unlink $\bU$ in the 3-sphere to be $0$, and the $H$-function for an $\ell$-component unlink in $S^{3}$ is the following:
$$H_{\bU_{\ell}}(s_1, \dots, s_\ell)=\sum_{i} (|s_{i}|-s_i)/2.$$
If  $L\subset S^{3}$ is  the Hopf link in the 3-sphere, its Alexander polynomial $\Delta(t_1, t_2)=1$ and $s_{i}\in \mathbb{Z}+1/2$.  By the formula in \cite{BorodzikGorskyImmersed, LiuSurgeries}, the value of the  H-function at each lattice point $(s_1, s_2)$ is  shown as follows:
\end{example}

\begin{figure}[H]
  \begin{tikzpicture}
    \def\myshift{(-0.25,-0.25)}
    \draw[thick,->] (-2,0) -- (2.5,0);
    \draw[thick,->] (0,-2) -- (0,2.5);
    \foreach \position  in {(0.5,0.5),(1.0,0.5),(1.5,0.5),(2.0,0.5),(0.5,1.0),(0.5,1.5),(0.5,2.0),(1.0,1.0),(1.0,1.5),(1.0,2.0),(1.5,1.0),(2.0,1.0)}
         \draw \position ++ \myshift node[scale=0.8] {$0$};
    \foreach \position in {(0,0.5),(0,1.0),(0,1.5),(0,2.0),(0.5,0),(1.0,0),(1.5,0),(2.0,0),(0,0)}
    \draw[blue!20!black] \position ++ \myshift node[scale=0.8] {$1$};
    \foreach \position in {(-0.5,0.5),(-0.5,1.0),(-0.5,1.5),(-0.5,2.0),(0.5,-0.5),(1.0,-0.5),(1.5,-0.5),(2.0,-0.5),(-0.5,0.0),(0.0,-0.5)}
    \draw[green!20!black] \position ++ \myshift node[scale=0.8] {$2$};
    \foreach \position in {(-1.0,0.5),(-1.0,1.0),(-1.0,1.5),(-1.0,2.0),(1.0,-1.0),(1.5,-1.0),(2.0,-1.0),(-0.5,-0.5),(-1.0,0.0),(0.0,-1.0),(0.5,-1.0)}
    \draw[purple!20!black] \position ++ \myshift node[scale=0.8] {$3$};
    \foreach \position in {(-1.0,-0.5),(-0.5,-1.0)}
    \draw[purple!20!black] \position ++ \myshift node[scale=0.8] {$4$};
    \draw[red!30!black] (-1.0,-1.0) ++\myshift node[scale=0.8] {$5$};
\foreach \x in {-1.0,-0.5,...,2.0} \draw (\x,-1.5) ++ \myshift node [scale=0.8] {$\vdots$};
\foreach \y in {-1.0,-0.5,...,2.0} \draw (-1.5,\y) ++ \myshift node [scale=0.8] {$\cdots$};
\foreach \x in {-1.0,-0.5,...,1.0} \draw (\x,2.5) ++ \myshift node [scale=0.8] {$\vdots$};
\foreach \y in {-1.0,-0.5,...,1.0} \draw (2.5,\y) ++ \myshift node [scale=0.8] {$\cdots$};
\draw (2.3,1.7) ++ \myshift node [scale=0.8] {$H(s_1,s_2)=0$};
\draw (2.7,0) node [scale=0.8] {$s_1$};
\draw (0,2.7) node [scale=0.8] {$s_2$};
\draw (2.7,0.5)++\myshift -- ++ (0.3,0) -- ++ (0.5,0.5) -- node [midway, above, scale=0.7] {$s_2=1/2$} ++ (0.8,0);
\draw (0.5,2.7)++\myshift -- ++ (0,0.2) -- ++ (0.5,0.3) -- node [midway, above, scale=0.7] {$s_1=1/2$} ++ (0.8,0);
  \end{tikzpicture}
\caption{The $H$-function of the Hopf link  \label{h-function}}
\end{figure}


\begin{thm}
\label{thm:alexchain_rational}
Suppose that $\Gamma$ is an arrow-decorated plumbing tree, with $V_\Gamma=V_G\cup V_\uparrow$. Let $L_\uparrow\subset \YG$ be the associated link and assume that $\YG$ is a rational homology 3-sphere. If $L_\uparrow$ is an L-space link, then the full link Floer complex $\cCFL(\YG,L_\uparrow)$ is determined by the generalized Alexander polynomials of $L_\uparrow$ and its sublinks. 
\end{thm}

\begin{proof}
By Theorem \ref{thm:algebraic-free resolution}, the chain complex $\cCFL(\YG, L_\uparrow)$ is a free resolution of $\cHFL(\YG, L_\uparrow)$.
In particular, $\cCFL(\YG,L_\uparrow)$ is determined by the homology group $\cHFL(\YG,L_\uparrow)$ as an $\cR_\ell$-module.
Therefore, it suffices to prove that the generalized Alexander polynomials determine the homology $\cHFL(\YG, L_\uparrow)$. 
By Lemma \ref{lem:is_determined},  the $H$-function is determined by the generalized Alexander polynomials of $L_\uparrow$ and its sublinks, 
it remains to prove that the $H$-function determines the $\cR_\ell$-module structure of $\cHFL(\YG,L_\uparrow)$.

There is a decomposition of $\bF[U]$-modules
 \[
 \cHFL(\YG, L_\uparrow)=\bigoplus_{(\ve{s}, \frt)\in \mathbb{H}(\YG, L_{\uparrow})}\cHFL(Y_G,L_\uparrow, \ve{s}, \frt).
 \] 
 Since $L_{\uparrow}$ is an L-space link in $\YG$, $\cHFL(Y_G,L_\uparrow, \ve{s}, \frt)\cong \mathbb{F}[U]$ for all $(\ve{s}, \frt)\in \mathbb{H}(\YG, L_{\uparrow})$. The $(\gr_w,A)$-grading of the generator of $\cHFL(Y_G,L_\uparrow, \ve{s}, \frt)$ is determined by the $H$-function. That is, the Alexander gradings of the generator equal $\ve{s}$ and the Maslov grading $\gr_w$ equals $-2H_{L_{\uparrow}}(\ve{s, \frt})$.  Hence,  it suffices to see the $\scU_i$ and $\scV_i$ actions on $\cHFL(\YG, L_\uparrow, \frt)$ are also determined by the $H$-function.

To see this, note that we may view $\scU_i$ as restricting to a map 
\begin{equation}
\cHFL(Y_G,L_\uparrow, \ve{s}, \frt)\iso \bF[U]\to \cHFL(Y_G,L_\uparrow, \ve{s}-e_{i}, \frt)\iso \bF[U].
\label{eq:identification-A-F[U]}
\end{equation}
Since the map $\scV_i$ goes in the opposite direction and $\scU_i\scV_i=U$, the map $\scU_i$ is given by either multiplication by $1$ or $U$, with respect to identifications in~\eqref{eq:identification-A-F[U]}.
We observe that the map $\scU_i$ has $\gr_w$-grading $-2$, and hence the choice of being $U$ or $1$ is determined by the $\gr_w$-gradings of the copies of $\bF[U]$ in ~\eqref{eq:identification-A-F[U]}, which is encoded by the $H$-function. 
 Similarly, the action of  $\scV_i$ is also determined by the $H$-function since $\scU_i\scV_i=U$, and the action of $\scU_i$ is determined by the $H$-function. Therefore, for L-space links, the Alexander polynomials determine the full link Floer chain complexes. 
\end{proof}

\subsection{Comparison to Gorsky and N\'{e}methi's link lattice homology}
\label{sec:comparison}

We now compare our chain complex $\mathbb{CFL}(\Gamma,V_\uparrow)$ to the definition of link lattice homology due to Gorsky and N\'{e}methi \cite{GorskyNemethiLattice}.

As a first step, we recall their definition. Let $L\subset S^3$ be a link of $\ell$ components. 
Write $C_i$ for the mapping cone complexes
\[
 C_i:=\Cone(\scV_i\colon\cR_\ell\to \cR_\ell),\ i=1,\dots,\ell.
 \]
  Write $\scK(L)$ for the (non-free) complex of $\cR_\ell$-modules obtained by tensoring the homology group $\cHFL(L)$ with $C_1,\dots, C_\ell$, over the ring $\cR_\ell$. 

\begin{rem} Note that $C_1\otimes_{\cR_\ell}\cdots \otimes_{\cR_{\ell}} C_\ell$ is the \emph{Koszul complex} for the regular sequence $(\scV_1,\dots, \scV_\ell)$ in $\cR_{\ell}$. See \cite{Weibel}*{Section~4.5}.
\end{rem}

If $L$ is an L-space link, Gorsky and N\'{e}methi define the link lattice complex to be the chain complex $\scK(L)$. Gorsky and N\'{e}methi prove the following:

  \begin{thm}[\cite{GorskyNemethiLattice}*{Theorem 2.9}]\label{thm:Gorsky-Nemethi} If $L$ is an algebraic link, then
\[H_*\scK(L)\iso H_*\left(\cCFL(L)/(\scV_1,\dots, \scV_\ell)\right).
\]
\end{thm}

\begin{rem}
Gorsky and N\'{e}methi prove Theorem~\ref{thm:Gorsky-Nemethi} at the level of graded $\bF$ vector spaces. There are additional module actions of $\bF[[\scU_1,\scV_1,\dots, \scU_\ell, \scV_\ell]]$ on both sides. Gorsky and Hom \cite{GorskyHom}*{Proposition~3.7} equip $\scK(L)$ with commuting endomorphisms $\scU_1,\dots, \scU_\ell$. In our notation, their action of $\scU_i$ coincides with the standard action of $\scU_i$ on the $\cHFL(L)$ tensor factor of $\scK(L)$. Of course, one can also consider the action of $\scV_i$, defined symmetrically. Note that as an endomorphism of $\scK(L)$, each $\scV_i$ is null-homotopic since $\scK(L)$ is defined by tensoring $\cHFL(L)$ with the Koszul complex of the sequence $(\scV_1,\dots, \scV_\ell)$. More explicitly, we can write $\scK(L)$ as a mapping cone 
\[
\scK(L)=\begin{tikzcd} \Cone\big(\scK_i^0(L) \ar[r, "\scV_i"] &\scK_i^1(L)\big), \end{tikzcd}
\]
where $\scK_i^0(L)$ (resp. $\scK_i^1(L)$) is the codimension one subcube which has $i$-coordinate $0$ (resp. $1$). As an endomorphism of $\scK(L)$, the module action of $\scV_i$ preserves both $\scK_i^0(L)$ and $\scK_i^1(L)$. We define an endomorphism $H$ of $\scK(L)$, which sends $\scK_i^1(L)$ to $\scK_i^0(L)$ via the identity, and observe that 
on $\scK(L)$:
\[
\scV_i=[\d, H].
\] 
\end{rem}

We now explain how our Theorem~\ref{thm:algebraic-free resolution} quickly recovers Theorem~\ref{thm:Gorsky-Nemethi}, and also to prove a refinement which takes into account the $\cR_\ell$-action. To state our result, we equip $\scK(L)$ with an $A_\infty$-module structure which has only $m_1$ and $m_2$ non-trivial. The action of $m_2$ corresponds to the standard the action on the $\cHFL(L)$ factor of the tensor product.

\begin{thm}\label{thm:rai_GN}
If $L\subset S^3$ is a plumbed L-space link, then
there is a homotopy equivalence of $A_\infty$-modules over $\bF[\scU_1,\scV_1,\dots, \scU_\ell,\scV_\ell]$
\[
\scK(L)\simeq\cCFL(L)/(\scV_1,\dots, \scV_\ell).
\]
\end{thm}

\begin{proof}
We first consider the proof only at the level of chain complexes, and then subsequently consider the $\cR_\ell$-module structure.

 Note that $C_i$ is a free resolution of $\cR_\ell/\scV_i$ as an $\cR_\ell$-module. Furthermore, the $\ell$-dimensional cube-shaped complex $C_1\otimes \dots \otimes C_\ell$ is a free resolution of $\bF[\scU_1,\scV_1,\dots, \scU_\ell,\scV_\ell]/(\scV_1,\dots, \scV_{\ell})$.

 We use the algebraic formalism of type-$D$ and type-$A$ modules of Lipshitz, Ozsv\'{a}th and Thurston \cite{LOTBordered} \cite{LOTBimodules} to give a small model of the derived tensor product of $A_\infty$-modules. We may view $C_i$ as a type-$D$ module ${}^{\cR_\ell} C_i$ whose underlying $\bF$ vector space has two generators, $\xs_i$ and $\ys_i$, and whose structure map $\delta^1$ is given by
 \[
 \delta^1(\xs_i)=\scV_i\otimes \ys_i.
 \]
 Similarly, the complex $C_1\otimes \cdots\otimes C_{\ell}$ naturally corresponds to a type-$D$ module over $\cR_\ell$, whose underlying vector space is generated by the points of an $\ell$-dimensional cube. We write ${}^{\cR_\ell} C_{1,\dots, \ell}$ for this cube-shaped complex.
 
 By definition, 
 \[
 \scK(L):=\cHFL(L)_{\cR_\ell}\boxtimes {}^{\cR_\ell}C_{1,\dots, \ell}.
 \]

 By Theorem~\ref{thm:algebraic-free resolution}, if $L$ is a plumbed L-space link, $\cHFL(L)$ and $\cCFL(L)$ are homotopy equivalent as $A_\infty$-modules over $\cR_\ell$. On the other hand, $\cCFL(L)_{\cR_{\ell}}$ is free over $\cR_{\ell}$, which translates to the fact that there is a type-$D$ module $\cCFL(L)^{\cR_{\ell}}$ such that
 \[
\cCFL(L)_{\cR_{\ell}}\iso  \cCFL(L)^{\cR_{\ell}}\boxtimes {}_{\cR_{\ell}} \cR_{\ell}{}_{\cR_{\ell}}.
 \]
 
 We observe that since ${}^{\cR_\ell}C_{1,\dots, \ell}$ is a free resolution of $\cR_{\ell}/(\scV_1,\dots, \scV_{\ell})$, it follows that
${}_{\cR_{\ell}} C_{1,\dots, \ell}\simeq {}_{\cR_{\ell}} \cR_{\ell}/(\scV_1,\dots, \scV_{\ell})$.

 Putting these relations together, we obtain
 \begin{equation}
 \begin{split}\scK(L)&=\cHFL(L)_{\cR_\ell}\boxtimes {}^{\cR_\ell}C_{1,\dots, \ell}\\
 &\simeq\cCFL(L)_{\cR_\ell}\boxtimes {}^{\cR_\ell}C_{1,\dots, \ell}\\
 &=\cCFL(L)^{\cR_\ell}\boxtimes {}_{\cR_{\ell}} \cR_{\ell}{}_{\cR_\ell}\boxtimes {}^{\cR_\ell}C_{1,\dots, \ell}\\
 &\simeq\cCFL(L)^{\cR_\ell}\boxtimes {}_{\cR_\ell} \cR_{\ell}/(\scV_1,\dots, \scV_\ell).
 \end{split}
\label{eq:KL=CFL-manipulation}
 \end{equation}

We now consider the $\cR_\ell$-module structures. We note that the complexes ${}^{\cR_{\ell}}C_{1,\dots, \ell}$ extends to a $DA$-bimodule ${}^{\cR_\ell} (C_{1,\dots, \ell})_{\cR_\ell}$. This $DA$-bimodule has the same generators and $\delta^1$ as ${}^{\cR_\ell} C_{1,\dots, \ell}$. Additionally, there is a $\delta_2^1$ action given by $\delta_2^1(\xs, a)=a\otimes \xs$ for any $a\in \cR_\ell$ and $\xs\in C_{1,\dots, \ell}$. We observe that by definition 
\[
\scK(L)_{\cR_\ell}= \cHFL(L)_{\cR_\ell}\boxtimes {}^{\cR_\ell} (C_{1,\dots, \ell})_{\cR_\ell}. 
\]
Additionally, it is easy to check that 
\[
{}_{\cR_\ell} \cR_{\ell}{}_{\cR_\ell} \boxtimes {}^{\cR_\ell} (C_{1,\dots, \ell})_{\cR_{\ell}}\simeq {}_{\cR_\ell} (\cR_\ell/(\scV_1,\dots, \scV_\ell))_{\cR_\ell},
\]
so that the manipulation from Equation~\eqref{eq:KL=CFL-manipulation} extends to an equivalence of $A_\infty$-modules. 
\end{proof}

\section{Computations}\label{sec:examples}

In this section, we give some computational tools. In Subsection~\ref{sub:generators}, we provide an algorithm to compute the link
Floer homology of an L-space link from its $H$-function. Next, in Subsection~\ref{sec:type-A-Tnn}, we describe the $\cR_n$-module $\cHFL(T(n,n))$ based on Gorsky and Hom's computation of the $H$-function of $T(n,n)$. By our Theorem~\ref{thm:algebraic-free resolution}, this $\cR_n$-module contains equivalent information to $\cCFL(T(n,n))$. In Subsections~\ref{sec:T33-free} and \ref{sec:T44-free}, we compute explicit free resolutions of $\cHFL(T(3,3))$ and $\cHFL(T(4,4))$.


\subsection{Generators and relations for the homology of an L-space link}\label{sub:generators}

In this section, we describe generators and relations for the modules $\cHFL(L)$ when $L$ is an L-space link. We focus on the case  $L\subset S^3$ to simplify the notation, but this is not essential.

 If $\ve{s}\in \bH(L)$, we write $X_{\ve{s}}\in \cHFL(L)$ for the generator of the $\bF[U]$-tower in Alexander grading $\ve{s}$. By definition, 
\[
A(X_{\ve{s}})=\ve{s}\quad \text{and} \quad \gr_w(X_{\ve{s}})=-2H_L(\ve{s}).
\]

We let $\cG$ denote the set of $X_{\ve{s}}\in \cHFL(L)$ which satisfy
\begin{equation}
\label{eq:freehomology}
H_L(\ve{s}-e_i)-H_L(\ve{s})=1 \quad \text{and}\quad  H_L(\ve{s})-H_L(\ve{s}+e_i)=0
\end{equation}
for all $i\in \{1,\dots, \ell\}$. Here, $e_i$ denotes the unit vector $(0,\dots, 1 ,\dots 0)$.

\begin{lem}\label{lemma:generating} Let $L$ be an $\ell$-component L-space link in $S^3$.
\begin{enumerate}
\item The set $\cG$ is finite and is the unique $\cR_\ell$-module generating set of $\cHFL(L)$ of minimal length.
\item The kernel of the natural map 
$$\cR_{\ell}^{n}\rightarrow \cHFL(L)$$
is spanned by the following generating set:
\begin{enumerate}[label=($r$-\arabic*), ref=$r$-\arabic*]
\item All elements of the form 
$$\a X_{\ve{s}}+\b X_{\ve{s}'}$$
where $X_{\ve{s}}, X_{\ve{s}'}\in \cG$, and $\a, \b$ are monomials such that $w_U(\a)=w_U(\b)=0$, $\gcd(\a,\b)=1$ with
\[
A(\a)+\ve{s}=A(\b)+\ve{s}'\quad \text{and} \quad \gr_w(\a)-2H_L(\ve{s}) =\gr_w(\b)-2 H_L(\ve{s}'). 
\]
We use the notation $w_U(\alpha)=\min\lbrace i_1, j_1\rbrace+\cdots+\min\lbrace i_{\ell}, j_{\ell}\rbrace$ for any monomial $\alpha=\scU_{1}^{i_{1}}\cdots \scU_{\ell}^{i_{\ell}}\scV_{1}^{j_{1}}\cdots \scV_{\ell}^{j_{\ell}}.$ 
\item \label{relation-2} For each $i$, $j$ and $X_{\ve{s}}\in \cG$ the element
\[
(\scU_i \scV_i+\scU_j\scV_j)X_{\ve{s}}.
\] 
\end{enumerate}
\end{enumerate}
\end{lem}
\begin{proof} Since $\cCFL(L)$ is a finitely generated $\cR_{\ell}$-module, and $\cCFL(L)$ admits Maslov and Alexander gradings, the $\cR_\ell$-module $\cHFL(L)$ is spanned by finitely many homogeneously graded vectors $X_{\ve{s}}$. Equation~\eqref{eq:freehomology} is equivalent to the statement that $X_{\ve{s}}$ is in the image of neither $\scU_i$ nor $\scV_i$ for any $i\in \{1,\dots, \ell\}$. Therefore any minimal length generating set can only contain elements of $\cG$. In particular, $\cG$ generates $\cHFL(L)$.

We now claim that any minimal length generating set must contain all elements of $\cG$. To see this, we write $\cG=\{X_{\ve{s}_1},\dots, X_{\ve{s}_n}\}$ and suppose
\[
X_{\ve{s}_j}=\sum_{i\neq j} \a_i X_{\ve{s}_i}
\]
for some homogeneously graded $\a_i\in \cR_\ell$. We may assume each $\a_i X_{\ve{s}_i}$ has the same $\gr_w$ and Alexander gradings as $X_{\ve{s}_j}$. However $\cHFL(L)$ is rank 1 over $\bF$ in each Maslov and Alexander grading in which it is supported. Therefore if $\a_i\neq 0$, then $X_{\ve{s}_j}=\a_i X_{\ve{s}_i}$. However Equation~\eqref{eq:freehomology} implies that $\a_i=1$ since none of the $X_{\ve{s}_i}$ admit non-trivial factorizations. Therefore every minimal length generating set contains all of $\cG$.

We now consider the claim about relations.  Suppose that 
\[
\sum_{i=1}^n \a_i X_{\ve{s}_i}=0.
\]
We may assume that each $\a_i$ is a monomial (or zero) and all $\a_i X_{\ve{s}_i}$ have the same Alexander and Maslov gradings. When all of the $\a_i X_{\ve{s}_i}$ have the same homogeneous Alexander and Maslov gradings, $\sum_{i=1}^n \a_i X_{\ve{s}_i}=0$  if and only if $\# \{i: \a_i\neq 0\}$ is even, since $\cHFL(L)$ has rank 1 over $\bF$ in each of the gradings in which it is supported. In particular, pairing elements of $\{i: \a_i \neq 0\}$ arbitrarily, we may write the sum $\sum_{i=1}^n \a_i X_{\ve{s}_i}$ as an $\cR_\ell$-linear combination of elements of the form $\a X_{\ve{s}}+\b X_{\ve{s}'}$, where $\a,\b\neq 0$ and such that $\a X_{\ve{s}}$ and $\b X_{\ve{s}'}$ have the same Maslov and Alexander gradings. By canceling common factors, we may assume that $\gcd(\a, \b)=1$. 

If $\max(w_U(\a),w_U(\b))>0$, we claim that we can still reduce this relation further. For concreteness, assume that $w_U(\a)>0$. This means that $\scU_j \scV_j$ divides $\alpha$, for some $j$.  Note that $\b\neq 1$, so assume that either $\scV_i$ or $\scU_i$ divides $\b$, for some $i$. By using ~\eqref{relation-2}, we may replace the factor of $\scU_j\scV_j$  in $\a$ with $\scU_i\scV_i$, and assume $\scU_i \scV_i$ divides $\a$. 
The resulting monomials $\a$ and $\b$ have $\gcd(\a,\b)\neq 1$, so we divide out common factors. This operation is non-increasing in the total degree of $\a$ and $\b$, so we may repeat this procedure until $\gcd(\a,\b)=1$ and $w_U(\a)=w_U(\b)=0$. This completes the proof.  
\end{proof}

\begin{rem}\label{rem:more-concrete-rels} We can describe the relations ~\eqref{relation-2} more concretely, as follows. Let $X_{\ve{s}}, X_{\ve{s}'}\in \cG$ and write
\[
\ve{s}=(s_1,\dots, s_\ell), \quad \text{and} \quad \ve{s}'=(s_1',\dots, s_\ell'). 
\]
 We let 
\[
\scP_{i}=
\begin{cases}
\scV_i& \text{ if } s_i<s_i'\\
\scU_i & \text{ if } s_i>s_i'
\end{cases}
\quad \text{and} \quad 
\scQ_i=
\begin{cases}
\scU_i & \text{ if } s_i<s_i'\\
\scV_i& \text{ if } s_i>s_i'.
\end{cases}
\]
If $s_i=s_i'$, then we view $\scP_i=\scQ_i=1$. The relations labeled~\eqref{relation-2} may be rewritten as those of the form
\[
\scP_1^{i_1}\cdots \scP_\ell^{i_\ell}\cdot X_{\ve{s}}=\scQ_1^{j_1} \cdots \scQ_{\ell}^{j_\ell}\cdot X_{\ve{s}'}
\]
ranging over all $X_{\ve{s}},X_{\ve{s}'}\in \cG$ and all sequences $i_1,\dots, i_\ell\ge 0$ and $j_1,\dots, j_\ell\ge 0$ such that
\[
i_k+j_k=|s_k-s_k'| \quad \text{and} \quad \gr_w(\scP_1^{i_1}\cdots \scP_\ell^{i_\ell}\cdot X_{\ve{s}})=\gr_w(\scQ_1^{j_1} \cdots \scQ_{\ell}^{j_\ell}\cdot X_{\ve{s}'}).
\]
for all $k$. Note in particular that there are only finitely many relations of the form~\eqref{relation-2}. 
\end{rem}

\subsection{The type-$A$ module of the $T(n,n)$ torus link}
\label{sec:type-A-Tnn}
We now compute the Heegaard Floer module $\cHFL(T(n, n))$ of torus link $T(n, n)$ in the three sphere. Recall that $\cR_n=
\bF[\scU_1,\scV_1,\dots, \scU_n,\scV_n]$. 
\begin{thm}
\label{thm:toruslink}
As an $\cR_n$-module, the  group $\cHFL(T(n,n))$ has a unique minimal generating set, consisting of $n$ generators, $X_1,\dots,X_n$. The relations are spanned by the following
\begin{align}
  \left(\prod_{i\in I_k}\scU_i\right) X_k&=\left(\prod_{j\in \{1, \dots, n\}\setminus I_k}\scV_j\right) X_{k+1}\label{eq:staircase}\\
  \scU_i\scV_i X_k&=\scU_j\scV_j X_k\label{eq:homology}.
\end{align}
Here $I_k$ runs through all subsets of the set $\{1,\dots,n\}$ of length $k$ (so \eqref{eq:staircase} has $\binom{n}{k}$ relations for each $k$), and in \eqref{eq:homology}, $i,j$, and $k$ range from $1$ to $n$.
\end{thm}

\begin{proof}
Since the torus link $T(n, n)$ is an L-space link, by Lemma \ref{lemma:generating} the minimal generating set consists of the set of all $X_{\ve{s}}$ which satisfy
\begin{equation}
\label{freehomology}
H(\ve{s}-e_i)-H(\ve{s})=1 \quad \text{and}\quad  H(\ve{s})-H(\ve{s}+e_i)=0,
\end{equation}
for each $i\in \{1,\dots, n\}$. Here,  $H(\ve{s})$ denotes the $H$-function of the torus link $T(n, n)$.

The $H$-function of the torus link $T(n, n)$ is computed in \cite{GorskyHom}. Its Alexander polynomial equals to 
$$\Delta(t_1, \dots, t_n)=((t_1\cdots t_n)^{1/2}-(t_1\cdots t_n)^{-1/2})^{n-2}, $$
which is symmetric in the variables $t_1,\dots, t_n$. Because of the symmetry of the Alexander polynomial of $T(n,n)$, the $H$-function for $T(n,n)$ is also symmetric in  $s_1,\dots, s_n$. Therefore we consider the case that  $s_1\leq s_2\leq \cdots \leq s_n$. 
By \cite[Theorem 4.3]{GorskyHom}
\begin{equation}
\label{computationH}
H(s_1, \dots s_n)=h\left(s_1-\dfrac{n-1}{2}\right)+h\left(s_2-\dfrac{n-1}{2}+1\right)+\cdots +h\left(s_n-\dfrac{n-1}{2}+n-1\right)
\end{equation}
where $h(s)=(|s|-s)/2$. We note that $h(s)$ is the $H$-function of the unknot.  It is not hard to compute that 
\begin{equation}\label{eq:H1}
H(s_1, \dots, s_n)=0
\end{equation} 
if $s_i\geq (n-1)/2$ for all $i$, and  
\begin{equation}\label{eq:H2}
H(\ve{s})=-(s_1+s_2+\cdots+s_n)
\end{equation} 
if $s_{i}\leq -(n-1)/2$ for all $i$.

We first consider the diagonal vertices, that is, $s_1=\cdots=s_n=m$, $m\in\Z+(n-1)/2$. 
If $m>(n-1)/2$, then 
\[
H(s_1, \dots, s_n)=H(s_1-1, s_2, \dots, s_n)=0
\]
by \eqref{eq:H1}, 
which does not satisfy  \eqref{freehomology}. Similarly, if $m<-(n-1)/2$, then 
$$H(s_1, \dots, s_n)=H(s_1, \dots, s_n+1)+1$$
by \eqref{eq:H2},
which also does not satisfy  \eqref{freehomology}. 
Now we consider the $n$ diagonal vertices where $-(n-1)/2\leq m \leq (n-1)/2$, that is, 
\[
s_1=s_2=\cdots=s_{n}=\dfrac{n+1}{2}-k
\]
for all integers $k$ between $1$ and $n$. 
The corresponding generators are $X_1,\dots, X_n$ with Alexander gradings
\[
\left(\dfrac{n+1}{2}-k, \dfrac{n+1}{2}-k, \dots, \dfrac{n+1}{2}-k\right).
\]
 By a straightforward computation,  Equation~\eqref{freehomology} is satisfied by values of the $H$-function at these vertices, so $X_1,\dots,X_n$ are all contained in the minimal generating set from Lemma~\ref{lemma:generating}. 
  We claim these diagonal vertices are the only ones where the $H$-function satisfies  Equation~\eqref{freehomology}. It suffices to prove  that the non-diagonal vertices do not satisfy Equation~\eqref{freehomology}.  Recall that we assume $s_1\leq s_2\leq \cdots \leq s_n$. Suppose that $s_n=s$ is the maximal value among the $s_i$, and there are exactly $\lambda$ coordinates equal to $s$ where $\lambda<n$, i.e., $s_{n-\lambda+1}=s_{n-\lambda+2}=\cdots=s_{n}=s$, and $s_{n-\lambda}<s$.

 If $s>(\lambda-1)-(n-1)/2$,  then by \eqref{computationH} 
\[
H(s_{1}, \dots, s_{n-\lambda}, s_{n-\lambda+1}-1, \dots, s_n)=H(s_1, \dots, s_n)
\]
which does not satisfy  \eqref{freehomology}. 

If $s\leq (\lambda-1)-(n-1)/2$, then by \eqref{computationH}
\[
H(s_1, \dots, s_{n-\lambda}+1, s_{n-\lambda+1}, \dots, s_n)+1=H(s_1, \dots, s_n),
\]
which also  does not satisfy  \eqref{freehomology}.  By Lemma \ref{lemma:generating},  $X_1, \cdots, X_n$ form a unique minimal  generating set of $\cHFL(T(n, n))$ over $\cR_n$. Based on the values of the $H$-function, one can compute that
\[
\gr_w(X_k)=-k(k-1).
\]

We now consider the relations satisfied by $X_1,\dots, X_n$ over $\cR_n$. For convenience, if $I=(i_1,\dots, i_n)$ is a sequence of nonnegative numbers, write $\scU^I$ for $\scU_1^{i_1}\cdots \scU_n^{i_n}$. Define $\scV^J$ similarly.  Lemma~\ref{lemma:generating} and Remark~\ref{rem:more-concrete-rels} immediately imply that if $1\le p<q\le n$, then the relations involving $X_{p}$ and $X_q$ are spanned by $(\scU_i\scV_i+\scU_j\scV_j) X_p$ and $(\scU_i \scV_i+\scU_j \scV_j) X_q$ for $i,j\in \{1,\dots, n\}$, as well as  sums
\[
\scU^I X_p+ \scV^J X_q
\] 
ranging over sequences of nonnegative integers $I$ and $J$ such that
\begin{equation}
I+J=(q-p,\dots, q-p)\quad \text{and}\quad  |I|_{L^1}=\sum_{s=p}^{q-1}s=\frac{(p+q-1)(q-p)}{2} .\label{eq:relations-indices}
\end{equation}
 If $q=p+1$, these are exactly the relations in the statement.

 \smallskip
 We will show by induction on $q$ that the relations between $X_p$ and $X_{q}$ are in the span of the relations between consecutive $X_i$ and $X_{i+1}$, as well as the relations $(\scU_i \scV_i+\scU_j\scV_j) X_k=0$ (labeled \eqref{relation-2} above). The case that $q=p+1$ is vacuous, so we suppose that $q>p+1$.

We let $I=(I_1, \dots, I_n)$ and $J=(J_1,\dots, J_n)$ be tuples of non-negative integers such that $I+J=(q-p,\dots, q-p)$ and  $|I|_{L^1}=p+\cdots+q-1$. We claim first that there is a tuple $I'$ such that 
\begin{equation}
0\le I'\le I,\quad |I'|_{L^1}=q-1,\quad \text{and} \quad  |I'|_{L^\infty}\le 2.
\label{eq:I'-props}
\end{equation}
  (We remark that it may not be possible to find such an $I'$ which satisfies $|I'|_{L^\infty}\le 1$).
For $k\in \N$, we write 
\[
a_k:=\# \{i: I_i \ge k\}.
\]
We observe that
\begin{equation}
|I|_{L^1}=\sum_{k=1}^{q-p} a_k=\frac{(p+q-1)(q-p)}{2}\quad \text{and} \quad a_1\ge a_2\ge \cdots \ge a_{q-p}. \label{eq:relations-a_i}
\end{equation}
Our claim about the existence of an $I'$ as above is equivalent to the claim that $a_1+a_2\ge q-1$. If, to the contrary  $a_1+a_2<q-1$, then we observe that
$a_2<(q-1)/2$ and therefore
\[
\sum_{k=1}^{q-p} a_k=a_1+a_2+\sum_{n=3}^{q-p} a_k<q-1+(q-p-2)\frac{q-1}{2}=\frac{(q-p)(q-1)}{2}.
\]
However this contradicts Equation~\eqref{eq:relations-a_i}, so we conclude that an $I'$ satisfying Equation~\eqref{eq:I'-props} must exist.


By induction, if $J'$ denotes $(q-p-1,\dots, q-p-1)-(I-I')$, then the relation
\begin{equation}
\scU^{I-I'} X_p=\scV^{J'} X_{q-1} \label{eq:relation-induction}
\end{equation}
is in the span of the claimed relations between consecutive $X_i$ and $X_{i+1}$, as well as the relations~\eqref{relation-2}.
 We observe that our relations imply
\[
\scU^I X_p= \scU^{I'} \scU^{I-I'} X_p=\scU^{I'} \scV^{J'} X_{q-1}.
\]

If $|I'|_{L^\infty}=1$, then we have the  relation
\[
\scU^{I'} X_{q-1}=\scV^{(1,\dots, 1)-I'} X_{q},
\]
which implies with Equation~\eqref{eq:relation-induction} that $\scU^I X_p=\scV^J X_{q}$, which would complete the proof. 

If instead $|I'|_{L^\infty}=2$, then we observe that any coordinate $i$ such that $I'_i=2$ has the property that $(I-I')_i<q-p-1$ and hence $J'_i>0$. In particular, $\scU^{I'} \scV^{J'}$ has a factor of $\scU_i \scV_i$ for each $i$ such that $I'_i=2$. Using relation~\eqref{relation-2} we can trade each of these $\scU_i \scV_i$ for a $\scU_j\scV_j$ where $I_j'=0$. Proceeding in this manner, we may relate $\scU^{I'} \scV^{J'} X_{q-1}$ with some $\scU^{I''} \scV^{J''} X_{q-1}$ where $|I''|_{L^1}=q-1$ and $|I''|_{L^\infty}=1$. The relations between $X_{q-1}$ and $X_q$ now show that
\[
\scU^{I'} \scV^{J'} X_{q-1}=\scU^{I''} \scV^{J''} X_{q-1}=\scV^{J} X_q,
\]
completing the proof.
\end{proof}

In general, a free resolution of the homology $\cHFL(T(n, n))$ can be computed algorithmically, see \cite{PeevaBook}, or, for a concrete
value of $n$, using a computer algebra system such as Macaulay2 \cite{M2}.

\subsection{The free complex of the $T(3,3)$ torus link}
\label{sec:T33-free}

 We present a free resolution of the torus link $T(3, 3)$.  The homology of the torus link $T(3, 3)$ is generated by $X_1, X_2 , X_3$ with the following relations:
$$\scU_i X_1= \prod_{j\in \{1, 2, 3\}\setminus \{i\}}\scV_j X_{2}, \quad \scV_i X_3= \prod_{j\in \{1, 2, 3\}\setminus \{i\}} \scU_j X_2; $$
$$\scU_{i}\scV_i X_{k}=\scU_j\scV_j X_{k}.$$
Then a free resolution of the  homology is 
\[
0\to C_3\xrightarrow{\d_3} C_2\xrightarrow{\d_2} C_1\xrightarrow{\d_1} C_0\to 0,
\]
with the spaces $C_0,C_1,C_2,C_3$ and the maps $\d_1,\d_2,\d_3$ defined as follows.

 The space $C_0=\scR_3^3$ is generated by $X_1,X_2,X_3$.
 Take the space $C_1=\scR_3^8$ generated by $b_1,b_2,b_3$, $B_1,B_2,B_3$,
 and $Z_1,Z_2$. For symmetry, it is helpful to consider an extra variable $Z_3$ which satisfies $Z_3=Z_1+Z_2$;
 $Z_3$ is not a generator of $C_1$.
  The differential $\partial_1\colon C_1\to C_0$ is
 given by
 \begin{align*}
   \partial_1 b_i&=\scU_i X_1+ \prod_{j\in \{1, 2, 3\}\setminus \{i\}}  \scV_j X_2, &
   \partial_1 B_i&=\scV_{i}X_3+\prod_{j\in \{1, 2, 3\}\setminus\{i\}} \scU_j X_2,\\
   \partial_1 Z_1&=\scU_2\scV_2 X_2+\scU_3\scV_3 X_2, & \partial_1 Z_2&=\scU_1\scV_1 X_2+\scU_3\scV_3 X_2\\ 
   \partial_1 Z_3&=\partial_1 (Z_1+Z_2)=\scU_1 \scV_1 X_2+\scU_2\scV_2 X_2.
\end{align*}
The link Floer homology of $T(3,3)$ is $\coker\partial_1$. Indeed,
the relations $\scU_i\scV_iX_k=\scU_j\scV_jX_k$ for $k=1$ and $k=3$
follow from other relations. For example,
\[0=\scV_1(\scU_1X_1-\scV_2\scV_3X_2)-\scV_2(\scU_2X_1-\scV_1\scV_3X_2)=(\scU_1\scV_1-\scU_2\scV_2)X_1.\]
We define the module $C_2=\scR_3^6$
with generators $c_1,c_2,c_3,d_1,d_2,d_3$ and let $\partial_2$ be the differential
\begin{align*}
  \partial_2 c_{k}&=\scU_i b_j+\scU_j b_i+\scV_k Z_k\\
  \partial_2 d_{k}&=\scV_i B_j+\scV_j B_i+\scU_k Z_k.
\end{align*}
Here $\{i,j,k\}$ ranges through all permutations of the set $\{1,2,3\}$.
There is a relation between $c_1,\dots,d_3$. That
is, there is a module $C_3=\scR_3$ generated by $e$ with
$\partial_3\colon C_3\to C_2$ given by
\[\partial_3 e=\scU_1 c_1+\scU_2 c_2+\scU_3 c_3+\scV_1 d_1+\scV_2 d_2+\scV_3 d_3.\]

It can be checked either directly, or using a computer algebra system, that $\ker\partial_1=\im\partial_2$, $\ker\partial_2=\im\partial_3$
and $\ker\partial_3=0$. That is to say, the complex we constructed is a free resolution of $\coker\partial_1$.

\begin{rem} By examining the resolution of $\cHFL(T(3,3))$ shown in Figure~\ref{fig:T33}, we see that the relations described in Lemma~\ref{lemma:generating} have some redundancy. For example $(\scU_1 \scV_1+\scU_2 \scV_2) X_1=0$ is a consequence of the relations $\scU_1 X_1=\scV_2 \scV_3 X_2$ and $\scU_2 X_2=\scV_1 \scV_3 X_2$. 
\end{rem}

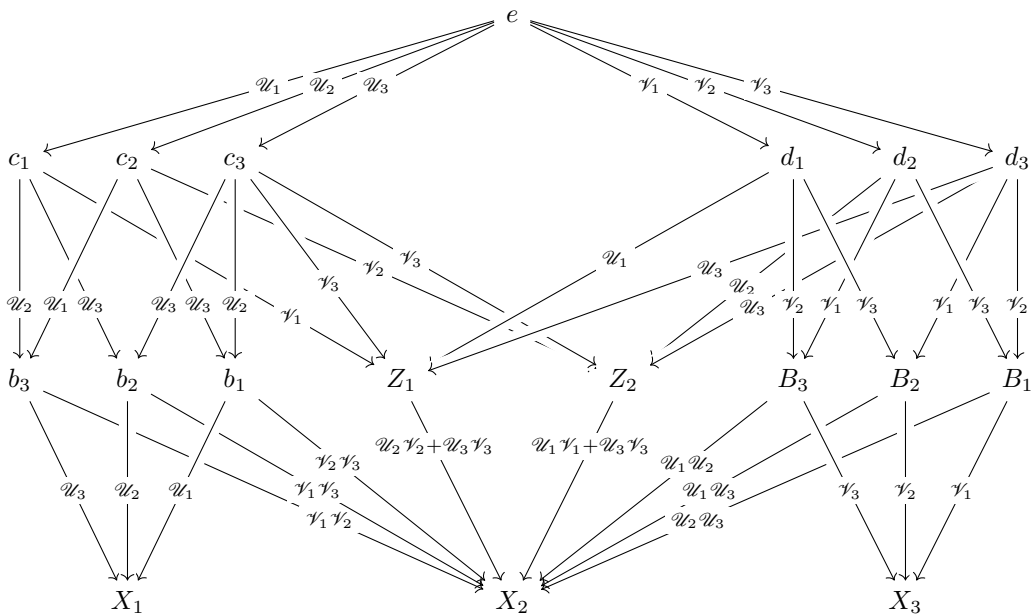
\begin{figure}[h]
\begin{tikzcd}[row sep=1.4cm,column sep=.75cm,labels=description] 
&&&[.8cm]&
e
	\ar[dllll,"\scU_1"]
	\ar[dlll,"\scU_2"]
	\ar[dll,"\scU_3"]
	\ar[drr,"\scV_1"]
	\ar[drrr,"\scV_2"]
	\ar[drrrr,"\scV_3"]
&&[.8cm]&&&\\
c_1
	\ar[d,"\scU_2", pos=.7]
	\ar[drrr,"\scV_1", pos=.75]
	\ar[dr,"\scU_3", pos=.7]
&
c_2
	\ar[dl,crossing over,"\scU_1", pos=.7]
	\ar[dr,crossing over,"\scU_3", pos=.7]
	\ar[drrrr,"\scV_2"]
&c_3
	\ar[d,crossing over,"\scU_2", pos=.7]
	\ar[dl,crossing over,"\scU_3", pos=.7]
	\ar[drrr,"\scV_3",crossing over,pos=.45]
	\ar[dr,"\scV_3",crossing over, pos=.58]
&&&&
d_1
	\ar[dlll,crossing over,"\scU_1", pos=.45]
&d_2
	\ar[dll,crossing over,"\scU_2", pos=.6]
&d_3
	\ar[d,"\scV_2", pos=.7,crossing over]
	\ar[dlll,"\scU_3",pos=.7, crossing over]
	\ar[dlllll,"\scU_3",pos=.5, crossing over]
	\\[1cm]
b_3
	\ar[drrrr,"\scV_1\scV_2", pos=.65]
	\ar[dr,"\scU_3",crossing over]
&
b_2
	\ar[drrr,"\scV_1\scV_3"]
	\ar[d,"\scU_2",crossing over]
&
b_1
	\ar[drr,"\scV_2\scV_3", pos=.35]
	\ar[dl,"\scU_1",crossing over]
&Z_1
	\ar[dr,"\scU_2\scV_2+\scU_3\scV_3" near start]
&
&Z_2
	\ar[dl,"\scU_1\scV_1+\scU_3\scV_3" near start]
&
B_3
	\ar[dll,"\scU_1\scU_2", pos=.35]
	\ar[from=u,crossing over, "\scV_2", pos=.7]
	\ar[from=ur,crossing over, "\scV_1", pos=.7]
&B_2
	\ar[dlll,"\scU_1\scU_3"]
	\ar[from=ul,crossing over, "\scV_3", pos=.7]
	\ar[from=ur,crossing over, "\scV_1", pos=.7]	
&B_1
	\ar[dllll,"\scU_2\scU_3", pos=.65]
	\ar[from=ul, "\scV_3",crossing over, pos=.7]
\\[1cm]
&X_1
	\ar[from=u, crossing over,"\scU_2"]
&&&X_2&&&
X_3
	\ar[from=ul,crossing over,"\scV_3"]
	\ar[from=u, crossing over,"\scV_2"]
	\ar[from=ur, crossing over,"\scV_1"]
\end{tikzcd}
\caption{The complex $\cCFL(T(3,3))$ as a free resolution.}
\label{fig:T33}
\end{figure}

\subsection{The free complex of the  $T(4,4)$ torus link}
\label{sec:T44-free}
To stress the usefulness and power of Theorem~\ref{thm:algebraic-free resolution}, we show how to compute
the link Floer chain complex of the $T(4,4)$ torus link. We start with the model of the link Floer homology for $T(4,4)$
as described in Theorem~\ref{thm:toruslink} and compute its free resolution. It is pretty straightforward
to find candidates for the relations between generators of the $T(4,4)$ torus link, and then candidates for the 
relations among relations (second syzygies), and to iterate
this procedure. This amounts to creating a complex $(\cC_i,\partial_i)$, whose homology at the zero grading is $\cHFL(T(4,4))$. Showing
that the relations are complete, that is, that the complex we construct is acyclic in higher gradings, is rather tedious. We have used
Macaulay \cite{M2} to verify this fact.

We can give a quick description of the free resolution of the $T(4,4)$ torus link. The free resolution has length four, as follows:
\[
0\to\cC_4\xrightarrow{\partial_4}\cC_3\xrightarrow{\partial_3}\cC_2\xrightarrow{\partial_2}\cC_1\xrightarrow{\partial_1}
\cC_0\to 0
\] 
Let $\cC_0$ be the free $\scR_4$ module
generated by $X_1,\dots,X_4$. Consider the module $\cC_1\cong\cR^{20}_4$
generated by $Z^k_{ij}$ with $k=2,3$, $j=i+1$, $a_{1},\dots,a_4$, $B_{ij}$ with $1\le i<j\le 4$ and $A_1,\dots,A_4$. 
As in the $T(3,3)$ case, for symmetry, we add variables $Z^k_{ij}$ for $j>i+1$, satisfying
$Z^k_{ij}=Z^k_{i,i+1}+Z^k_{i+1,i+2}+\dots+Z^k_{j-1,j}$. If $j<i$ we set $B_{ij}:=B_{ji}$.
There is a map
$\partial_1\colon \cC_1\to\cC_0$ given by
\begin{align*}
  \partial_1 Z^k_{i,i+1}&=(\scU_i\scV_i+\scU_{i+1}\scV_{i+1})X_k & \partial_1 a_i&=\scU_iX_1+\prod_{j\neq i}\scV_jX_2\\
  \partial_1 B_{ij}&=\scU_i\scU_jX_2+\prod_{\substack{i'<j'\\ \{i,j,i',j'\}=\{1,2,3,4\}}} \scV_{i'}\scV_{j'}X_3 &
  \partial_1 A_i & = \prod_{j\neq i}\scU_j X_3+\scV_i X_4.
\end{align*}
By Theorem~\ref{thm:toruslink}, the link Floer homology of $T(4,4)$ is $\coker \partial_1$.
We define now the module $\cC_2\cong\cR^{28}_4$. It is generated by $a_{ij}$, $A_{ij}$ with $1\leq i< j\leq 4$, and $c^{ij}_{k}, C^{k}_{ij}$.

The indexing of $c$ and $C$ generators is a bit complex.
We choose, $1\le k\le 4$, and $\{i,j\}$ is a subset of $\{1,2,3,4\}\setminus\{k\}$ with $j>i$. To obtain a smaller resolution,
we reduce the number of generators by declaring that the $c$-generators are $c^{23}_1$ and $c^{34}_1$, $c^{13}_2$ and $c^{34}_2$, $c^{12}_3$
and $c^{24}_3$, as well as $c^{12}_4$ and $c^{23}_4$. Whenever another configuration of indices appears in the differential, we declare it to be the sum of the other two generators with the same subscript, like $c^{24}_1$ is not a generator, but $c^{24}_1=c^{23}_1+c^{34}_1$ etc.

For $C$-generators, for each three-element subset $\{a,b,c\}\subset\{1,2,3,4\}$ we choose two out of three permutations $\{i,j,k\}$ with
$i<j$ and these indices are $C$-generators. The third object is declared to be the sum of the other two. More specifically,
we choose 
$C^1_{23}, C^2_{13}$ to be generators and $C^{3}_{12}=C^1_{23}+C^2_{13}$,
$C^1_{24}, C^4_{12}$ as generators, and  $C^2_{14}=C^1_{24}+C^4_{12}$,
$C^1_{34}, C^3_{14}$ as generators, with $C^4_{13}=C^1_{34}+C^3_{14}$. Finally,
$C^2_{34}, C^3_{24}$ are generators, whereas $C^4_{23}=C^2_{34}+C^3_{24}$.

Therefore, we have $8$ $c$-generators and $8$ $C$-generators.
The  map $\partial_2\colon\cC_2\to\cC_1$ is  given by 
\begin{align*}
  \partial_2 a_{ij}&=\scU_j a_i+\scU_i a_j+\prod_{k\neq i, j} \scV_k Z^{2}_{ij} \\
  \partial_2 A_{ij}&= \scV_j A_i+\scV_i A_j+\prod_{k\neq i, j} \scU_k Z^{3}_{ij}\\
  \partial_2 C^{k}_{ij}&=\scU_j B_{ik}+\scU_{i} B_{jk}+\scV_{\ell} Z^{3}_{ij}\\ 
  \partial_2 c^{ij}_{k}&=\scV_{i} B_{ik}+\scV_{j} B_{jk}+\scU_{k} Z^{2}_{ij}, 
 \end{align*}
 where $\ell=\{1,2,3,4\}\setminus\{i,j,k\}$.

 The module $\cC_3\cong\cR^{14}_4$ is generated by $a_{ijk}, A_{ijk}, B'_{ij}$ where $1\leq i<j<k\leq 4$ with the map $\partial_3\colon \cC_3\to\cC_2$ given by 
\begin{align*}
  \partial_3a_{ijk}&=\scU_{i} a_{jk}+\scU_{j} a_{ik}+\scU_{k} a_{ij}+\scV_\ell  (\scV_i c^{jk}_{i}+\scV_j  c_{j}^{ik}+\scV_k  c_{k}^{ij})\\
  \partial_3 A_{ijk}&=\scV_{i} A_{jk}+\scV_j A_{ik}+\scV_k A_{ij}+\scU_{\ell }(\scU_{i} C_{jk}^{\ell }+\scU_j C_{ik}^{\ell }+\scU_k C_{ij}^{\ell })\\
  \partial_3 B'_{ij}&=\scV_{k} C^{k}_{ij}+\scV_{\ell } C^{\ell }_{ij}+\scU_i c^{k\ell }_{j}+\scU_j c^{k\ell }_{i},
\end{align*}
where we let $\ell=\{1,2,3,4\}\setminus\{i,j,k\}$.

The module $\cC_4\cong\cR^{2}_4$ is generated by $d_{1234},D _{1234}$ with the   map $\partial_4$ is  given by 
\begin{align*}
  \partial_4(d_{1234})&=\sum_{i}\scU_{i} a_{jk\ell}+\sum_{\substack{i<j}}\scV_i\scV_j B'_{ij}\\
  \partial_4(D_{1234})&=\sum_i \scV_{i} A_{jk\ell}+\sum_{\substack{i<j,\ k<\ell\\\{i,j,k,\ell\}=\{1,2,3,4\}}} \scU_{i}\scU_{j} B'_{kl}.
\end{align*}
In each of the two equations, where the sum runs over the indices $i$, we let $1\le j<k<\ell\le 4$ be such that
$\{i,j,k,\ell\}$ is a permutation of $\{1,2,3,4\}$.

It can be verified that the above description of $\cCFL(T(4,4))$  is acyclic except at resolution grading $0$, with $\coker\partial_1$ being the link Floer homology of $T(4,4)$.

\bibliographystyle{custom}
\def\MR#1{}
\bibliography{biblio}

\end{document}